\RequirePackage{fix-cm} 
\documentclass[a4paper,smallextended,envcountsect,envcountsame]{svjour3}%
\makeatletter\if@twocolumn\PassOptionsToPackage{switch}{lineno}\else\fi\makeatother

\usepackage{amsfonts,amssymb,amsbsy,tabulary,amsmath}

\usepackage[T1]{fontenc}
\smartqed 
\usepackage{graphicx}
\usepackage{mathptmx}      

\usepackage{hyperref,endfloat} 
\usepackage{url,multirow,morefloats,floatflt,cancel,tfrupee}
\usepackage{colortbl}
\usepackage{xcolor}
\usepackage{pifont}
\usepackage[nointegrals]{wasysym}
\makeatletter

\def\mcWidth#1{\csname TY@F#1\endcsname+\tabcolsep}

\def\cAlignHack{\rightskip\@flushglue\leftskip\@flushglue\parindent\z@\parfillskip\z@skip}
\def\rAlignHack{\rightskip\z@skip\leftskip\@flushglue \parindent\z@\parfillskip\z@skip}

\@ifundefined{etal}{}{}

\usepackage{ifxetex}
\ifxetex\else\if@twocolumn\@ifpackageloaded{stfloats}{}{\usepackage{dblfloatfix}}\fi\fi

\AtBeginDocument{
\expandafter\ifx\csname eqalign\endcsname\relax
\def\eqalign#1{\null\vcenter{\def\\{\cr}\openup\jot\m@th
  \ialign{\strut$\displaystyle{##}$\hfil&$\displaystyle{{}##}$\hfil
      \crcr#1\crcr}}\,}
\fi
}

\AtBeginDocument{%
  \setlength{\oddsidemargin}{\dimexpr(\paperwidth-\textwidth)/2-1in}%
  \setlength{\evensidemargin}{\oddsidemargin}%
  \setlength{\topmargin}{%
    \dimexpr(\paperheight-\textheight)/2-\headheight-\headsep-1in}%
}

\AtBeginDocument{%
  \@ifpackageloaded{endfloat}%
   {\renewcommand\efloat@iwrite[1]{\immediate\expandafter\protected@write\csname efloat@post#1\endcsname{}}}{\newif\ifefloat@tables}%
}%

\def\BreakURLText#1{\@tfor\brk@tempa:=#1\do{\brk@tempa\hskip0pt}}
\let\lt=<
\let\gt=>
\def\processVert{\ifmmode|\else\textbar\fi}

\@ifundefined{subparagraph}{
\def\subparagraph{\@startsection{paragraph}{5}{2\parindent}{0ex plus 0.1ex minus 0.1ex}%
{0ex}{\normalfont\small\itshape}}%
}{}

\newcommand\role[1]{\unskip}
\newcommand\aucollab[1]{\unskip}
  
\@ifundefined{tsGraphicsScaleX}{\gdef\tsGraphicsScaleX{1}}{}
\@ifundefined{tsGraphicsScaleY}{\gdef\tsGraphicsScaleY{.9}}{}
\def\checkGraphicsWidth{\ifdim\Gin@nat@width>\linewidth
	\tsGraphicsScaleX\linewidth\else\Gin@nat@width\fi}

\def\checkGraphicsHeight{\ifdim\Gin@nat@height>.9\textheight
	\tsGraphicsScaleY\textheight\else\Gin@nat@height\fi}

\def\fixFloatSize#1{}
\let\ts@includegraphics\includegraphics

\def\inlinegraphic[#1]#2{{\edef\@tempa{#1}\edef\baseline@shift{\ifx\@tempa\@empty0\else#1\fi}\edef\tempZ{\the\numexpr(\numexpr(\baseline@shift*\f@size/100))}\protect\raisebox{\tempZ pt}{\ts@includegraphics{#2}}}}

\AtBeginDocument{\def\includegraphics{\@ifnextchar[{\ts@includegraphics}{\ts@includegraphics[width=\checkGraphicsWidth,height=\checkGraphicsHeight,keepaspectratio]}}}

\DeclareMathAlphabet{\mathpzc}{OT1}{pzc}{m}{it}

\def\URL#1#2{\@ifundefined{href}{#2}{\href{#1}{#2}}}

\def\UrlOrds{\do\*\do\-\do\~\do\'\do\"\do\-}%
\g@addto@macro{\UrlBreaks}{\UrlOrds}

\edef\fntEncoding{\f@encoding}

\makeatother

\newif\ifmultipleabstract\multipleabstractfalse%
\usepackage[square,numbers]{natbib}

\makeatletter	

\AtBeginDocument{%
  \@ifpackageloaded{endfloat}%
   {
\renewcommand*\efloat@process[2]{%
  \ef@ifct{#1}{%
    \expandafter\immediate\expandafter\closeout\csname efloat@post#1\endcsname
    \ef@setct{#1}{0}%
    \clearpage                                                         
        
    \efloat@ifflag{#2list}{
      {\normalsize\efloat@listof{#2}}
    }{}%

    \efloat@ifflag{#2head}{%
      \section*{\@nameuse{#2section}}
      \suppressfloats[t]
    }{}

    \markboth                                                          
      {\expandafter\uppercase\expandafter{\csname #2section\endcsname}}
      {\expandafter\uppercase\expandafter{\csname #2section\endcsname}}

    \def\efloat@type{#2}%
    \processdelayedfloat@hook
    \@nameuse{process#2s@hook}%
    \clearpage
    \@input{\jobname.#1}%
  }{}}
   
   }{}%
}%

  \makeatother

\usepackage{graphicx}
\graphicspath{ {images/} }

\usepackage{nicefrac}

\usepackage{caption}
\usepackage{subcaption}

\usepackage{enumerate}
\usepackage{xcolor}
\usepackage{mathtools}

\usepackage{mathrsfs}	
\usepackage{amssymb}
\usepackage{bbm}

\numberwithin{equation}{section}
\numberwithin{figure}{section}

\usepackage{scalerel,stackengine}
\stackMath
\newcommand\reallywidehat[1]{%
	\savestack{\tmpbox}{\stretchto{%
			\scaleto{%
				\scalerel*[\widthof{\ensuremath{#1}}]{\kern-.6pt\bigwedge\kern-.6pt}%
				{\rule[-\textheight/2]{1ex}{\textheight}}
			}{\textheight}%
		}{0.5ex}}%
	\stackon[1pt]{#1}{\tmpbox}%
}
 
\newcommand{\bE}{\mathbb{E}}

\newcommand{\bN}{\mathbb{N}}
\newcommand{\bP}{\mathbb{P}}
\newcommand{\bQ}{\mathbb{Q}}
\newcommand{\bR}{\mathbb{R}}
\newcommand{\bS}{\mathbb{S}}

\newcommand{\bZ}{\mathbb{Z}}

\newcommand{\cC}{\mathcal{C}}

\newcommand{\cX}{\mathcal{X}}
\newcommand{\cY}{\mathcal{Y}}

\newcommand\Item[1][]{%
	\ifx\relax#1\relax  \item \else \item[#1] \fi
	\abovedisplayskip=0pt\abovedisplayshortskip=0pt~\vspace*{-\baselineskip}}
\newcommand{\Oh}{
	\mathchoice
	{{\mathcal{O}}}
	{{\mathcal{O}}}
	{{\mathcal{O}}}
	{\scalebox{1.2}{$\mathcal{O}$}}
}

\newcommand{\oh}{
	\mathchoice
	{{\scriptstyle\mathcal{O}}}
	{{\scriptstyle\mathcal{O}}}
	{{\scriptscriptstyle\mathcal{O}}}
	{\scalebox{.7}{$\scriptscriptstyle\mathcal{O}$}}
}

\newcommand{\E}{\mathbb{E}}
\renewcommand{\P}{\mathbb{P}}

\newcommand{\1}{\mathbbm{1}}   

\begin{document}
	
\title{Rearranged Stochastic Heat Equation}\author{François Delarue and William R.P. Hammersley \thanks{W. Hammersley is supported by French ANR project 
		ANR-19-P3IA-0002 -- 3IA C\^ote d'Azur -- Nice -- Interdisciplinary Institute for Artificial Intelligence. F. Delarue is supported by the European Research Council (ERC) under the European Union’s Horizon 2020 research and innovation programme (ELISA project, Grant agreement No. 101054746).}}{\raggedright }
\institute{Universit\'e C\^ote d'Azur, CNRS, 
	Laboratoire J.A. Dieudonn\'e}\titlerunning{{\mbox{Rearranged Stochastic Heat Equation}}}

\authorrunning{François Delarue and William Hammersley}
        \date{\today}
      \maketitle 
\begin{abstract} The purpose of this work is to provide an explicit construction of a strong Feller semigroup on the space of probability measures over the real line that additionally maps bounded measurable functions into Lipschitz continuous functions, with a Lipschitz constant that blows up in an integrable manner in small time. Our construction relies on a rearranged version of the stochastic heat equation on the circle driven by a coloured noise. Formally, this stochastic equation writes as a reflected equation in infinite dimension.
Under the action of the rearrangement, the solution is forced to live in a space of quantile functions that is isometric to the space of probability measures on the real line. We prove the equation to be solvable by means of an Euler scheme in which we alternate flat dynamics in the space of random variables on the circle with a rearrangement operation that projects back the random variables onto the subset of quantile functions. A first challenge is to prove that this scheme is tight. A second one is to provide a consistent theory for the limiting reflected equation and in particular to interpret in a relevant manner the reflection term. The last step in our work is to establish the aforementioned Lipschitz property of the semigroup by adapting earlier ideas from the Bismut-Elworthy-Li formula. 
	\vspace{4pt}

	\vspace{2pt} 


\end{abstract}
 
\section{Introduction} 


{\bf \textit{Mean-field models with common noise}}. Our work is motivated by recent developments in the theory of mean-field models, at the intersection of stochastic analysis, calculus of variations and control and game theories. Although mean field models have a long history, stemming from statistical mechanics (see the pioneering work \cite{Kac56}), the problems studied in recent years are, in comparison, of an increasing complexity. For example, the solutions of control or game problems give rise, in the mean-field regime, to partial differential equations posed on the space of probability measures, whose understanding {remains an active area of research} in the case of control and {with even more open questions} in the case of games (see 
\cite{CardaliaguetSouganidis,CDLL,CecchinDelarue,GangboMayorgaSwiech} and the references therein for a recent state of the art on these questions).

Stochastic mean-field models lead, as soon as they evolve with time, to the study of dynamics with values in the space of probability measures. Although the latter are
understood as evolutions of the law of a typical particle, representative of the mean-field continuum, these probability measures remain most often deterministic. For example, they may be governed by non-linear Fokker-Planck equations or, depending on the terminology, may obey nonlinear Markovian dynamics, see for instance the seminal work by McKean \cite{McKean66}
and the monograph
\cite{Kolokoltsov_book}. Nevertheless, many recent works have underlined the interest in considering random dynamics on the space of probability measures. From a modelling point of view, the nonlinear Fokker-Planck equations become stochastic when the particles composing the mean-field continuum are subject to common noise, see for instance
the earlier works 
\cite{dawsonVaillancourt1995,KurtzXiong,KurtzXiong2,vaillancourt1988}
and also the more recent monographs \cite{CDLL,CarmonaDelarueII} within the framework of control and games. The presence of a common noise also raises interesting mathematical challenges; although it is possible in some cases to adapt the usual techniques of mean-field models, the understanding of the impact of common noise is in fact rather limited. In particular, there is currently no catalogue listing the varying effects of common noise on the statistical behaviour of solutions, unlike the theory of finite-dimensional diffusion processes, in which the impact of noise has been widely studied.
 
 \vskip 4pt
 
\noindent {\bf \textit{Models with a smoothing effect}}.
Typically - and this is the framework of this paper - it may be relevant to ask about the possible regularisation properties of the semigroup induced by a stochastic Fokker-Planck equation or by a mean-field model with a common noise. Although the expected properties are certainly limited when the common noise is of finite dimension (since the ambient space is of infinite dimension), the situation is different when the common noise is allowed to be infinite-dimensional. In other words, it is reasonable to imagine that a sufficiently ``large'' common noise could indeed provide regularisation phenomena. There is an example in the literature.
The Fleming-Viot process with 
mutations induced by diffusions 
is a probability measure valued process whose semigroup is strong Feller and maps bounded functions into Lipschitz continuous functions, see \cite{Stannat}. The generator, which acts on functionals of probability measures, contains two parts: a first-order term that coincides with the 
operator coming from a deterministic linear Fokker-Planck equation 
and a second-order term (that should be regarded as being induced by 
a form of common noise) yielded by the ``sampling replacement'' rule characterising the Moran and Fleming-Viot models. However, it must be stressed that 
the small-time smoothing property is rather poor, as the Lipschitz constant of the functions returned by the semigroup
may blow up exponentially fast in small time. This may seem anecdotal, yet such a limitation renders this noise almost impossible to use to establish regularisation by noise results.
 
\subsection{Diffusions with values in the space of probability measures} 
\noindent {\bf \textit{Wasserstein diffusions}}.
Searching for common noise(s) able to force some practicable smoothing properties on the space of probability measures is connected to a \emph{distinct} question addressed by a series of authors for almost fifteen years: what should be a Brownian motion on the space of probability measures? Whilst there has not yet been an answer to this question that may be called canonical, the existing candidates are usually referred to as ``Wasserstein diffusions'' 
(we emphasise that we do not propose a candidate Wasserstein diffusion in the sense described below). 
This terminology echoes the notion of Wasserstein space, defined as the space ${\mathcal P}_2({\mathbb R}^d)$ of probability measures  (on ${\mathbb R}^d$, for 
some $d \geq 1$) with finite second moment, equipped with the $2^{nd}$ Wasserstein distance ${\mathcal W}_2$. 
Many works from calculus of variations demonstrate the interest to endow the Wasserstein space with a kind of 
Riemannian structure, see \cite{JKO,Otto1,Otto2} and the book \cite{AGS}. In this approach, 
the tangent space at $\mu \in {\mathcal P}_2({\mathbb R}^d)$ is 
the closure in $L^2({\mathbb R}^d,{\mathbb R}^d;\mu)$ ($\mu$-square integrable functions from ${\mathbb R}^d$ into itself) of smooth compactly supported gradient vector fields on 
${\mathbb R}^d$. 
Accordingly, the \textit{Wasserstein derivative} or intrinsic gradient of a functional defined on ${\mathcal P}_2({\mathbb R}^d)$
reads, at any $\mu \in {\mathcal P}_2({\mathbb R}^d)$,
as the gradient of
a real-valued function (i.e., a potential) defined on ${\mathbb R}^d$. Roughly speaking, this 
potential corresponds to the so-called flat/functional derivative 
used to formulate the generator of the aforementioned Fleming-Viot process, see \cite{DawsonMarch,Dawson,Stannat}. 
Wasserstein diffusions are usually expected to be valued in ${\mathcal P}_2({\mathbb R})$ and consistent with 
${\mathcal W}_2$; i.e. the small time large deviations having rate functional ${\mathcal W}_2^2$ and 
the local variance (or quadratic variation) in the 
corresponding chain rule (or It\^o formula) 
is expected to derive from the
Riemannian metric. 
Whilst the Fleming-Viot process
is not a ${\mathcal W}_2$-Wasserstein diffusion, examples are known. The most famous is the $1d$ Wasserstein diffusion constructed by von Renesse and Sturm, in
\cite{vRenesseSturm}, wherein they introduce a parametrised class of \textit{entropy} probability measures 
on {${\mathcal P}([0,1])$} { - the space of probability measures on $[0,1]$ - }
and then to consider, under each of these probability measures, the Markov process 
associated with the Dirichlet form generated by the Riemannian metric. 
The entropy probability measures are constructed by transferring Poisson-Dirichlet measures on the space of 
quantile functions on 
{$[0,1]$} onto {${\mathcal P}([0,1])$}, by means of the 
 isometry that exists between the two spaces when the former is equipped with the $L^2$-norm and the latter with ${\mathcal W}_2$.
The same isometry
plays a key role in our work, however we use slightly different quantile functions.

Although the work \cite{vRenesseSturm} has had a great impact in the field, it is fair to say that this Wasserstein diffusion remains a difficult approach. In particular, definition \textit{via} a Dirichlet
form does not permit generic starting points and, to the best of our knowledge, there has not been any 
systematic analysis of the semigroup's properties. We refer to \cite{AndresvRenesse,Sturm} for particle approximations of this Wasserstein 
diffusion and \cite{DoringvRenesse} for a log-Sobolev inequality. 
Several works have been written in the wake of \cite{vRenesseSturm}. For example, in \cite{Konarovskyi,Konarovsky0}, Konarovskyi proposed an alternative construction in one dimension, leading to another definition of the Wasserstein diffusion. From the particle system perspective, this approach aims at evolving a cloud of massive random particles, with the heavier particles having smaller fluctuations. The particles aggregate, becoming heavier as they collide. 
 As opposed to the Dirichlet form construction, the model allows one to consider arbitrary initial conditions, but the collision rules force the dynamics to instantaneously take its values in the set of finitely supported probability measures.
 The analysis has been pushed further in 
\cite{KonarovskyivRenesse}, but 
many questions remain open, starting with uniqueness when the cloud of particles is initialised from a continuum.    
We refer to \cite{Konarovskyi2} and the references therein for an extension allowing for fragmentation and to \cite{Marx1} for a mollification
of the coalescing dynamics, for which uniqueness holds true.
Last but not least, the $1d$ dynamics constructed in 
\cite{Konarovskyi,Konarovsky0}
are somehow extended to the higher dimensional setting in 
\cite{DelloSchiavo0}
but using the theory of Dirichlet forms in the spirit of \cite{vRenesseSturm}. 
\vskip 4pt

\noindent{\bf \textit{Connection with the Dean-Kawasaki equation}}. The aforementioned works are connected with stochastic Fokker-Planck equations. 
In \cite{Konarovskyi,Konarovsky0,vRenesseSturm}, these Wasserstein diffusions each induce generators (acting on functionals of probability measures) sharing similarities with the generator of the so-called Dean-Kawasaki equation. 
Formally, the latter is a stochastic version of a standard Fokker-Planck equation (of order 1 or 2 depending on the cases) including an additional noisy term whose local quadratic variation derives exactly from the Riemannian metric 
on ${\mathcal P}({\mathbb R}^d)$ (with $d=1$ in 
\cite{Konarovskyi,Konarovsky0,vRenesseSturm}). 
However, it has been proved in 
\cite{KonarovskyivRenesse3,KonarovskyivRenesseTobias}
that the Dean-Kawasaki equation, 
in its strict version, cannot be solvable except in trivial cases where it reduces to a finite dimensional particle system (which requires the initial distribution to be finitely supported).
This negative result has an interesting consequence: some extra correction is needed in the dynamics, which is exactly 
what is done in 
 \cite{Konarovskyi,Konarovsky0,vRenesseSturm}. However, so far there has not been any canonical choice for 
such a correction. 

The very spice of the Dean-Kawasaki equation may be explained as follows. 
When the solution is at some 
probability measure $\mu \in {\mathcal P}({\mathbb R})$, 
a typical particle 
in the {mean-field} continuum, located at some point $x \in {\mathbb R}$, should be subjected to 
the value at this point $x$ of a cylindrical Wiener noise on $L^2({\mathbb R},{\mathbb R};\mu)$, which makes no sense in general.  
This suggests that Dean-Kawasaki dynamics can be approached by replacing the cylindrical noise by a coloured noise. 
To a certain extent, this idea is the basis of the two contributions 
\cite{Ding} and \cite{Marx2}. 
 
In \cite{Marx2}, the resulting semigroup is shown to  have an (albeit weak) mollification effect on functions over ${\mathcal P}({\mathbb R})$.

\subsection{Our contribution}
\label{subse:1.3}

\noindent{\bf \textit{Smoothing properties of the  Ornstein-Uhlenbeck process}}.
Unlike many of the aforementioned works, our {aim} is not to provide another candidate Wasserstein diffusion. 
Our {primary} motivation in this contribution 
is to construct as explicitly as possible a probability-measure valued process having sufficiently strong smoothing properties. Although not discussed further within this text, our long-term goal is to propose a corresponding theory of linear or nonlinear parabolic Partial Differential Equations (PDEs) on the space of probability measures and to exhibit, in this context, second order operators allowing to smooth singularities that may appear in the corresponding hyperbolic PDEs. 
Of course, such a process should share similarities with Wasserstein diffusions, but as we will see, the diffusion introduced in this paper does not satisfy the pre-requisites for being a Wasserstein diffusion. 

Our approach is based on two observations. First, Lions \cite{LionsVideo} showed in his lectures  on mean-field games at the Collège de France, that in the study of mean-field models, it can prove useful to \emph{lift} probability measures into random variables, i.e., to invert the map sending a random variable to its statistical distribution.
Although the inverse is multi-valued, it has been shown that Lions' lifting principle provides a clear picture of the Wasserstein derivative: in short, it can be represented as a Fr\'echet derivative on a Hilbert space of square-integrable random variables, see e.g. \cite{gangbo}. Our second remark is a well-known fact from stochastic analysis: we know how to construct a Hilbert-valued diffusion process with strong smoothing properties. A simple example is the Ornstein-Uhlenbeck process driven by an appropriate operator, see for instance \cite{cerrai,daprato,daPratoZabczyk2014stochEqnsInfDim}.  
This suggests the following procedure: we should project onto the space of probability measures, an Ornstein-Uhlenbeck process taking values in a space of square-integrable random variables. 
Whilst this looks very appealing, this idea has an obvious drawback. In general, the projection should destroy the Markov
nature of the dynamics; transition probabilities
started from two different random variables representing the same probability measure may not be the same.  

Our construction is thus inspired from the Lie-Trotter-Kato formula and 
related splitting methods. We alternate between, one 
step in the space of random variables following some prescribed Ornstein-Uhlenbeck dynamics, and a projection operation to return back from the space of random variables to the space of probability measures. Choosing the probability space carrying the random variables is simple: we work on the circle, {${\mathbb S} \cong \mathbb{R}/\mathbb{Z}$}, equipped with Lebesgue measure. The choice of projection is much more difficult. It is an essential aspect in implementing the splitting scheme
and, as in the contributions \cite{Ding,Konarovskyi,Konarovsky0,Marx1,Marx2,vRenesseSturm}, it leads us to limit our study to the one-dimensional case, with the following two advantages. First, probability measures can be easily identified with quantile functions on the circle (or `symmetric non-increasing functions', see Proposition \ref{prop:rearrangement:def}),  
which makes the choice of projection easier as it suffices to send a function on the circle to an appropriate rearrangement. Second, the rearrangement operation is an easy way to transform a random variable on the circle into 
a quantile function whilst preserving its statistical law (under the Lebesgue measure). 

The resulting scheme in which we combine `flat' dynamics and rearrangement is very much inspired by earlier works of Brenier on discretisation schemes for conservation laws, 
see for instance \cite{Brenier2,Brenier1}, with the main difference being that the works of Brenier are mostly for deterministic dynamics. Since we choose the Laplacian to be the driving operator in the Ornstein-Uhlenbeck dynamics, 
we call the resulting equation the `rearranged stochastic heat equation'. 
\vskip 4pt

\noindent{\bf \textit{Rearranged and reflected equations}}.
The presence of the noise raises many subtleties in our construction. One particular issue is that 
the rearrangement operation and the Laplacian driving the Stochastic Heat Equation (SHE) do not marry well. Obviously, they do not commute. As a result, the smoothing effect of the Laplacian (acting on 
functions on $\bS$) is weaker when the rearrangement is present. At least, this is what we observe in our computations. 
This has a rather dramatic consequence on the {choice of the}
noise. One key feature of the SHE 
is that after convolution with the heat kernel, the cylindrical white noise driving the SHE gives a true random function. When the SHE is rearranged 
(as we do here), this no longer seems to be the case. 
In order to remedy 
this problem, we need to colour the noise driving the SHE. 
As expected, this impacts the generated semigroup's smoothing properties. Nevertheless we succeed to show that the rate at which the derivative of the semigroup 
{blows up} in small time is integrable, as we initially intended.
{It remains} an open question 
 whether the same construction can be achieved for the SHE driven by a cylindrical white noise. 

Another difficulty is to obtain a suitable formulation of the rearranged SHE. Although Brenier's works 
 \cite{Brenier2,Brenier1}
 quite clearly suggest to see the rearrangement as a reflection and indeed to write the rearranged SHE as a reflected equation, again, the presence of the noise requires additional precautions. The study of
reflected differential equations 
is in general more complicated in the stochastic case than in the deterministic case because the solutions are no longer of bounded variation. 
We refer to the seminal article \cite{Lions:Sznitman} in the case of finite dimensional equations. To the best of our knowledge, there is no general theory covering our infinite dimensional formulation of the  rearranged SHE. We therefore propose a tailor-made interpretation in which the reflection term is constructed by hand. Schematically, the rearranged SHE is written as a stochastic partial differential equation (SPDE) on the space $L^2({\mathbb S}):=L^2({\mathbb S},\textrm{\rm Leb}_{{\mathbb S}})$ (of functions on the circle that are square-integrable with respect to the Lebesgue measure) subject to a reflection term forcing the solution to remain in the cone of our chosen quantile functions (symmetric non-increasing). 
This representation is reminiscent of the $1d$ reflected stochastic differential equation studied by Nualart and Pardoux \cite{NualartPardoux} (and extended in \cite{Donati-Martin}), in which the SHE is constrained to be positive. Although the latter positivity constraint may be interpreted as a constraint on the monotonicity of the primitive, the rearranged SHE that we study here is not the primitive of the Nualart-Pardoux reflected equation.

The form of the reflection in \cite{NualartPardoux} was further specified in the later contributions \cite{Zambotti1,Zambotti2} due to Zambotti. These results provide a more refined description of the solution's behaviour at the domain's boundary. In our approach we are not able at this stage, to give a similar picture. Our construction of the reflection process and its associated integral is too elementary. In particular, {we consider only the action of the reflection process on functions that are far more regular than the solution of the equation itself}. 
 Fortunately, this does not prevent us from obtaining a characterisation of the solutions, sufficient to carry out our program to the end. In fact, Zambotti's results are based on a formula of integration by parts that allows one to reinterpret the solutions 
of the Nualart-Pardoux equation by means of the theory of Dirichlet forms. The adaptation to our case remains completely open. 
We refer however to 
the papers \cite{Barbu3,Barbu2,Barbu1,rocknerZhuZhu2012reflectInfDimConv}
for more general works that have been published subsequently on reflected stochastic differential equations in infinite dimension.
\vskip 4pt

\noindent{\bf \textit{Description of the results}}. 
The rearranged SHE is proven well-posed in the strong sense. The main solvability result is Theorem 
\ref{thm:main:existence:uniqueness} and the reader may find the notion of solution in Definition \ref{def:existence}. The proof holds in two main steps. The first is to show existence of weak solutions and the second one is to prove that uniqueness holds in the strong sense. 
Strong existence then follows from a standard adaptation of Yamada-Watanabe's theorem. As is often the case, the first step is more challenging. Weak solutions are obtained as weak limits of linear interpolations of an Euler scheme: each iteration is a small time step of Ornstein-Uhlenbeck dynamics in $L^2({\mathbb S})$ followed by rearrangement {of} the terminal random variable. Part of the challenge is to show that the scheme is tight (in the space of continuous functions). 
This is done in Section \ref{sec:2} by using several key properties of the rearrangement 
operation, as presented in Section \ref {sec:rearrangement}. To complete the proof of the existence of a weak solution, we need to 
give an appropriate sense to the 
reflection process, which is one of the goals of Section \ref{sec:3}. The main point in proving strong uniqueness is to impose, in the definition of a solution, a weak form of orthogonality between the solution and the reflection.  
The second main statement of the article is Theorem \ref{thm:main:lipschitz}, which says that the semi-group generated by our rearranged SHE
is strongly Feller, i.e., maps bounded measurable functions
 into continuous 
 functions. Moreover, the semi-group returns Lipschitz continuous functions, 
 with Lipschitz constant diverging integrably in small time. {The proof of Theorem \ref{thm:main:lipschitz}}
draws heavily on previous works on the so-called Bismut-Elworthy{-Li} formula, an integration by parts formula for the 
transition probabilities of a 
 diffusion process, see
  for instance
  \cite{Elworthy92,ElworthyLi,Thalmaier} 
in the finite-dimensional framework and
\cite{DPEZ}
and
\cite[Chapter 7]{cerrai}
in infinite dimension. 
Such an integration by parts is strongly related to 
Malliavin calculus, see 
for instance
Exercise 2.3.5 in
the book 
\cite{Nualartbook}, 
together with the papers
 \cite{Bismut}
and
 \cite{norris1986simplifiedMallCalc}. 
 Transposition 
 of the Bismut-Elworthy{-Li} formula to the reflected setting is known however to raise some technical difficulties.
 A major obstacle, is to prove differentiability of the flow with respect to the initial condition. 
 We refer to 
\cite{deuschelZambotti2005bismutElworthySDErefl}
for the first result in this direction (drifted {Brownian} motion with reflection in the orthant)
and to 
\cite{Andres1,Andres2,LipschutzRamanan1,LipschutzRamanan2} for further results. 
None of these results (which are all in finite dimension) apply to our case. 
At this stage, we do not know if similar results hold for the rearranged SHE.  
Instead, in our analysis, we use the sole property that the flow (generated by the rearranged SHE) 
 is Lipschitz continuous with respect to 
the initial condition and thus
almost everywhere differentiable 
when the initial condition is restricted to a finite-dimensional space.
\vskip 4pt

\noindent{\bf \textit{Comparison with recent literature and further prospects}}.
A few weeks before we put this work on arXiv, another arXiv pre-publication was published (\cite{RenWang}) in which the authors introduce, on the space of probability measures, a Dirichlet form whose construction has some similarities with the construction of the rearranged SHE that we introduce here. Note that the results of the two papers do not overlap, but an in-depth study would be necessary to link the two constructions more properly. In short, the work \cite{RenWang} aims at projecting on the space of probability measures a Gaussian measure constructed on an $L^2$ space of random variables
and then at considering, under this measure, the Dirichlet form generated by the Riemannian metric on ${\mathcal P}_2({\mathbb R}^d)$ (with 
$d \geq 1$). For example, in $1d$, this Gaussian measure can be the invariant measure of the SHE driven by a cylindrical white noise. Although 
this example (in $1d$) does not fit our assumptions (since we need the noise to be coloured), it is worth noting that, if we had to write formally the generator of the rearranged SHE in this case, it would be different from the one computed in \cite[Theorem 4.1]{RenWang}.  

We also highlight that our construction has a simple particle interpretation. At each time step of the Euler scheme, we can indeed consider a particle approximation of the SHE, as given for example by a finite volume discretisation. Then, 
at the end of each time step, the rearrangement operation, when implemented on the particles, simply consists in ordering them. We do not discuss this further in the rest of the article (for obvious reasons of length).

The reader may wonder about higher dimensional extensions. Although this is indeed a natural equation, we think it is useful to recall that many of the aforementioned works (notably those concerning the construction of 
a Wasserstein diffusion) are also in one dimension. From this point of view, this limitation in our model should not come as a surprise. 
As for the possible ways to extend the construction to the case  $d \geq 2$, one possibility is to use the tools of optimal transport (\cite{BrenierPolar}), but this perspective is open at this stage.
The reader may also worry about the fact that, in  dimension $d \geq 2$, the stochastic heat equation (when driven by the Laplace operator) requires a coloured noise, 
of a higher regularity than what we use here. 
In fact, this would be 
only the case if we considered the stochastic heat equation on a space of dimension $d$ (typically the $d$-dimensional torus). Actually, our belief is that 
we could define the stochastic heat equation on the $1d$ torus, but regard it 
as a system of $d$ equations.
That said, another possibility could be to replace the Laplacian by another operator.  

\vskip 4pt

\noindent{\bf \textit{Organisation of the paper}}.
 We introduce some preliminary material in 
 Section  
\ref{sec:rearrangement}, including some (known) results on 
the symmetric rearrangement 
on ${\mathbb S}$.
Section 
\ref{sec:2}
is dedicated to the 
analysis of the approximating scheme. In particular, the reader will find all the required assumptions on the noise in 
the introduction of 
Section \ref{sec:2}. 
Tightness is established in 
Proposition 
\ref{prop:tightness}. 
The definition of a solution to the rearranged SHE is 
clarified in 
Section 
\ref{sec:3}, see Definition \ref{def:existence}. 
Existence and uniqueness are guaranteed by 
Theorem 
\ref{thm:main:existence:uniqueness} . 
The smoothing properties of the semigroup is studied in 
Section 
\ref{sec:4}, the main Lipschitz estimate being stated in 
Theorem 
\ref{thm:main:lipschitz}.

\section{Preliminary Material}
\label{sec:rearrangement}

\subsection{The symmetric non-increasing rearrangement}

{Throughout}, the circle {${\mathbb S}$} is chosen to be parametrised by the interval $(-1/2,1/2]$ and $0$ is regarded as 
a privileged fixed point on the circle, {i.e., ${\mathbb S}:=({\mathbb R}+1/2)/{\mathbb Z}$}.
 
\begin{proposition}
\label{prop:rearrangement:def}
Given a measurable function $f : \bS \rightarrow {\mathbb R}$, there exists a unique function, called symmetric non-increasing rearrangement of $f$ and denoted $f^* : \bS \rightarrow [-\infty,+\infty]$, that satisfies the following two properties:
\begin{enumerate}
\item $f^*$ is symmetric (with respect to $0$), is non-increasing and right-continuous on the interval $[0,1/2)$,
and is left-continuous at $1/2$ (left- and right-continuity being here understood for the topology on $[-\infty,+\infty]$),
\item \textbf{Cavalieri's principle:} the image of the Lebesgue measure $\textrm{\rm Leb}_{{\mathbb S}}$
by $f^*$ is the same as the image of the Lebesgue measure by $f$, namely, for all $ a \in {\mathbb R}$, $
\textrm{\rm Leb}_{{\mathbb S}}(\{ x \in {\mathbb S} : f^*(x) \leq a\}) 
= 
 \textrm{\rm Leb}_{{\mathbb S}}(\{ x \in {\mathbb S} : f(x) \leq a\})$.
\end{enumerate}
\end{proposition}
Intuitively, $f^*$ should be regarded as a quantile function, the symmetrisation procedure here forcing an obvious form of `continuous periodicity' 
(whose interpretation requires some care as $f^*$ may have jumps).
Indeed, it must be noted that 
the collection of functions $f^*$ satisfying item 1 in the definition above are one-to-one with the set ${\mathcal P}_2({\mathbb R})$ of probability measures on ${\mathbb R}$
that have a finite-second moment. In fact, for $f^*$ as in item 1 and for a probability measure $\mu \in {\mathcal P}_2({\mathbb R})$, 
the measure $\textrm{\rm Leb}_{\mathbb S} \circ (f^*)^{-1}$ is equal to $\mu$ if and only if 
$x \in [0,1] \mapsto f^*((1-x)/2)$ coincides with the usual quantile function, i.e. the usual generalised inverse of the (right-continuous) cumulative distribution function. 
The reader is referred to Baernstein \cite{baernstein2019symmetrizationInAnalysis} for further details,
see in particular Definition 1.29 therein for the general definition of symmetric rearrangements in the Euclidean setting and Chapter 7 in the same book for a specific treatment of spherical symmetric rearrangements. We use the following quite often:
\begin{definition}
\label{def:1.2:U2}
A function $f  : {\mathbb S} \rightarrow {\mathbb R}$ 
is said to be symmetric non-increasing if $f=f^*$.  
 The collection of equivalence classes in $L^2({\mathbb S})$ containing a symmetric non-increasing function is denoted  by $U^2({\mathbb S})$.
 
It is a cone. 
\end{definition}
 
Below, we often consider elements of $L^2_{\rm sym}({\mathbb S})$. They are defined as functions in $L^2(\bS)$ that are Lebesgue almost everywhere symmetric. 
One of these elements is said to be \emph{non-increasing} (we refrain from tautological use of the word symmetric given the context of the circle) if it coincides almost everywhere with an element of 
$U^2({\mathbb S})$. Notice that we may choose the latter representative to be uniquely defined as a symmetric non-increasing function. Indeed, two elements of $U^2({\mathbb S})$ that coincide in $L^2({\mathbb S})$ 
coincide in fact everywhere on ${\mathbb S}$ (courtesy of the left- and right-continuity properties). 
Also, the following proposition is of clear importance. 
\begin{proposition}
\label{prop:closedness} $L^2_{\rm sym}({\mathbb S})$ and $U^2({\mathbb S})$ are closed subsets of $L^2({\mathbb S})$ equipped with $\| \cdot \|_2$. 
\end{proposition}

\begin{proof} Closedness of 
$L^2_{\rm sym}({\mathbb S})$ is obvious. Closedness of 
$U^2({\mathbb S})$ follows from Lemma 
\ref{lem:nonexpansion} below:
if  
$(f_n)_{n \geq 1}$  in 
$U^2({\mathbb S})$ converges to some $f \in L^2_{\rm sym}({\mathbb S})$, then 
$f=f^*$.
\qed
\end{proof}

\subsection{Reformulating the main results}

{Our} diffusion process with suitable smoothing properties
on 
${\mathcal P}_2({\mathbb R})$
{arrives} via the construction of a diffusion process with values in $U^2({\mathbb S})$. 
 The equivalence 
relies on the fact that the mapping $f^* \in U^2({\mathbb S}) \mapsto 
\text{Leb}_{\mathbb S} \circ (f^*)^{-1} \in {\mathcal P}_2({\mathbb R})$ is {an isometry} when ${\mathcal  P}_2({\mathbb R})$ is equipped with the ${\mathcal W}_2$-{Wasserstein} distance, i.e., 
for any $f^*$, $g^*$ in $U^2({\mathbb S})$, 
\begin{equation*}
\| f^* - g^* \|_2 = {\mathcal W}_2
\bigl( \text{Leb}_{\mathbb S} \circ (f^*)^{-1},
\text{Leb}_{\mathbb S} \circ (g^*)^{-1}
\bigr), 
\end{equation*}
\begin{equation*}
	{\text{where}} \quad 
{\mathcal W}_2(\mu,\nu)^2 := \inf_{\pi \in {\mathcal P}({\mathbb R^2})) : \pi \circ e_x^{-1}=\mu, \pi \circ e_y^{-1}=\nu}
\int_{{\mathbb R}^2} \vert x-y \vert^2 \pi(dx,dy), 
\end{equation*} with $e_x:(x,y)\in {\mathbb R}^2 \mapsto x$ and
$e_y:(x,y)\in {\mathbb R}^2 \mapsto y$
being the two evaluation mappings on 
${\mathbb R}^2$. 
This identity is a consequence of Lemma \ref{lem:nonexpansion}, since for any $\pi$ as above, there exist two (measurable) functions 
$f$ and $g$ from ${\mathbb S}$ to ${\mathbb R}$ such that 
$\pi=\text{Leb}_{\mathbb S} \circ (f,g)^{-1}$.

In this framework, our main results can be (re)formulated as follows:
\begin{enumerate}
\item We introduce a 
stochastic differential equation on $U^2({\mathbb S})$ 
in the form of a reflected (or rearranged) stochastic equation
on $L^2({\mathbb S})$  
whose reflection term forces 
solutions to stay within the cone $U^2({\mathbb S})$, whenever
they are initialised from 
$U^2({\mathbb S})$, see
Theorem \ref{thm:main:existence:uniqueness}. 
Solutions induce
a Lipschitz continuous flow with values in 
$U^2({\mathbb S})$.
The construction of the rearranged equation relies on 
an Euler scheme, in which we alternate 
some flat dynamics in the space 
$L^2_{\rm sym}({\mathbb S})$ with 
the rearrangement operation that projects back the solution onto 
$U^2({\mathbb S})$. 

\item The second main statement is Theorem  
\ref{thm:main:lipschitz}, which says that the semigroup generated by our rearranged stochastic equation 
maps bounded measurable functions
on 
$U^2({\mathbb S})$
 into Lipschitz continuous 
 functions on
 $U^2({\mathbb S})$. 
 Recast on ${\mathcal P}_2({\mathbb R})$ (through the isometry between 
$U^2({\mathbb S})$ and ${\mathcal P}_2({\mathbb R})$), we get in this way a 
semigroup that maps bounded measurable functions 
on ${\mathcal P}_2({\mathbb R})$ into Lipschitz continuous functions (with respect 
to the $2$-Wasserstein distance ${\mathcal W}_2$).
\end{enumerate}

\subsection{Key properties of the symmetric non-increasing rearrangement} 

In the subsection, we expand a list of  {useful} properties that are satisfied by $f^*$. The first one just follows from item 2 in the statement of 
Proposition  
\ref{prop:rearrangement:def}.

\begin{lemma}[Preservation of $L^p$ norms]
\label{lem:isometry}
With the same 
notations as in Proposition 
\ref{prop:rearrangement:def}, we have, for any $p \in [1,\infty]$, 
$ \lVert f^* \rVert_p = \lVert f  \rVert_p$.
\end{lemma}

The next result, called the  Hardy-Littlewood inequality, is fundamental. 

\begin{lemma}[Hardy-Littlewood inequality]
\label{lem:hardy}
Let $f$ and $g$ be two measurable real-valued functions defined on 
${\mathbb S}$ such that $\|f \|_p < \infty$ and $\|g \|_q < \infty$, for $p,q \in [1,\infty]$, 
with $1/p+1/q=1$. Then, 
$$ \int_{\bS} f(x)g(x) dx \leq \int_{\bS} f^*(x)g^*(x) dx.$$
\end{lemma}

 We refer to \cite[Corollary 2.16]{baernstein2019symmetrizationInAnalysis} 
 for a general statement in the Euclidean setting, but stated under the conditions that 
 $f$ and $g$ are non-negative, 
 and to 
 \cite[Section 7.3]{baernstein2019symmetrizationInAnalysis}
 or
 \cite{Baernstein_correction,baernstein1989convRearrOnCirc}
for a version without non-negativity constraints that is specifically stated on the circle. We now turn to the well-known property of non-expansion:
 \begin{lemma}[Non-expansion property]
 \label{lem:nonexpansion}
 Let $f$ and $g$ be two measurable real-valued functions 
 with $\|f \|_p < \infty$ and $\|g \|_p < \infty$, for $p \in [1,\infty]$.
 Then, $ \lVert f^*-g^* \rVert_p \leq \lVert f-g \rVert_p.$ 
 \end{lemma}
 We refer to \cite[Corollary 2.23]{baernstein2019symmetrizationInAnalysis}
 for the Euclidean setting (which requires $f$ and $g$ to be positive valued) and 
 to  \cite[Section 7.3]{baernstein2019symmetrizationInAnalysis}
 for the extension to the spherical setting (which no longer requires $f$ and $g$ to be positive valued).

The
following statement is taken from 
\cite{Baernstein_correction,baernstein1989convRearrOnCirc}, see also
\cite[Theorem 8.1]{baernstein2019symmetrizationInAnalysis}.

\begin{lemma}[Riesz rearrangement inequality]
\label{lem:Riesz}
Let $f$, $g$ and $h$ be three measurable real-valued functions on 
${\mathbb S}$, such that $\| f \|_p < \infty$, 
$\| g \|_q < \infty$ and 
$\| h \|_r < \infty$ for $p,q,r \in [1,\infty]$ with $1/p+1/q+1/r=1$. Then,
\begin{equation*}
\int_{\bS} \int_{\bS} f(x) g(x-y) h(y) dx dy \leq \int_{\bS} \int_{\bS} f^*(x) g^*(x-y) h^*(y) dx dy.
\end{equation*}
\end{lemma}

\subsection{The heat kernel and the rearrangement operator}
We now address several basic properties of the composition of the rearrangement operator and the heat kernel. First, we recall that the periodic heat semigroup (with specific diffusivity parameter $1$) on the circle $\mathbb{S}$, which we denote $(e^{t\Delta})_{t \geq 0}$, has the following kernel (see Dym and McKean p.63 \cite{dymMcKean1972bookFourierSeriesAndInts}):
\begin{equation}
\label{eq:Gamma}
\Gamma_t(x):= \frac{1}{\sqrt{4\pi t}}\sum_{n\in\bZ} \exp\left\{
-\frac{(x-n)^2}{4t}\right\}, \quad t >0, \ x \in {\mathbb S}.
\end{equation}

\begin{lemma}
\label{lem:heatkernelisdecreasing}
For any $t>0$, the function $x \mapsto \Gamma_t(x)$ is non-increasing on $(0,1/2)$ and non-decreasing on 
$(-1/2,0)$. 
That is, 
$\Gamma_t(\cdot)=\Gamma_t(\cdot)^*$, the rearrangement acting on $x$.
\end{lemma} 
The proof of 
Lemma 
\ref{lem:heatkernelisdecreasing} is not 
trivial, due to the series underpinning the expression of {$\Gamma$}. 
The reader will find a general discussion on spherical heat kernels 
in the recent paper \cite{Nowak}, but specific (and much easier) computations that suffice for the proof of  the above statement can be found in 
\cite{Andersson}. We conclude this subsection with:  
\begin{lemma}
\label{lem:heatpreservessymmetry}
For  $f$ in $U^2({\mathbb S})$ and  $t>0$, the convolution $f*\Gamma_t(\cdot)$ is also in $U^2({\mathbb S})$.
\end{lemma}
\begin{proof}
Let $h= f*\Gamma_t(\cdot)$, for a given $t>0$. 
Lemmas
\ref{lem:Riesz}
and
\ref{lem:heatkernelisdecreasing}
yield 
\begin{equation*}
\| h \|_2^2 = 
\int_{\bS} \int_{\bS} f(x) {\Gamma_t(x-y)} h(y) dx dy
\leq 
\int_{\bS} \int_{\bS} f(x) {\Gamma_t(x-y)} h^*(y) dx dy
=
\langle h,h^* \rangle_2.
\end{equation*}
By the preservation of $L^p$ norms, 
$\| h - h^*\|_2^2 = \| h \|_2^2 + \| h^*\|_2^2 - 2 \langle h,h^* \rangle_2
= 2 \| h \|_2^2 - 2 \langle h,h^* \rangle_2 \leq 0$, 
which yields $h=h^*$ almost everywhere. Since $h$ is continuous (by convolution), so is $h^*$ (see 
\cite[Subsection 2.4]{baernstein2019symmetrizationInAnalysis}). Therefore, $h$ and $h^*$ coincide. 
\qed
\end{proof}
By combining Proposition 
\ref{prop:closedness} 
with Lemma 
\ref{lem:heatpreservessymmetry}, we obtain 
the following stronger closedness property: 
\begin{proposition}
\label{prop:closedness:H-1} 
{For any $a>0$, 
$\{ f \in L^2_{\rm sym}({\mathbb S}) :  \| f \|_2 \leq a \}$ and $\{ f \in U^2({\mathbb S}) : \| f \|_2 \leq a\}$} are closed subsets of $H_{\rm sym}^{-1}({\mathbb S})$ equipped with $\| \cdot \|_{2,-1}$. 
\end{proposition}

\begin{proof}
Take a {bounded} sequence $(f^n)_{n \geq 1}$ in 
$L^2_{\rm sym}({\mathbb S})$ 
that converges 
(for $\| \cdot \|_{2,-1}$)
to some $f \in H^{-1}_{\rm sym}({\mathbb S})$. 
By lower semi-continuity of the $L^2$-norm with respect to the 
$H^{-1}$-norm, we deduce that $f$ belongs to $L^2({\mathbb S})$ {with $\| f \|_2 \leq \liminf_{n \rightarrow \infty} \| f^n\|_2$}.

Assume now 
that the sequence 
$(f^n)_{n \geq 1}$ 
takes values 
in 
$U^2({\mathbb S})$. 
Then, for each $\varepsilon >0$, $f*\Gamma_\varepsilon= e^{\varepsilon \Delta} f $ 
is in 
$U^2({\mathbb S})$. 
This follows from the following two points. 
Firstly, for each $n \geq 1$, 
$e^{\varepsilon \Delta} f_n$ is in $U^2({\mathbb S})$ (as a consequence of 
Lemma \ref{lem:heatpreservessymmetry}). 
Secondly, 
$\| e^{\varepsilon \Delta} f_n - e^{\varepsilon \Delta}f \|_2$ tends to $0$ 
as $n$ tends to $\infty$. By closedness of $U^2({\mathbb S})$ with respect to 
the $L^2$ norm, we get that 
$e^{\varepsilon \Delta} f \in U^2({\mathbb S})$. 

Finally, 
since $f$ is in $L^2({\mathbb S})$,  
$\| e^{\varepsilon \Delta} f - f \|_2$ tends to $0$ with $\varepsilon$,
and we can invoke again the fact $U^2({\mathbb S})$ 
is closed with respect to the $L^2$ norm. 
\end{proof}
\color{black} 

\subsection{Some notation}
 
{We introduce a few notations related with functional and Fourier analysis.}
The space of continuous functions from one metric space $\cX$ to another, $\cY$, 
is denoted $\cC(\cX,\cY)$. For $k \geq 1$, we denote by $\cC^\infty_0( {\mathbb R}^k)$
the space of
infinitely differentiable real-valued functions on ${\mathbb R}^k$ with compact support.

We recall that $\bS$ is the circle parametrised by the interval of length $1$. Also, we let
\begin{equation*}
\begin{split}
e_m^{\Re} : x \in \bS \mapsto \sqrt{2} \cos(2 m \pi x), \quad 
e_m^{\Im} : x \in \bS \mapsto \sqrt{2} \sin(2 m \pi x),  
\end{split}
\end{equation*}
{for any  natural number $m$,
together 
with $e_0^{\Re} :\equiv 1$ and  
$e_0^{\Im} :\equiv 0$,}
form the complete Fourier basis on $L^2(\bS)$, where $L^2( \bS) $ is the space of square integrable functions on $\bS$. Usually, we just use the even (cosine) Fourier functions, which prompts us to use the shorter notation 
$e_m$ for $e_m^{\Re}$. 

The Lebesgue measure on ${\mathbb S}$ is denoted $\textrm{\rm Leb}_{{\mathbb S}}$ with $d\textrm{\rm Leb}_{{\mathbb S}}(x)$ written as $dx$. For any $p \geq 1$, we call $\| \cdot \|_p$
the $L^p$ norm on 
the space of measurable functions $f$ on $({\mathbb S},\textrm{\rm Leb}_{{\mathbb S}})$ with $\int_{\bS} |f(x)|^p dx<\infty$. 
Similarly, when 
$p=\infty$, the notation 
$\| \cdot \|_\infty$ is used for the $L^\infty$ (supremum) norm, i.e. $\| f \|_\infty:=\textrm{essup}\{f(x):x\in\bS\}$. 
The inner product between two elements $f$ and $g$ in $L^2({\mathbb S})$ is denoted $\langle f,g \rangle_2$, or $\langle f, g \rangle$, or  also $f \cdot g$ depending on the context. 
For an element $f \in L^1(\bS)$ and a non-negative integer $m$, we call $\hat{f}_m^{\Re} := \int_{\bS} f(x) e_m^{\Re}(x) dx$ the 
cosine
Fourier mode of $f$ of index $m$ and $\hat{f}_m^{\Im} := \int_{\bS} f(x) e_m^{\Im}(x) dx$
the sine Fourier mode of $f$ of index $m$. 
When $f$ is Lebesgue almost everywhere (written a.e. hereafter) symmetric, i.e. 
$f(-x)=f(x)$ {a.e.}, all the sine Fourier modes are $0$ and we write 
$\hat{f}_m=\langle f, e_m\rangle$ instead of $\hat{f}_m^{\Re}$. 
In that case, $\hat{f}^m$ is a real number. 
We denote by $L^2_{\rm sym}({\mathbb S})$ the set of functions $f$ in 
$L^2({\mathbb S})$ that are {a.e.} symmetric. More generally, for a parameter $\mu \in {\mathbb R}$, we denote by
$H^{\mu}_{\rm sym}({\mathbb S})$ the Sobolev space of  symmetric {functions/distributions}
$f$ such that
$\| f \|_{2,\mu}^2 := \sum_{m \in {\mathbb N}_0} (m \vee 1)^{2 \mu}  \hat{f}_m^2 < \infty$,
{with the notation $\hat{f}_m$ extending in an obvious manner to the distributional case}
(${\mathbb N}$ is the collection of natural numbers, and ${\mathbb N}_0 := {\mathbb N} \cup \{0\}$). Of course, $H^0_{\rm sym}({\mathbb S})$ is just $L^2_{\rm sym}(\bS)$. 
The inner product on 
$H^{\mu}_{\rm sym}({\mathbb S})$
is denoted $\langle f,g \rangle_{2,\mu}:=\sum_{m\in\bN_0} (m \vee 1)^{2 \mu} \hat{f}_m\hat{g}_m$. 

For any integer $k \geq 1$, we denote by 
${\mathcal C}^k({\mathbb S})$ the space of $k$-times continuously differentiable functions on 
${\mathbb S}$. For a real number $x$, we write $\lfloor x \rfloor$ for the floor of $x$,
 $\lceil x \rceil$ for the ceiling of $x$ and $x_+:=\max(x,0)$ 
 (resp. $x_- = \min(-x,0)$)
 for the positive (resp. negative) part of $x$.  
 For two reals $x$ and $y$, we let $x \vee y := \max(x,y)$ and $x \wedge y := \min(x,y)$.  
Moreover, for a differentiable real-valued function on $\bS$, we write $D f$ for the derivative of $f$. And,
we let $\Delta := D^2$.

As for constants that are used in the various inequalities, they are
usually written in the form $c_{a,b}$ or $C_{a,b}$, where the subscripts are quantities on which the current constant depends,
and are implicitly allowed to vary from line to line.

\section{Approximation Scheme and its Estimates}
\label{sec:2}

 
{
Our construction relies on a discretisation scheme in which we alternate one random move in the Hilbert space $L_{\rm sym}^2({\mathbb S})$ and rearrangement, forcing the output of the scheme to 
stay within the subset of symmetric non-{increasing} functions $U^2({\mathbb S})$. 
\vskip 5pt
\noindent \textit{Definition of the noise.}
The randomisation in $L^2_{\rm sym}({\mathbb S})$ obeys an Euler scheme with Gaussian increments. We introduce the following Wiener process, $( W_t )_{t \geq 0}$:
}
\begin{equation}
\label{def:tilde:noise}
	\begin{split}
		{W}_t :=B^0_t e_0 +\sum_{m\in \bN}m^{-\lambda} B^m_t e_m  \equiv \sum_{m\in \bN_0}\lambda_m B^m_t e_m, \quad t \geq 0, \\ 
	\end{split}
\end{equation}
where
$\lambda > \nicefrac{1}2$ and 
the sequence $(\lambda_m)_{m \in {\mathbb N}_0}$ 
is given by
$\lambda_0:=1$ and 
 $\lambda_m:=m^{-\lambda}$ for $m \in {\mathbb N}$. 
 Here, $\{(B_t^m)_{t \geq 0}\}_{m \in {\mathbb N}_0}$ are independent standard Brownian motions constructed on a 
 { filtered probability space $(\Omega,{\mathcal A},{\mathbb F}, {\mathbb P})$ (satisfying the usual conditions)}.  
  
{Our choice $\lambda > \nicefrac12$  precludes the  white noise} and forces the sequence 
 $(\lambda_m)_{m \in {\mathbb N}_0}$ to be square
summable.
In particular, the process $(  W_t )_{t \geq 0}$ can be equivalently defined 
as an $L^2_{\rm sym}(\bS)$-valued Brownian motion with covariance function  
\begin{equation}
\label{eq:Q:covariance}
Q : (f,g) \in \bigl( L^2_{\rm sym}(\bS) \bigr)^2 \mapsto s \wedge t \sum_{m \in {\mathbb N}_0} \lambda_m^2 \hat f^m \hat g^m = s \wedge t \, \langle f,g \rangle_{2,-\lambda}.
\end{equation}
\vskip 5pt

\noindent \textit{Definition of the scheme.}
The approximation scheme is constructed via composition of the stochastic convolution associated with $W$ and the rearrangement operator $*$ defined in Proposition \ref{prop:rearrangement:def}. 
Given a stepsize {$h \in (0,1)$}, we define $( X^h_n)_{n\in\bN_0}$ by
\begin{equation}
	\label{eq scheme n}
	\begin{split}
		X^{h}_{n+1}= & \left(  e^{h\Delta}X^h_{n}+\int_0^{h} e^{(h-s)\Delta}dW^{n+1}_s \right)^*,\ X^h_0: = 
		 X_0, \   	W^{n+1}_s :=  W_{s+nh}-W_{nh},
	\end{split}
\end{equation} 
where $X_0$ is a $U^2(\bS)$-valued random variable assumed to be independent of 
$( W_t )_{t \geq 0}$ (see 
\textbf{Assumption on $X_0$} for a clear formulation). {Note that measurability
of $X^h_{n+1}$, seen as a random variable with values in $L^2_{\rm sym}({\mathbb S})$ 
(equipped with its Borel $\sigma$-field)
is guaranteed by the continuity of the rearrangement operation (see Lemma \ref{lem:nonexpansion}).}

In Subsection
\ref{subse:3.1}, 
the dependence of $W^{n+1}$ on $h$ is suppressed in the notation, since $h$ is kept fixed. It is only in the forthcoming Subsection
\ref{sec tightness}
 that $h$ becomes variable as we let  {the latter}  tend to $0$.  
 
\vskip 5pt

\noindent 
{
\textit{Reminders about the stochastic convolution.} 
} For an ${\mathcal F}_0$-measurable initial condition $X_0$ with values  in $L^2_{\rm sym}(\bS)$, 
the stochastic convolution provides a weak solution to the {SHE (see Da Prato and Zabczyk \cite[Ch.5]{daPratoZabczyk2014stochEqnsInfDim} for a comprehensive introduction)}
\begin{equation}\notag 
	\begin{split}
		dX_t = & \Delta X_t dt + dW_t,\quad t \geq 0, 
	\end{split}
\end{equation}
written 
{on $(\Omega,{{\mathcal A}},{\mathbb F},{\mathbb P})$.} That is to say, that for all {$t\geq 0$} and $\varphi \in {\mathcal C}^2({\mathbb S})$, the process $\hat{X}:=( \hat{X}_t)_{t \geq 0}$ defined by
\begin{equation}\notag
	\begin{split}
	&\hat{X}_t:=e^{t\Delta}X_0+\int_0^te^{(t-s)\Delta}dW_s, \quad t \geq 0,
	\\
		{\text{satisfies} \quad {\mathbb P}\text{-a.s.,}}\quad 
	\langle &\hat{X}_t,\varphi \rangle =  \langle X_0 ,\varphi \rangle + \int_0^t \langle \hat{X}_s,\Delta \varphi \rangle ds + \langle W_t, \varphi \rangle,\quad t \geq 0. 
	\end{split}
\end{equation} 
By \cite[Theorem 2, p.146]{kotelenez1982submartIneqSPDE}, the process $\hat{X}$ has  a version with continuous sample paths. From \cite[Corollary 1, p.345]{zangeneh1990measurabilityofSolutionSemiLinSPDE}, this version is adapted.  
 Additionally, from   \cite[Theorem 6, p.4]{salavatiZangeneh2016pthPowerMaxIneqStochConv}, the following pathwise estimate holds for $p\geq 2$,
\begin{align}
\label{eq pth norm ineq}
		&\lVert \hat{X}_t \rVert_2^p \leq   \lVert \hat{X}_0 \rVert_2^p + p\int_0^t\lVert \hat{X}_s \rVert_2^{p-2} \langle \hat{X}_s,d W_s \rangle  + \frac{p(p-1)}{2}\int_0^t \lVert \hat{X}_s \rVert_2^{p-2} d [W]_s ,
\\
\label{eq:bracket:tildew}
 {\text{with} }\ 
&[ W]_t
= \sum_{m \in {\mathbb N}_0}
\lambda_m^2 t, \quad t \geq 0.
\end{align} 
{Due to}
\cite[Theorem 2, p.147]{Zangeneh_Stochastics}, see also
 \cite[Theorem 5, p.4]{salavatiZangeneh2016pthPowerMaxIneqStochConv}, 
 for $p\geq2$, $T>0$, 
\begin{equation}\label{eq max ineq}
	\begin{split}
		\mathbb{E}\left[ \sup_{0\leq t\leq T}\left\lVert \int_0^t e^{(t-s)\Delta}dW_s \right\rVert_2^p\right] \leq  c_p \mathbb{E}\left[ [W]_T^{p/2} \right].
	\end{split}
\end{equation}
{
Subsection 
\ref{subse:3.1}
is dedicated to {proving estimates} on the scheme that are uniform in the stepsize 
$h$. 
Tightness is
addressed in
Subsection \ref{sec tightness}.
}
\vskip 5pt

\noindent \textit{Distributional derivative of the noise.} 
{For any $t \geq 0$}, we let 
\begin{equation}
\label{eq:tildew:distributional:derivative}
  w_t := D   W_t = - 2 \pi 
\sum_{m\in \bN} m^{1-\lambda} B^m_t e_m^{\Im},
\end{equation} 
{which is a Brownian motion with values in} 
$H^{-1}_{\textsf{\rm anti-sym}}(\bS)$, 
the latter being defined as the dual of the space 
$H^{1}_{\textsf{\rm anti-sym}}(\bS)$
of anti-symmetric periodic functions with a square-integrable generalised gradient.  
\vskip 5pt

\noindent 
\textbf{Assumption on $X_0$.} 
Throughout the rest of the paper, we assume that 
$X_0$ 
is an ${\mathcal F}_0$-measurable random variable with values  in $U^2(\bS)$, satisfying
\begin{equation}
\label{eq:assumption:X0}
\forall p \geq 1, \quad {\mathbb E} 
\Bigl[ \bigl\| X_0 \bigr\|_2^{2p} \Bigr] < \infty.
\end{equation}

\subsection{$L^p$ estimates of the solution}
\label{subse:3.1}
We start with some preliminary estimates for the $L^p$ norm of the
process 
$(X_n^h )_{n \in {\mathbb N}_0}$.  
\begin{lemma}
	\label{lem unif est noise}
For $p>0$ (and for $h$ being the stepsize of the scheme and $\lambda$ the exponent colouring the noise),
	\begin{align}
	\label{eq:lem unif est noise:1}
			\E \biggl[ \biggl\Vert \int_0^{h} e^{(h-s)\Delta}dW_s \biggr\rVert_2^{2p} \biggr] \leq  & c_{p,\lambda} h^p; 
		\  
			{\text{when }}p=1,\ 	c_{1,\lambda} = 
			\sum_{m\in\bN_0}\lambda^2_m = \frac{d}{dt} 
			[   W]_t.
		\end{align}	
\end{lemma}

\begin{proof} 
The proof is standard and follows from combining {Theorem} 4.36 in \cite{daPratoZabczyk2014stochEqnsInfDim}, p114 (refer to p.96 therein for related notation), with Fourier analysis and
{\eqref{def:tilde:noise}}.		
		\qed
\end{proof}

As a consequence, we have:

\begin{lemma}
	\label{lem unif est0}
For $T>0$ and $p\geq 2$ (and for $h$ being the stepsize of the scheme and $\lambda$ the exponent colouring the noise),
	\begin{equation}\label{eq unif est0.1}
		\begin{split}
		\sup_{n \in {\mathbb N}_0 : nh  \leq T}
			\E \left[\bigl\lVert X^h_{n} \bigr\rVert_2^{p} \right] \leq 
			c_{p,\lambda,T} \left( 1+  \E \left[\left\lVert X_{0} \right\rVert_2^{p  }\right] \right). 
		\end{split}
	\end{equation}
\end{lemma}

\begin{proof}
The first step follows from the fact that the rearrangement preserves $L^p$ norms.  
\begin{equation}\label{eq unif est0}
	\begin{split}
			\E & \left[\bigl\lVert X^h_{n} \bigr\rVert_2^{p} \right]  = \E \left[\Big\lVert   e^{h\Delta}  X^h_{n-1} + \int_0^{h} e^{(h-s)\Delta}dW^n_s \Big\rVert_2^{p} \right].	\\
	\end{split}
\end{equation} 	
The \textcolor{black}{mild} solution to the stochastic heat equation started from $X_{n-1}^h$
and driven by $( W^n_r)_{0 \leq r \leq h}$ (see 
	\eqref{eq scheme n})
is denoted here by
\begin{equation}\label{eq unif est0.5}
	\begin{split}
		\hat{X}^{h,{n-1}}_s &:=  e^{s\Delta}X_{n-1}^h + \int_0^se^{(s-r)\Delta}dW^{n}_r, \quad s \in [0,h]. 
	\end{split}
\end{equation}
Then, by estimate \eqref{eq pth norm ineq},
\begin{align}\label{eq unif est0.2}
		&\E    \left[\lVert    X^h_{n} \rVert_2^{p} \right]  
		\\
		&\leq  \E \left[  \bigl\lVert X^h_{n-1} \bigr\rVert_2^p + p\int_0^{h}\lVert \hat{X}^{h,n}_s \rVert^{p-2}_2 \langle \hat{X}^{h,n}_s,dW^{n}_s\rangle  + \tfrac{p(p-1)}{2}  \int_0^{h}  \lVert \hat{X}^{h,n}_s \rVert^{p-2}_2d[W^n]_s  \,   \right]. \nonumber  
\end{align} 		
{One may remove the martingale terms} (by induction 
over the index 
$n$ in 
\eqref{eq unif est0}, the left-hand side therein is obviously finite, and then the 
left-hand side in 
\eqref{eq unif est0.5}
has a finite $p$-moment for any $p \geq 1$). It remains to control $\mathbb{E}[  \lVert \hat{X}^{h,n}_s \rVert^{q}_2]$ for $q\geq 0$.
\begin{align} \label{eq unif est0.2.1}
		 \lVert \hat{X}^{h,n}_s  \rVert^{q}_2 =    \Bigl\lVert e^{s\Delta}X_{n-1}^h + \int_0^se^{(s-r)\Delta}dW^{n}_r \Bigr\rVert^{q}_2 
		&\leq   c_q\left( \lVert X_{n-1}^h   \rVert^{q}_2 + \Bigl\lVert   \int_0^se^{(s-r)\Delta}dW^{n}_r \Bigr\rVert^{q}_2  \right),
	\end{align}
using the contraction property of the heat semigroup. 
In light of Lemma 
\ref{lem unif est noise},
\begin{equation}
\label{eq unif est0.2.2}
	\begin{split}
		\mathbb{E}\Bigl[   \lVert \hat{X}^{h,n}_s \rVert^{q}_2  \Bigr] & \leq c_q \left( \mathbb{E}\left[\lVert X_{n-1}^h   \rVert^{q}_2 \right]+ c_{q,\lambda} h^{q/2} \right) .
	\end{split}
\end{equation} 
Choosing $q=p-2$ and injecting the above bound in 
\eqref{eq unif est0.2}, we obtain 
\begin{equation}
\label{eq unif est0.2.3}
	\begin{split}	
		\E   \left[\lVert    X^h_{n} \rVert_2^{p} \right] 
		\leq &
		\E   \left[\lVert    X^h_{n-1} \rVert_2^{p} \right] 
	+ 	  c_{p,\lambda} h  \E \left[  \bigl\lVert X^h_{n-1} \bigr\rVert_2^{p-2} 
		+   h^{(p-2)/2}   \right].\\
		\end{split}
\end{equation}
{The assumption $h < 1$ together with the} bound $a^{p-2} \leq 1 + a^p$, for $a \geq 0$, gives  
\begin{equation*}
	\begin{split}	
		\E   \left[\lVert    X^h_{n} \rVert_2^{p} \right] 
		\leq &
	\Bigl( 1 + c_{p,\lambda} h \Bigr) 
		\E   \left[\lVert    X^h_{n-1} \rVert_2^{p} \right] 
	+ 	  c_{p,\lambda} h.
	\\
		\end{split}
\end{equation*} 
The conclusion follows from the discrete version of Gronwall's lemma. 
\qed
\end{proof} 

\subsection{Tightness}\label{sec tightness}

  
 { For
 $X_0$ taking values in $U^2({\mathbb S})$ and satisfying ${\mathbb E} 
[ \| X_0 \|_2^{2p} ] < \infty$ for any $p \geq 1$, we  address the tightness properties of the scheme, see Proposition 
\ref{prop:tightness} for the main statement.} 
{Whilst it would be possible to study tightness in $\cC([0,\infty),L^2_{\rm sym}({\mathbb S}))$, it is in fact much simpler to work  in 
$\cC([0,\infty),H^{-1}_{\rm sym}({\mathbb S}))$ (see for instance \cite{hambly-ledger} for another use of $H^{-1}$ in the analysis of McKean-Vlasov equation).}  
{To proceed,} we define 
the following linear interpolation 
 $( \tilde{X}^h_t )_{t \geq 0}$ 
of the scheme: 
 \begin{equation}
	\label{eq interpolation}
	\tilde{X}^{h}_{t} := (\lceil \nicefrac{t}{h} \rceil - \nicefrac{t}{h}  )X^h_{\lfloor \nicefrac{t}{h} \rfloor} + (\nicefrac{t}{h}-\lfloor \nicefrac{t}{h} \rfloor )X^h_{\lceil \nicefrac{t}{h} \rceil}.
\end{equation}
Lemma \ref{lem unif est0} {(applied with $T+1$ instead of $T$)} gives us the following bound: 
 
\begin{corollary}
\label{cor:2.5}
For an initial condition $X_0 \in U^2({\mathbb S})$ with finite moments of any order, for a real $T>0$ and for any real $p \geq 1$, we have
\begin{eqnarray}
		\sup_{t\leq T}\E \Bigl[ \bigl\lVert\tilde{X}^h_t \bigr\rVert^{2p}_{2} \Bigr] \leq C_{p,\lambda,T,\E [\lVert X_{0} \rVert_2^{2p }]}.
\end{eqnarray}
\end{corollary} 
Here is now the main result of this subsection: 
\begin{proposition}
\label{prop:tightness}
For any finite time horizon $T>1$, 
the linear interpolation schemes $\{ \tilde X^h\}_{{h \in (0,1)}}:=\{ (\tilde X^h_t)_{t \geq 0}\}_{h \in (0,1)}$ 
{have tight laws on} $\cC([0,T],{H^{-1}_{\rm sym}(\mathbb S)})$. 
Moreover, for any $p \geq 1$, there exists a constant $C_{p,\lambda,T,\E[\left\lVert X_0\right\rVert_2^{2p}]}$, independent of $h$, such that 
\begin{equation*}
{\mathbb E} \Bigl[ \sup_{n : nh \leq T+h}  \| X_n^h \|_2^{2p} \Bigr] \leq C_{p,\lambda,T,\E[\left\lVert X_0\right\rVert_2^{2p}]}.
\end{equation*}
\end{proposition}

\begin{proof}
{The proof is to verify 
Kolmogorov-Chentsov's criterion. Throughout the proof, we let $N_0:=\lceil T/h \rceil$.}
\vspace{5pt}
\\
\textit{First step.} 
Consider the quantity $\sup_{n \in \{0,\cdots,N_0\}}
\| X_{n}^{h} - X_0 \|_2$. By the triangle inequality,
$$\sup_{n \in \{0,\cdots,N_0\}}
\| X_{n}^{h} - X_0 \|_2 \leq \sup_{n \in \{0,\cdots,N_0\}}\left\{  \| X_{n}^{h} - e^{nh\Delta}X_0 \|_2 +\| e^{nh\Delta}X_0 - X_0 \|_2  \right\}.
$$
{To handle the first summand, one begins by use of Lemmas 
\ref{lem:nonexpansion} and \ref{lem:heatpreservessymmetry} and the contractive property of the heat semigroup:}
\begin{equation*}
	\begin{split}
		\| X_{n}^h - e^{nh \Delta} X_0 \|_2^2
		&\leq    \Bigl\| e^{h \Delta} X_{n-1 }^h - e^{nh \Delta} X_0 +  \int_{(n-1)h}^{nh} e^{ (nh-s) \Delta} dW_s \Bigr\|_2^2
		\\
		&\leq
		\Bigl\|  X_{n-1 }^h - e^{(n-1)h \Delta} X_0  \Bigr\|_2^2 + 
		\Bigl\| \int_{(n-1)h}^{nh} e^{ (nh-s) \Delta} dW_s \Bigr\|_2^2
		\\
		&\hspace{15pt} + 
		2 
		\Bigl\langle e^{h \Delta} X_{n-1 }^h - e^{nh \Delta} X_0, 
		\int_{(n-1)h}^{nh} e^{ (nh-s) \Delta} dW_s
		\Bigr\rangle.
	\end{split}
\end{equation*}
By iteration,  
\begin{align}
		\| X_{n }^h - e^{nh \Delta} X_0 \|_2^2
		&\leq \sum_{k=1}^{n} 
		\Bigl\| \int_{(k-1)h}^{kh} e^{ (kh-s) \Delta} d W_s \Bigr\|_2^2 \nonumber
		\\
		&\hspace{15pt}
		 + 
		2 \sum_{k=1}^n  
		\Bigl\langle e^{h \Delta} X_{k-1}^h - e^{kh \Delta} X_0, 
		\int_{(k-1)h}^{kh} e^{ (kh-s) \Delta} d W_s
		\Bigr\rangle \label{eq:IV:1} 
		\\
		&{=:  T^1_n + 2 T^2_n,} \nonumber
\end{align}
{with the convention $T^1_0=T^2_0=0$.}

One now studies the regularity of
the 
 two discrete processes $(T_n^1)_{n \geq 0}$ and $(T^2_n)_{n \geq 0}$ (indexing by $h$ is omitted). 
 
For $( T^1_n)_{n \geq 0}$, 
observe that, for $p \geq 1$
and  $0 \leq m < n$, {by Lemma \ref{lem unif est noise} and the generalised means inequality}, 
\begin{equation*}
	\begin{split}
		{\mathbb E}
		\Bigl[ \bigl\vert T^1_n - T^1_m \vert^{p} 
		\Bigr]  = 
		{\mathbb E}
		\biggl[ 
		\biggl(
		\sum_{k=m+1}^{n} 
		\Bigl\| \int_{(k-1)h}^{kh} e^{ (kh-s) \Delta} d W_s \Bigr\|_2^2
		\biggr)^p
		\biggr]
		\leq c_{p,\lambda} \bigl( h(n-m) \bigr)^p. 
	\end{split}
\end{equation*}
We now turn to the process $(T^2_n)_{n \geq 0}$. It is a martingale. By the Burkholder-Davis-Gundy inequality,
\begin{equation*}
	\begin{split}
		{\mathbb E} 
		\Bigl[
		\bigl\vert
		T_n^2 - 
		T_m^2
		\bigr\vert^{p}
		\Bigr]
		&\leq {\mathbb E} 
		\biggl[
		\biggl\vert
		\sum_{k=m+1}^n  
		\Bigl\langle e^{h \Delta} X_{k-1}^h - e^{kh \Delta} X_0, 
		\int_{(k-1)h}^{kh} e^{ (kh-s) \Delta} d W_s
		\Bigr\rangle
		\biggr\vert^{p}
		\biggr]
		\\
		&\leq 
		{\mathbb E} 
		\biggl[ 
		\biggl( 
		\biggl[
		\sum_{k=m+1}^\cdot  
		\Bigl\langle e^{h \Delta} X_{k-1}^h - e^{kh \Delta} X_0, 
		\int_{(k-1)h}^{kh} e^{ (kh-s) \Delta} d W_s
		\Bigr\rangle
		\biggr]_n
		\biggr)^{p/2} 
		\biggr],
	\end{split}
\end{equation*}
where the notation $[\, \cdot \, ]_n$ denotes the quadratic variation up the $n^{th}$ instant (note that here, this is from the $(m+1)^{st}$ instant). This may be estimated by writing
\begin{align}	
	\label{eq bracket calc}
		&\biggl[
		\sum_{k=m+1}^\cdot  
		\Bigl\langle e^{h \Delta} X_{ k-1 }^h - e^{kh \Delta} X_0, 
		\int_{(k-1)h}^{kh} e^{ (kh-s) \Delta} d W_s
		\Bigr\rangle
		\biggr]_n 
		\nonumber
		\\
		&= \sum_{k=m+1}^n 
		\biggl( 
		\Bigl[ \reallywidehat{ e^{h \Delta} X_{k-1 }^h - e^{kh \Delta} X_0}^{0} 
		\Bigr]^2 h   +
		\sum_{\ell \in {\mathbb N}} 
		\Bigl[ \reallywidehat{ e^{h \Delta} X_{k-1}^h - e^{kh \Delta} X_0}^{\ell} 
		\Bigr]^2 \hspace{-2pt}
		\int_{(k-1)h}^{kh} 
		\hspace{-7pt} 
		 \frac{e^{-  {8} \pi^2 (kh-s) \ell^2}} {\ell^{2\lambda}}{}ds
		\biggr)
		\nonumber
		\\
		&\leq h  \sum_{k=m+1}^n 
		\sum_{\ell \in {\mathbb N}_0} 
		\Bigl[ \reallywidehat{ e^{h \Delta} X_{k-1}^h - e^{kh \Delta} X_0 }^{\ell} 
		\Bigr]^2= h \sum_{k=m+1}^n 
		\Bigl\| e^{h \Delta} X_{k-1}^h - e^{kh \Delta} X_0  
		\Bigr\|_2^2.  
\end{align}
Applying the generalised means inequality and using, from Corollary 
\ref{cor:2.5}, that $\sup_{k=0,\cdots,N_0}{\mathbb E}[ \| X^h_{k h} \|_2^{p}] \leq C_{p,\lambda,T,\E[\left\lVert X_0\right\rVert_2^{p}]}$, 
one obtains
\begin{equation*}
	\begin{split}
	{\mathbb E} 
		\Bigl[
		\bigl\vert
		T_n^2 - 
		T_m^2
		\bigr\vert^{p}
		\Bigr]
&		\leq  
		c_{p,\lambda,T,\E[\left\lVert X_0\right\rVert_2^{p}]} \bigl( h(n-m) \bigr)^{\frac{p}{2}}.  
	\end{split}
\end{equation*} 
Returning to \eqref{eq:IV:1}, via application of the Kolmogorov-{Chentsov} continuity theorem {[see Theorem 1.2.1 in \cite{revuzYor1999ctsMartAndBM}]} (to the linear interpolation of the two processes, {$T^1$ and $T^2$ in \eqref{eq:IV:1}),} {we deduce} that, for $\alpha \in (0,(\frac{p}{2}-1)/2p)$,
\begin{equation*}
	\| X_{n}^h - e^{nh \Delta} X_0 \|_2
	\leq \Xi^{{h}} (nh)^{\alpha},
	\qquad {n \in \{0,\cdots,N_0\},}
\end{equation*}
almost surely for a (non-negative) random variable $\Xi^{{h}}$ with 
a finite $L^{2p}({\mathbb P})$-moment
that satisfies 
${\mathbb E}[(\Xi^{{h}})^{2p}] \leq c_{p,\lambda,T,\E[\left\lVert X_0\right\rVert_2^{2p}]}$. 
 Consequently, \begin{equation*}
	\| X_{n}^h -  X_0 \|_2
	\leq \Xi^{{h}} (nh)^{\alpha} + w(nh), \qquad  {n \in \{0,\cdots,N_0\},}
\end{equation*}
where $w(x)$ is a random variable that depends on $X_0$, that tends almost surely to $0$ with $x$ and
that is dominated by $2 \| X_0\|_2$. 
Notice that this implies in particular that 
\begin{equation}
\label{eq:lp:moments:sup}
{\mathbb E} \Bigl[ \sup_{n : nh \leq T+h}  \| X_n^h \|_2^{2p} \Bigr] \leq c_{p,\lambda,T,\E[\left\lVert X_0\right\rVert_2^{2p}]}.
\end{equation} 
\vspace{5pt}

\noindent \textit{Second step.} 
From here, we take $p \geq 2$. 
Differently from the first step, we now 
estimate the increments of the scheme in the space $H^{-1}_{\rm sym}({\mathbb S})$. 
To begin, apply the triangle and generalised means inequalities:
\begin{equation}\label{eq gen mean tight}
	\begin{split}
		&\E  \left[ \left\lVert X^h_{n}-X^h_{m}\right\rVert_{2,-1}^{2p}  \right] 	
		\\
		& \leq  2^{2p-1} \E \left[ \left\lVert X^h_{n }-e^{(n-m)h\Delta}X^{h}_{m} \right\rVert_{2,-1}^{2p} + \left\lVert X^{h}_{m} - e^{(n-m)h\Delta}X^h_{m} \right\rVert_{2,-1}^{2p} \right]. \\	 
	\end{split}
\end{equation}
The second summand in the above right hand side is simpler to handle.
{By expanding in Fourier modes the left hand side below, one gets, for any $u \in L^2(\bS)$,  
\begin{equation}
\label{eq gen mean tight.2}
\begin{split}
	\lVert e^{(t-s)\Delta}u - u \rVert_{2,-1}^2 
	 %
		\leq c (t-s)  \lVert u \rVert_2^2,
	\end{split}
\end{equation}
for a universal constant $c>0$. This is enough for the desired estimation. For the first summand in \eqref{eq gen mean tight}, one proceeds via the following sequence of inequalities. Starting with
Lemmas
\ref{lem:nonexpansion} and \ref{lem:heatpreservessymmetry},}
\begin{equation}
\label{eq:Fourier:increment}
	\begin{split}
		 \left\lVert    X^h_{n}-e^{(n-m)h\Delta}X^{h}_{m} \right\rVert_{2,-1}^{2p}   
		&\leq 
		 \left\lVert    X^h_{n}-e^{(n-m)h\Delta}X^{h}_{m} \right\rVert_{2}^{2p}   
		\\
		 &\leq    \left\lVert e^{h\Delta}X^h_{n-1}+\int_0^he^{(h-s)\Delta}d W^{n}_s -e^{(n-m)h\Delta}X^{h}_{m} \right\rVert_{2}^{2p},  \\  
	\end{split} 
\end{equation}			
which may be estimated by means of 
\eqref{eq pth norm ineq}, by considering the process
\begin{equation*}
\hat{X}_s^{h,n-1} := e^{s \Delta} \Bigl[ X_{n-1}^h - e^{(n-1-m)h \Delta} X_m^h \Bigr] + \int_0^s e^{(s-r) \Delta} d W_r^n, \quad s \in [0,h].
\end{equation*}
Following the same sequence of inequalities as in 
\eqref{eq unif est0.2}, \eqref{eq unif est0.2.1}, 
\eqref{eq unif est0.2.2}
and 
\eqref{eq unif est0.2.3}, we obtain\footnote{\color{black} Although the reader may find the computations reminiscent of 
	\eqref{eq:IV:1}, the objective is in fact different. In 
\eqref{eq:IV:1}, the goal is to apply Kolmogorov-Chentsov's theorem
to the process $(\| X_n^h - e^{nh \Delta} X_0 \|_2^2)_{0 \leq \lfloor n/h \rfloor \leq T}$. The purpose here is obviously not the same.}
\begin{equation*}
	\begin{split}	
		&\E   \left[\left\lVert
		X^h_{n}-e^{(n-m)h\Delta}X^{h}_{m} 
		 \right\rVert_2^{2p}    \right] 
		\leq 
		\E  
		 \left[\left\lVert
		X^h_{n-1}-e^{(n-1-m)h\Delta}X^{h}_{m} 
		 \right\rVert_2^{2p}    \right] 
		 \\
	&\hspace{15pt} + 	  c_{p,\lambda} h \left(  \E 
			 \left[\left\lVert
		X^h_{n-1}-e^{(n-1-m)h\Delta}X^{h}_{m} 
		 \right\rVert_2^{2p-2}    \right] 
		+   h^{p-1}   \right),\\
		\end{split}
\end{equation*}
which gives, by iteration, 
\begin{equation*}
	\begin{split}	
		&\E   \left[\left\lVert
		X^h_{n}-e^{(n-m)h\Delta}X^{h}_{m} 
		 \right\rVert_2^{2p}   \right] 
		 \leq c_{p,\lambda} h
		\sum_{k=m+1}^{n-1}
		 \left(  \E 
			 \left[\left\lVert
		X^h_{k}-e^{(k-m)h\Delta}X^{h}_{k} 
		 \right\rVert_2^{2p-2}    \right] 
		+   h^{p-1}   \right).\\
		\end{split}
\end{equation*}
We proceed by induction on $p$, assuming for a while that $p$ is an integer (greater than or equal to 1). 
 {When $p=1$, the above inequality yields
$\E   [ \lVert
		X^h_{n}-e^{(n-m)h\Delta}X^{h}_{m} 
		 \rVert_2^{2}   ] 
		\leq c_{1,\lambda} h (n-m).$} By induction, we get, for any $p \in {\mathbb N}$, 
\begin{equation}
\label{eq gen mean tight.4}
	\begin{split}	
		\E   \left[\left\lVert
		X^h_{n}-e^{(n-m)h\Delta}X^{h}_{m} 
		 \right\rVert_2^{2p}    \right] 
		\leq c_{p,\lambda} \bigl( h (n-m) \bigr)^p.\\
		\end{split}
\end{equation} 
When $p \geq 1$, we apply \eqref{eq gen mean tight.4} to $\lceil p \rceil$ and then get the conclusion for $p$ by H\"older's inequality applied with exponent $\lceil p \rceil/p$.
By 
\eqref{eq gen mean tight}, 
\eqref{eq gen mean tight.2}, 
\eqref{eq:Fourier:increment}
and
\eqref{eq gen mean tight.4}, we obtain
\begin{equation*}
	\begin{split}
		&\E  \left[ \left\lVert X^h_{n}-X^h_{m}\right\rVert_{2,-1}^{2p}  \right] 	
		 \leq  
		c_{p,\lambda} \bigl( h (n-m) \bigr)^p  \Bigl( 1 + {\mathbb E} \Bigl[ \sup_{0 \leq k \leq n} 
		\| X^h_k\|_2^{2p} \Bigr] \Bigr). \\	 
	\end{split}
\end{equation*}
Inserting the conclusion of the first step (see \eqref{eq:lp:moments:sup}), 
we have established that
\begin{equation}
\label{eq:3rdstep:1}
\begin{split}
	{\mathbb E} \Bigl[ \bigl\| X_{n}^h - 
	X_{m}^h \bigr\|_{2,-1}^{2p}
	\Bigr]  
	&\leq
	c_{p,\lambda} (nh - mh)^{p}.
\end{split}
\end{equation}

\noindent \textit{Conclusion.}
The conclusion of the first step (see  \eqref{eq:lp:moments:sup} again) says that, for all $t \in [0,T]$,  
the family 
$\{ \tilde X_t^h\}_{0 < h < 1}$
is tight on $H^{-1}_{\rm sym}({\mathbb S})$ (as bounded subsets of 
$L^2_{\rm sym}({\mathbb S})$ are relatively compact in 
$H^{-1}_{\rm sym}({\mathbb S})$).
And the conclusion of the second step (see \eqref{eq:3rdstep:1})
says that, for any $T>0$ and any $\varepsilon \in (0,1)$, 
the trajectories $((\tilde X^h_t)_{0 \le t \le T})_{0 < h < 1}$ 
live, with probability greater than $1-\varepsilon$, in a set of equicontinuous trajectories from 
$[0,T]$ to $H^{-1}_{\rm sym}({\mathbb S})$. 
\qed
\end{proof}

\color{black}
\section{Limiting Dynamics: Characterisation and Well-posedness}


\label{sec:3}

This section addresses the weak limits of the schemes. As discussed in the introduction, it is expected that those weak limits, say denoted by $X$, should satisfy a reflected stochastic differential equation in infinite dimension,  understood in the sense,
\begin{equation}\label{eq reflSHE}
\begin{split}
dX_t = \Delta X_t dt + dW_t + d\eta_t, \quad t \geq 0,
\end{split} 
\end{equation}
{for $X_0$ satisfying the standing assumption \eqref{eq:assumption:X0}. Here, $(\eta_t)_{t \geq 0}$ should be understood as 
a forcing term that reflects $X$ into $U^2({\mathbb S})$. Equations of this type have been treated {in} \cite{rocknerZhuZhu2012reflectInfDimConv}, but in the absence of an integration by parts formula (analogous {to} \cite{Zambotti1,Zambotti2} for the models in \cite{Donati-Martin,NualartPardoux}) these results do not apply in our setting. Consequently, our approach follows the application of limit theorems to the tested/weak behaviour of the schemes. Below, we often refer to \eqref{eq reflSHE} - with accompanying conditions on the process $\eta$ - as the {\bf rearranged stochastic heat equation (or rearranged SHE)}. The reader impatient for the exact solution concept should skip momentarily to Definition \ref{def:existence}, with caution that the fourth condition contains an integral that is defined en route. 
\newline\indent
The purpose of this section is thus to identify conditions satisfied by any weak limit that are, ultimately, sufficient to prove that the weak limit (of the schemes) is unique. 
This goal is reached in a series of five subsections. In Subsection 
\ref{subse:3:1}, we prove that weak limits satisfy 
an equation of the form 
\eqref{eq reflSHE}. Subsection 
\ref{subse:3:2} concerns the construction of an integral with respect to the reflection process $(\eta_t)_{t \geq 0}$. Notably this integral is non-decreasing with respect to $U^2({\mathbb S})$-valued integrand processes. Moreover, in Subsection \ref{subse:3:3} we establish an orthogonality property between $X$ and $\eta$, key to proving uniqueness 
of the weak limit. 
The rigorous definition of a solution to 
\eqref{eq reflSHE}
together with the main statement of its existence and uniqueness 
are given in
Subsection
 \ref{subse:3:4}, see Definition 
 \ref{def:existence}
 and Theorem \ref{thm:main:existence:uniqueness}. We end the section 
with a proof of the Lipschitz regularity of the flow induced by the solution in Subsection 
\ref{subse:3:5}. 
}

\subsection{Testing of the weak limits}
 
\label{subse:3:1}

Our analysis {of the weak limits relies on the 
statement below, in which we use the notion of `time-locally bounded trajectories with 
respect to $ \| \cdot \|_2$ (resp. $\| \cdot \|_{2,-2}$)'. For a normed vector space $(E,\| \cdot \|)$, a function $t \in [0,\infty) \mapsto f_t \in E$ is said to be `time-locally bounded with respect to 
$\| \cdot \|$' if, for any $T>0$, $\sup_{0 \le t \le T} \| f_t \| < \infty$.}

\begin{proposition}
\label{prop:weak:limit:1}
Let $(X_t,W_t)_{t \geq 0}$ be a weak limit (over 
$\cC([0,\infty),{H_{\rm sym}^{-1}({\mathbb S}) \times L_{\rm sym}^2({\mathbb S})})$ equipped with the topology of uniform convergence on compact subsets) 
of the processes $\{ (\tilde X^h_t,W_t)_{t \geq 0}\}_{h >0}$ as $h$ tends to $0$, this weak limit being 
constructed on the same filtered probability space $(\Omega,{\mathcal A},{\mathbb F},{\mathbb P})$ as the scheme itself
and the second component $(W_t)_{t \geq 0}$
of the weak limit {abusively denoted the same} as the noise in the scheme. 

Then,
$(X_t,W_t)_{t \geq 0}$
is ${\mathbb F}$-adapted, 
$(X_t)_{t \geq 0}$ is 
$U^2({\mathbb S})$-valued (i.e., 
each $X_t$ has symmetric and non-increasing values, 
see 
\eqref{def:1.2:U2})
{and has time-continuous trajectories with respect to $\| \cdot \|_{2,-1}$
and time-locally bounded trajectories with respect to $\| \cdot \|_2$,}
 and 
$(W_t)_{t \geq 0}$
is an $L^2_{\rm sym}({\mathbb S})$-valued $Q$-Brownian motion 
with respect to ${\mathbb F}$. 
Moreover, {there exists an ${\mathbb F}$-adapted process $(\eta_t)_{t \geq 0}$ 
with values in $H_{\rm sym}^{-2}({\mathbb S})$, 
with time-continuous trajectories with respect to $\| \cdot \|_{2,-3}$ and time-locally bounded trajectories with respect to $\| \cdot \|_{2,-2}$,}
such that, with probability 1, for any 
$u \in H^2_{\rm sym}({\mathbb S})$:
\begin{enumerate}
\Item \begin{equation}\label{eq RSHE}
	\begin{split}
		{\text{$\forall$ $t>s \geq 0$,}} \quad  
		\langle  X_t-{X}_s ,u \rangle    =     \int_s^t \langle   {X}_r  ,\Delta     u \rangle dr +\langle W_t-W_s ,u \rangle + \langle \eta_t-\eta_s,u\rangle,    
\end{split} 
\end{equation}
\item if 
$u$ is non-increasing in the sense of  
Definition \ref{def:1.2:U2}, then
the path $(\langle \eta_t ,u \rangle)_{t \geq 0}$ is non-decreasing (with $t$) 
and starts from $0$ at time $0$. 
\end{enumerate} 
\end{proposition}
The hypothesis that the weak limit can be constructed on the same space $(\Omega,{\mathcal A},{\mathbb P})$ 
as in Section 
\ref{sec:2} can be made without loss of generality. In short, this just requires the probability space to be `rich enough' (e.g., it is 
 an atomless Polish probability space, {on which we can construct arbitrarily distributed random variables with values in any other Polish space}), which as additional assumption, is not a hindrance for us.  
 The claim that 
 $(\Omega,{\mathcal A},{\mathbb P})$ 
 can be equipped with the filtration ${\mathbb F}$ requires {a little} more care: 
 ${\mathbb F}$ cannot be any given filtration, which is a common feature with weak limits of processes. We clarify the choice of ${\mathbb F}$
 in the proof below. We do this only for the convenience of using the same notation ${\mathbb F}$ for this specific choice, as we are convinced that there is no risk of confusion for the reader. 
Similarly, denoting the second component of the weak limit by
$(W_t)_{t \geq 0}$ is also rather abusive, but is justified by the fact that the second
component's law in any weak limit remains that of a $Q$-Brownian motion with values in 
$L^2_{\rm sym}({\mathbb S})$, see \eqref{def:tilde:noise}
and 
\eqref{eq:Q:covariance}. 
	\begin{remark}
	\label{rem:weak:limit:1}
	The proof of Proposition \ref{prop:weak:limit:1}
	shows that the shape of the process $(\eta_t)_{t \geq 0}$ in 
	\eqref{eq RSHE} can be further clarified. 
	Indeed, 
	denote by $(V_t)_{t \geq 0}$ the solution to the (usual) SHE with 
	$X_0$ as initial condition
	and with $(W_t)_{t \geq 0}$ as driving noise (recalling that we keep this notation 
	in the limit setting), i.e., the $L^2_{\rm sym}({\mathbb S})$-valued process,
	\begin{equation}
		\label{eq:V:SPDE}
		V_t :=   e^{t \Delta} X_0 +  \int_0^t e^{(t-s) \Delta} 
		d W_s, \quad t \geq 0,
	\end{equation}
	 and let $Y_t := X_t-V_t$. Then, for $t \geq 0$, one has, with probability 1, 
	for any $v \in H_{\rm sym}^2({\mathbb S})$,
	\begin{equation}
		\label{eq:rem:weak:limit:1}
		\begin{split}
			\forall t \geq 0, \quad \langle \eta_t,v \rangle   =\langle  {Y}_t  , v \rangle  - \int_0^t \langle    {Y}_r  ,\Delta v  \rangle dr.
		\end{split}
	\end{equation}
\end{remark}
The following formal argument (made rigorous below) gives intuition for item {2} of Proposition \ref{prop:weak:limit:1}. For $\delta$ small, and $v$ non-increasing.
		\begin{equation*}
			\begin{split}
				\quad & \langle \eta_{t+\delta} -  \eta_t,v \rangle  =   \langle  {Y}_{t+\delta} -{Y}_{t}    , v \rangle  - \int_t^{t+\delta} \langle    {Y}_r  ,\Delta v  \rangle dr 
				\\
				&\approx 	\langle  {Y}_{t+\delta}     -    e^{\delta \Delta} {Y}_t  , v  \rangle 			
				=
				\Bigl\langle  {X}_{t+\delta}    -    e^{\delta \Delta} {X}_t
				-\int_t^{t+\delta}e^{(t+\delta-s)\Delta}dW_s
				 , v
				   \Bigr\rangle 	
				   \\
				&  =   \Bigl\langle \Bigl(e^{\delta\Delta}X_t+\int_t^{t+\delta}e^{(t+\delta-s)\Delta}dW_s
				\Bigr)^*-  \Bigl(e^{\delta\Delta}X_t+\int_t^{t+\delta}e^{(t+\delta-s)\Delta}dW_s
				\Bigr) , v \Bigr\rangle \geq 0,
			\end{split}
		\end{equation*}
	the last line following from Lemma \ref{lem:hardy}. To implement this argument onto the scheme, we let 
	 (with the same notation as in \eqref{eq scheme n}
for $W^{n+1}$ and with 
$V$ as in 
\eqref{eq:V:SPDE}):
\begin{equation}
	\label{eq V n}
	\begin{split}
	V_n^h := V_{nh}, \quad \textrm{i.e.} \quad 
		V_{n+1}^h = &  e^{h\Delta}V_{n}^h+\int_0^{h} e^{(h-s)\Delta}dW^{n+1}_s, \quad V_0^h =  X_0.
	\end{split}
\end{equation}
 The so-called shifted scheme, $X^h-V^h=(X^h_n-V^h_n)_{n \geq 0}$ is denoted ${Y}^h=(Y^h_n)_{n \geq 0}$, so that  \eqref{eq scheme n} may be rewritten as
\begin{align}
X_{n+1}^h &= \left( V_{n+1}^h +e^{h\Delta}(X^h_n-V_n^h) \right)^* , \quad n \geq 0. \nonumber
\\
\label{eq scheme rewrite}
Y_{n+1}^h &= \left( V_{n+1}^h +e^{h\Delta} Y_n^h \right)^* - V_{n+1}^h, \quad n \geq 0.
\end{align}
Following  
	\eqref{eq interpolation}, we introduce the
	 interpolations $\tilde{Y}^{h}_{t}$ and $\tilde{V}^{h}_{t}$ of  $Y^h$ and $V^h$. Then, 
 \begin{equation}
	\label{eq interpolation:tilde:Y:2}
	\tilde{Y}^{h}_{t} = \tilde X_t^h - \tilde V_t^h,  
	 \quad t \geq 0.
\end{equation}
\begin{proof}[of Proposition \ref{prop:weak:limit:1}.]
Throughout the proof, we fix $T>0$. {It suffices to study the weak limits, as $h$ tends to $0$, of $\{\tilde{X}^h,W\}$ on $[0,T]$}. 
{For a given $h>0$, with probability 1 and for any $u\in 
{H_{\rm sym}^3({\mathbb S})}$,} 
\begin{equation}\notag
\begin{split}
	\bigl\langle  {Y}^h_{n+1}- {Y}^h_n ,u \bigr\rangle &=  \bigl\langle  {Y}^h_{n+1}-e^{h\Delta} {Y}^h_n ,u \bigr\rangle + \bigl\langle  (e^{h\Delta}-I) {Y}^h_n ,u \bigr\rangle
	\\
	&
	=  \bigl\langle {Y}^h_{n+1}-e^{h\Delta} {Y}^h_n ,u \bigr\rangle + 
	\int_0^h\bigl\langle   e^{s\Delta} {Y}^h_n,\Delta u \big\rangle \,ds,
\end{split} 
\end{equation}
where we used the identity $\partial_s e^{s \Delta} = \Delta e^{s \Delta}$. 

Rearranging, and working under the additional assumption that $u$ is non-increasing, we use the rewritten shifted scheme \eqref{eq scheme rewrite} 
{and Lemma \ref{lem:hardy}} to show:
\begin{align}
\label{eq scheme eta cand}
&\langle  {Y}^h_{n+1}- {Y}^h_n ,u \rangle - 
	\int_0^h   \big\langle   e^{s\Delta} {Y}^h_n\,,\Delta u \big\rangle \, ds 
	\\
	&=  \langle  {Y}^h_{n+1}-e^{h\Delta} {Y}^h_n ,u \rangle 
	=
\Bigl\langle  \left( V_{n+1}^h +e^{h\Delta} Y_n^h \right)^* - \left( V_{n+1}^h 
+e^{h\Delta} {Y}^h_n
\right),u \Bigr\rangle
	\geq 0, \nonumber  
\end{align} 
We rewrite the second term {on the first line}:
\begin{equation}
\label{eq:3:11:aaa}
	\begin{split}
	{  \int_0^h \langle    e^{s\Delta} {Y}^h_n ,\Delta u \rangle ds  
	= \int_0^h \langle    {Y}^h_n , ( e^{s\Delta}-I) \Delta u \rangle ds +  h \langle {Y}^h_n ,\Delta u \rangle}.
	\end{split}
\end{equation} 
{
Summing over $n$ and letting $N_r = \lfloor r / h \rfloor$, for $r>0$, 
we get for any $(s,t) \in [0,T]^2$,}
\begin{align}\label{eq scheme to cts} 
	& \left\lvert \sum_{n=N_s}^{N_t} 
	\int_0^h  \bigl\langle   e^{r\Delta} {Y}^h_n ,\Delta u \bigr\rangle \,dr - \int_s^t \left\langle  \tilde{Y}_r^h ,\Delta u \right\rangle dr \right\rvert  
	\nonumber
	\\
     &  \leq  	\left\lvert \sum_{n=N_s}^{N_t}  \int_0^h  \bigl\langle {Y}^h_n,  \left(  e^{r\Delta}-I \right) \Delta u \bigr\rangle \,dr   \right\rvert  +     \left\lvert \sum_{n=N_s}^{N_t} h \left\langle       Y^h_n  ,\Delta u \right\rangle - 
   \int_s^t \langle  \tilde{Y}_r^h,\Delta u \rangle dr \right\rvert \nonumber
   \\ 
    & \leq  c_T \sup_{0\leq r \leq h}
      \bigl\lVert ( e^{r\Delta} - I ) \Delta u \bigr\|_2 \sup_{n \in \{0,\cdots,N_t\}} \Vert {Y}^h_n  \|_2 +   \left\lvert \sum_{n=N_s}^{N_t} h \langle       Y^h_n  ,\Delta u \rangle - 
   \int_s^t \langle  \tilde{Y}^h_r,\Delta u \rangle dr \right\rvert
   \nonumber
   \\ 
   &=: T_1^h(t) + T_2^h(s,t).  
\end{align} 
{Since $\Delta u \in L^2({\mathbb S})$,
we know that }
{$\lim_{h\searrow0}\sup_{0\leq r \leq h}
       \|  ( e^{r\Delta} - I ) \Delta u  \|_2 =0.$}
      Together with
      Proposition 
      \ref{prop:tightness} (recalling that 
      $Y^h_n=(X^h_n-V^h_n)_{n \geq 0}$), we deduce that 
%
\begin{equation}
\label{eq:T1:sec:3}
\forall \varepsilon >0, 
\quad 
\lim_{h \searrow 0}
     {\mathbb P}
     \Bigl( \bigl\{ \sup_{0 \leq t \leq T} T_1^h(t) \geq \varepsilon \bigr\}
     \Bigr) = 0.
     \end{equation}
Similarly, by tightness of 
$\{ \tilde X^h     \}_{h \in (0,1)}$ 
on 
$\cC([0,\infty),{H^{-1}_{ {\rm sym}}({\mathbb S})})$, we deduce that 
$\{ \tilde Y^h     \}_{h \in (0,1]}$ 
is {also} tight on  
$\cC([0,\infty),{H^{-1}_{ {\rm sym}}({\mathbb S})})$ (by \eqref{eq interpolation:tilde:Y:2}), 
from which we easily get that
(since $\tilde Y^h$ is the linear interpolation of $Y^h$
{and because $\Delta u \in H^1_{\rm sym}({\mathbb S})$}) 
%
\begin{equation}
\label{eq:T2:sec:3}
\forall \varepsilon >0, 
\quad 
\lim_{h \searrow 0}
     {\mathbb P}
     \Bigl( \bigl\{ \sup_{0 \leq s < t \leq T} T_2^h(s,t) \geq \varepsilon \bigr\}
     \Bigr) = 0.
     \end{equation} 
      {Returning} to 
      \eqref{eq scheme to cts},
      the last two displays 
      \eqref{eq:T1:sec:3}
      and
      \eqref{eq:T2:sec:3}
     yield, {for all $\varepsilon >0,$}
     \begin{equation}
     \label{eq:T:sec:3:2}
\lim_{h \searrow 0}
     {\mathbb P}
     \biggl( \sup_{0 \leq s < t \leq T} 
     \left\lvert \sum_{n=N_s}^{N_t} \int_0^h   \bigl\langle e^{r\Delta} {Y}^h_n ,\Delta u \bigr\rangle
     \,dr - \int_s^t \left\langle  \tilde{Y}_r^h ,\Delta u \right\rangle dr \right\rvert
     \geq \varepsilon \biggr) = 0.
     \end{equation}      
     {It remains to insert \eqref{eq:T:sec:3:2} into 
     \eqref{eq scheme eta cand}, by summing the (non-negative) left-hand side of \eqref{eq scheme eta cand} from $N_s$ to $N_t$ and subtracting this from the 
     main term in 
      \eqref{eq:T:sec:3:2}.     This 
   supplies  us with a term $T_3^h(s,t)$ such that}
     \begin{align}
        \label{eq:T:sec:4}
\forall (s,t) \in [0,T]^2 : s< t, \quad \langle  \tilde{Y}^h_{t}- \tilde{Y}^h_{s},u \rangle - 
	\int_s^t \left\langle  \tilde{Y}_r^h ,\Delta u \right\rangle dr
	\geq T_3^h(s,t), 
\\
     \label{eq:T:sec:5}
{\text{and}}\quad \forall \varepsilon >0, 
\quad 
\lim_{h \searrow 0}
     {\mathbb P}
     \biggl( \sup_{0 \leq s < t \leq T} 
\bigl\vert     T_3^h(s,t)
\bigr\vert
     \geq \varepsilon \biggr) = 0.
     \end{align}  
Now we let $h$ tend to $0$. Following the statement, we slightly abuse notation and write 
$(X_t,W_t)_{0 \leq t \leq T}$ a weak limit of $\{ (\tilde X^h_t,W_t)_{0 \leq t \leq T} \}_{h \in (0,1]}$. 
We denote by ${\mathbb F}$, the usual augmentation of the filtration generated by 
$(X_t,W_t)_{0 \leq t \leq T}$. 
{Clearly}, 
$(W_t)_{0 \leq t \leq T}$
is an $L^2_{\rm sym}({\mathbb S})$-valued $Q$-Brownian motion with respect to 
${\mathbb F}$. 
Along the same subsequence, 
$\{ (\tilde X^h_t,W_t,\tilde V^h_t,\tilde Y^h_t)_{0 \leq t \leq T} \}_{h \in (0,1]}$
converges to 
$(X_t,W_t,V_t,Y_t)_{0 \leq t \leq T}$, where $(V_t)_{0 \leq t \leq T}$ solves the {SHE}  
\eqref{eq:V:SPDE}
(driven by the limit process $W$)
and $Y_t =X_t - V_t$, for $t \in [0,T]$ (in particular, $Y_0=0$). 
Obviously, 
$(X_t,W_t,V_t,Y_t)_{0 \leq t \leq T}$
is ${\mathbb F}$-adapted. 
{By 
Proposition 
\ref{prop:closedness:H-1}, 
$(X_t)_{0 \leq t \leq T}$ 
takes values in 
$U^2({\mathbb S})$ (since 
$(\tilde X^h_t)_{0 \leq t \leq T}$ 
does, for 
any $h>0$) in the sense that, with probability 1, for all 
$t \geq 0$, $X_t$ belongs to $U^2({\mathbb S})$.
By construction, it is continuous with respect to $\| \cdot \|_{2,-1}$.
Using Proposition \ref{prop:tightness}
together with the fact that 
the mapping $(x_t)_{0 \le t \le T} \mapsto \sup_{0 \le t \le T} \| x_t \|_2$ is lower semi-continuous with respect to 
the uniform convergence topology on 
$\cC([0,T],H^{-1}_{\rm sym}({\mathbb S}))$, we deduce that, with probability 1, 
$\sup_{0 \le t \le T} \| X_t \|_2 < \infty$.
In turn, 
$(Y_t)_{0 \leq t \leq T}$ is valued in $L^2_{\rm sym}({\mathbb S})$, with bounded 
trajectories (the bound being random), because
$(V_t)_{0 \leq t \leq T}$ has continuous trajectories with values in $L^2_{\rm sym}({\mathbb S})$.
Moreover, 
using 
   \eqref{eq:T:sec:4}
     and  \eqref{eq:T:sec:5}, we obtain, when $u \in H^3_{\rm sym}({\mathbb S})$ is non-increasing, 
\begin{equation}\label{eq eta cts:22}
	\begin{split}
\forall (s,t) \in [0,T]^2 : s<t, \quad		\langle  {Y}_t- {Y}_s , u \rangle  - \int_s^t \langle    {Y}_r  ,\Delta u  \rangle dr  \geq 0. \\
	\end{split} 
\end{equation}
The above is true, for any $u \in H^3_{\rm sym}({\mathbb S})$, with probability 1, for all 
$t >s \geq 0$. By a separability argument, it is true with probability 1, for all $u \in H^3_{\rm sym}({\mathbb S})$
and for all $t>s \geq 0$. And, then by a new density argument (using the fact that 
$(Y_t)_{t \geq 0}$ has time-locally bounded trajectories with respect to $\| \cdot \|_{2}$), 
it is true with probability 1, for all $u \in H^2_{\rm sym}({\mathbb S})$ and for all $t > s \geq 0$. This prompts us to let, for any $v \in H^2_{\rm sym}({\mathbb S})$, 
\begin{equation*}
\forall t \in [0,T], \quad 
\langle \eta_t,v \rangle :=\langle  {Y}_t  , v \rangle  - \int_0^t \langle    {Y}_r  ,\Delta v  \rangle dr.
\end{equation*}
Obviously, the processes $(\eta_t)_{0 \le t \le T}$ 
has bounded trajectories with respect 
to $\| \cdot \|_{2,-2}$ and continuous trajectories with respect to 
$\| \cdot \|_{2,-3}$. 
By 
\eqref{eq eta cts:22}, we deduce that, a.s., 
for any $u  \in H^2_{\rm sym}({\mathbb S}) \cap U^2({\mathbb S})$,  $(\langle \eta_t,u \rangle)_{0 \leq t \leq T}$ is non-decreasing. This proves item 2 in the statement. Moreover, 
by replacing $(Y_t)_{0 \leq t \leq T}$ by $(X_t-V_t)_{0 \leq t \leq T}$ in the definition of $(\eta_t)_{0 \leq t \leq T}$ and 
by recalling {from \eqref{eq:V:SPDE}} that, for any $v \in H^2({\mathbb S})$,  
\begin{equation*}
\langle  {V}_t  , v \rangle  - \int_0^t \langle    {V}_r  ,\Delta v  \rangle dr = 
\langle  e^{t \Delta } X_0 , v \rangle
+
\int_0^t e^{(t-r) \Delta} 
d W_r,
\end{equation*}
we easily verify item 1 in the statement, completing the proof}.  
\qed
\end{proof}
{
In fact, the local in time boundedness property of the trajectories 
of $(X_t)_{t \geq 0}$ (in $\| \cdot \|_2$)
and 
$(\eta_t)_{t \geq 0}$ (in $\| \cdot \|_{2,-2}$)
can be improved.
Using again Proposition \ref{prop:tightness}
and the notation defined in Remark 
\ref{rem:weak:limit:1} together with the fact that 
the mapping $(x_t)_{0 \le t \le T} \mapsto \sup_{0 \le t \le T} \| x_t \|_2$ is lower semi-continuous with respect to 
the uniform convergence topology on 
$\cC([0,T],H^{-1}_{\rm sym}({\mathbb S}))$, we have}  
{\begin{proposition}
\label{prop:weak:limit:2}
For any $p \geq 1$ and any $T>0$, 
there exists constant $C_{p,\lambda,T,{\mathbb E}[\Vert X_0 \Vert^{2p}]}$, such that for any weak limit $X$
as in Proposition 
\ref{prop:weak:limit:1} and $Y$ as in Remark \ref{rem:weak:limit:1},
\begin{equation*}
{\mathbb E} \Bigl[ \sup_{t \in [0,T]}  \| X_t \|_2^{2p} \Bigr] + {\mathbb E} \Bigl[ \sup_{t \in [0,T]}  \| Y_t \|_2^{2p} \Bigr] \leq C_{p,\lambda,T,{\mathbb E}[\Vert X_0 \Vert^{2p}]}.
\end{equation*} 
\end{proposition}} 

\subsection{Integral with respect to the reflection process}
\label{subse:3:2} 

Our next objective is to construct an integral with respect 
to the reflection process $(\eta_t)_{t \geq 0}$ identified in the statement of Proposition 
\ref{prop:weak:limit:1}. {We make use of the resulting integral in order 
to establish uniqueness of the weak limits obtained in Proposition 
\ref{prop:weak:limit:1}.}

{In the construction, we use the fact that, with probability 1, 
  the path $(\eta_t)_{t \geq 0}$ satisfies the forthcoming two assumptions $({\bf E1})$ and 
  $({\bf E2})$, which are spelled out as follows for  a {\bf deterministic} trajectory 
$(n_t)_{t \geq 0}$:
\begin{itemize}
\item[({\bf E1})]
$t \mapsto n_t$ is a function 
from 
$[0,\infty)$ to 
$H^{-2}_{\rm sym}({\mathbb S})$, {locally bounded with respect to $\| \cdot \|_{2,-2}$ and continuous with respect to $\| \cdot \|_{2,-3}$};
\item[({\bf E2})] 
For any 
  $u \in H^2_{\rm sym}({\mathbb S}) \cap U^2({\mathbb S})$, the 
  function $t \in [0,\infty) \mapsto  \langle \eta_t , u \rangle$ is non-decreasing.
  \end{itemize}}
  {The integral we construct below 
holds for a path $(n_t)_{t \geq 0}$ satisfying only the two assumptions 
({\bf E1}) and ({\bf E2}). In particular, the integral with respect to 
$(\eta_t)_{t \geq 0}$ is then obtained pathwise, by choosing $(n_t)_{t \geq 0}$ as the current realisation of $(\eta_t)_{t \geq 0}$.} 
Now, 
let 
  $u \in H^2_{\rm sym}({\mathbb S}) \cap U^2({\mathbb S})$.
  {Even though 
  $(n_t)_{t \geq 0}$ is continuous with respect to the weaker norm $\| \cdot  \|_{2,-3}$, the fact that 
  it also takes values in $H^{-2}_{\rm sym}({\mathbb S})$ implies that 
  $t \in [0,\infty) \mapsto \langle \eta_t,u \rangle$ is continuous. Then, if} we consider in addition another continuous {\bf deterministic}
trajectory 
$(z_t)_{t \geq 0}$ valued in 
$L^2_{\rm sym}({\mathbb S})$, ({\bf E2}) allows us to define 
\begin{equation}
\label{eq:integral:RS:0}
\biggl( \int_0^t \langle z_r,u \rangle  d\langle n_r, u \rangle \biggr)_{t \geq 0}
\end{equation}
as a time-continuous Riemann-Stieltjes integral.  
From this, we want to give a meaning to 
the as yet informally written integrals
$( \int_0^t z_r \cdot dn_r)_{t \geq 0}$,
where the dot $\cdot$
in the notation is intended to denote a form of duality presence between the integrand and the integrator. Our definition of the integral \label{eq:integral:RS:1} is done by analogy with Parseval's identity, setting $u$ in 
\eqref{eq:integral:RS:0}
to be ({cosine}) elements in the Fourier basis. 
The next step is to expand $\int_0^t z_r \cdot dn_r$ along the ({cosine}) Fourier basis $(e_m)_{m \in {\mathbb N}_0}$, noticing that one may 
decompose each $e_m$ as the difference of two elements of $U^2({\mathbb S})$, $e_m^+$ and $e_m^-$,  with
\begin{equation}
\notag
\begin{split}
e_m^{\pm}(x):= & e_m(0) \iota +\int_0^x \left[ -\1_{(-\nicefrac{1}{2},0]}(y) \bigl(De_m(y)\bigr)_\mp + \1_{[0,\nicefrac{1}{2})}(y) \bigl(De_m(y) \bigr)_\pm \right](y) dy, \\ 
\end{split}
\end{equation} 
with $\iota=1$ if $\pm=+$ and $0$ if $\pm=-$. 
The functions $e_m^+$ and $e_m^-$ are in $U^2({\mathbb S})$ (courtesy of the symmetry properties of $e_m$) and $e_m=e_m^+-e_m^-$. Therefore, one may set:
\begin{equation}
\label{eq fouriersplit2}
\begin{split}
	\int_s^t \langle z_r,e_m \rangle  d\langle n_r,  e_m \rangle:= &  \int_s^t \langle z_r,e_m \rangle  d\langle n_r,  e_m^+ \rangle - \int_s^t \langle z_r,e_m \rangle  d\langle n_r,  e_m^- \rangle. \\ 
\end{split}
\end{equation}   
{For any $\varepsilon >0$, $(z_r)_{r \geq 0}$ can be replaced by 
$(e^{\varepsilon \Delta} z_r)_{r \geq 0}$ in 
\eqref{eq fouriersplit2}, which is a  consequence of 
Lemma \ref{lem:heatpreservessymmetry}. 
The following statement is key in the construction of our integral \eqref{eq:integral:RS:1}.}
\begin{lemma}
\label{lem:integral:RS:bound:0}
{For any $k \in {\mathbb N}$ and any $\varepsilon >0$}, there exists 
a constant $c_{k,\varepsilon}$ such that, for any  two
(deterministic) 
curves 
$(n_t)_{t \geq 0}$ 
and 
$(z_t)_{t \geq 0}$, 
with 
$(n_t)_{t \geq 0}$ satisfying 
{\rm ({\bf E1})} and 
{\rm ({\bf E2})}
and with $(z_t)_{t \geq 0}$ a continuous path in $L^2_{\rm sym}({\mathbb S})$, {and any $T \geq 0$ and $m \in {\mathbb N}_0,$}
\begin{equation}
\label{eq:integral:RS:bound:0}
\begin{split}
&\sup_{t \in [0,T]} 
\biggl\vert 
\int_0^t \langle e^{\varepsilon \Delta} z_r,e_m \rangle
d \langle n_r,e_m^{\pm} \rangle \biggr\vert 
\leq \frac{c_{k,\varepsilon}}{m^k \vee 1} 
\| n_T \|_{2,-2} \, 
\sup_{t \in [0,T]} 
\| z_t \|_{2}.
\end{split}
\end{equation}
When $(n_t)_{t \geq 0}$ is understood as a realisation of 
$(\eta_t)_{t \geq 0}$, the term 
$\| n_T \|_{2,-2}$ 
becomes $\| \eta_T \|_{2,-2}$ and, {with $(Y_t)_{t \geq 0}$ as in 
\eqref{eq:rem:weak:limit:1}}, it can be upper bounded by 
\begin{equation}
\label{eq:integral:RS:bound:0:++}
\| \eta_T \|_{2,-2} \leq c_T 
\sup_{t \in [0,T]}
\| Y_t \|_2.
\end{equation}
\end{lemma} 
In 
\eqref{eq:integral:RS:bound:0}, we use the notation $e_m^{\pm}$  to indicate that the result holds true with 
both $e_m^+$ and $e_m^-$. Also, note that the $L^2$ contributions of $e_m^+$ and $e_m^-$ diverge with $m$. 
This is precisely the reason why we consider integrands of the form $(e^{\varepsilon\Delta} z_t)_{t \geq 0}$, since  the heat kernel forces the higher modes of the convolution to decay exponentially fast.  
In brief, for any $k$ and $\varepsilon$ as in the statement, we can find two constants $c_k$ and $c_{k,\varepsilon}$ such that,  
\begin{equation}
\label{eq:integral:RS:bound:1}
\begin{split}
\forall m \in {\mathbb N}, \ \forall r \geq 0, \quad 
\lvert \langle e^{\varepsilon\Delta} z_r,e_m\rangle \rvert  \leq  c_k\bigl\lvert \bigl\langle D^{2k}e^{\varepsilon\Delta}z_r,\frac{1}{m^{2k}} e_m\bigr\rangle \bigr\rvert \leq c_{k,\varepsilon} \frac{1}{m^{2k}} \lVert z_r\rVert_2 .
\end{split}
\end{equation}
Obviously, the proof of 
Lemma 
\ref{lem:integral:RS:bound:0}
relies on the 
 bound 
\eqref{eq:integral:RS:bound:1}, whence appears the constant $c_{k,\varepsilon}$ in the statement.  

\begin{proof}[Proof of Lemma \ref{lem:integral:RS:bound:0}.]
{
We begin with the following simple observation. It is easy to verify that each $e_m^{\pm}$ belongs to $H^2_{\rm sym}({\mathbb S})$ with $\| e_m^{\pm} \|_{2,2} \leq c (m^2 \vee 1)$.}
 
{By ({\bf E2}),  one can use standard properties of the Riemann-Stieltjes integral:}
\begin{equation}
\begin{split}
\sup_{t \in [0,T]} 
\biggl\vert 
\int_0^t \langle e^{\varepsilon \Delta} z_r,e_m \rangle
d \langle n_r,e_m^{\pm} \rangle \biggr\vert 
&\leq \sup_{t \in [0,T]} \bigl\vert  \langle e^{\varepsilon \Delta} z_t,e_m \rangle
\bigr\vert \times 
\langle n_T,
e_m^{\pm}
\rangle
\\
&\leq 
c_{k,\varepsilon} \frac{m^2 \vee 1}{m^{2k} \vee 1}
\| n_T \|_{2,-2}
 \sup_{t \in [0,T]} \lVert z_t\rVert_2, 
\end{split}
\end{equation}
with the last line following from 
\eqref{eq:integral:RS:bound:1}
together with the bound 
$\| e_m^{\pm} \|_{2,2} \leq c (m^2 \vee 1)$.
This shows
\eqref{eq:integral:RS:bound:0}.
As for the proof of
\eqref{eq:integral:RS:bound:0:++}, 
we just make use of  
\eqref{eq:rem:weak:limit:1}. 
\qed
\end{proof}

Lemma 
\ref{lem:integral:RS:bound:0}
allows us to {make the following definition}:
\begin{definition}
\label{def:integral:RS}
For two
(deterministic) 
curves 
$(n_t)_{t \geq 0}$ 
and 
$(z_t)_{t \geq 0}$, 
with 
$(n_t)_{t \geq 0}$ satisfying 
{\rm ({\bf E1})} and 
{\rm ({\bf E2})} and with $(z_t)_{t \geq 0}$ 
{being a continuous function from 
$[0,\infty)$ to  $L^2_{\rm sym}({\mathbb S})$}, we can define, almost surely, for 
any $\varepsilon>0$  
the integral process
{
$
( \int_0^t   e^{\varepsilon \Delta} z_r \cdot d n_r )_{t \geq 0}
$
}
 as the limit, for the uniform topology on compact subsets:
\begin{equation*}
\int_0^t   e^{\varepsilon \Delta} z_r \cdot d n_r
:=
\lim_{M \rightarrow \infty} 
\sum_{m=0}^M 
\biggl( 
\int_0^t   \bigl\langle e^{\varepsilon \Delta} z_r,e_m \bigr\rangle  d \langle n_r,e_m^+ \rangle
- 
\int_0^t   \bigl\langle e^{\varepsilon \Delta} z_r,e_m \bigr\rangle  d \langle n_r,e_m^- \rangle
\biggr). 
\end{equation*}
It satisfies
\begin{equation}
\label{eq:integral:RS:bound:0:bis}
\forall T \geq 0,
 \quad 
\sup_{t \in [0,T]} 
\biggl\vert 
\int_0^t  e^{\varepsilon \Delta} z_r
\cdot
d   n_r  
  \biggr\vert 
\leq {c_{\varepsilon}} \| n_T \|_{2,-2} \times \sup_{t \in [0,T]} \| z_t \|_2.
\end{equation}
When $(n_t)_{t \geq 0}$ is understood as a realisation of 
$(\eta_t)_{t \geq 0}$, the term 
$\| n_T \|_{2,-2}$ 
becomes $\| \eta_T \|_{2,-2}$ and can be upper bounded as in  
\eqref{eq:integral:RS:bound:0:++}.
\end{definition}

\begin{remark}
\label{rem:RS}
The following {two} remarks are in order: 
\begin{enumerate}
\item In Definition \ref{def:integral:RS}, not only is the convergence uniform in time in a fixed segment $[0,T]$, for some $T>0$, but it is also uniform
with respect to $(z_t)_{0 \leq t \leq T}$ 
when the latter is required to satisfy $\sup_{t \in [0,T]} \| z_t \|_2 \leq A$ for some 
given $A>0$. This is a direct consequence of the form of the rate of convergence given by \eqref{eq:integral:RS:bound:0}.
\item  
{
Lemma \ref{lem:integral:RS:bound:0}
and 
Definition
\ref{def:integral:RS}
extend 
to the case when 
$(z_t)_{t \geq 0}$ is piecewise constant (i.e., there exists  an increasing locally-finite 
sequence of time indices $(t_k)_{k \geq 0}$, with $t_0=0$, such that 
$t \in [t_k,t_{k+1}) \mapsto z_t  \in 
L^2_{\rm sym}({\mathbb S})$ is constant for each $k \geq 0$)}. 
\end{enumerate}
\end{remark}

The following lemma explains
the interest of the second remark right above.

\begin{lemma}
\label{lem:RS:3}
Within the same framework as in Definition 
\ref{def:integral:RS}
but with $(z_t)_{t \geq 0}$ therein being piecewise constant
(with the same jumping times $(t_k)_{k \geq 0}$ as in Remark 
\ref{rem:RS}),
{$(\int_0^t   e^{\varepsilon \Delta} z_r \cdot d n_r)_{t \geq 0}$ coincides with the 
process defined by standard Riemann sums, i.e.,}
\begin{equation*}
\int_0^t   e^{\varepsilon \Delta} z_r \cdot d n_r
= \sum_{k \geq 0 : t_k \leq t} 
\bigl\langle 
e^{\varepsilon \Delta} z_{t_k} , 
\eta_{t_{k+1} \wedge t} - 
\eta_{t_k} 
\rangle, \quad t \geq 0.
\end{equation*}
In particular, if $z_{t_k}$, for each $k \geq 0$, is symmetric non-increasing, then
\begin{equation*}
\forall t \geq 0, \quad 
\int_0^t   e^{\varepsilon \Delta} z_r \cdot d \eta_r
\geq 0.
\end{equation*}
\end{lemma}  

Before we prove 
Lemma
\ref{lem:RS:3}, we state the following important corollary.

\begin{corollary}
\label{cor:RS:3}
Within the same framework as in Definition 
\ref{def:integral:RS}, with $(z_t)_{t \geq 0}$ therein being continuous, 
we let, for any 
$k \in {\mathbb N}$, 
$z^k=(z_t^k)_{t \geq 0}$ be the piecewise constant approximation of 
$z=(z_t)_{t \geq 0}$ of stepsize $1/k$, namely
$
z_t^k := z_{ \lfloor k t \rfloor/k},$ for $ t \geq 0.
$
Then, the following convergence holds true, uniformly on compact subsets, 
\begin{equation*}
\int_0^t   e^{\varepsilon \Delta} z_r \cdot d n_r
= 
\lim_{k \rightarrow \infty}
\int_0^t   e^{\varepsilon \Delta} z_r^k \cdot d n_r 
, \quad t \geq 0.
\end{equation*}
{In particular, if $z_{t}$ is in $U^2({\mathbb S})$ for each $t \geq 0$, then, 
for all $t \geq 0$, 
$\int_0^t   e^{\varepsilon \Delta} z_r \cdot d n_r
\geq 0$.}
\end{corollary}

\begin{proof}[of Lemma
\ref{lem:RS:3}]
{Back to Definition 
\ref{def:integral:RS} - but for a path of the type discussed in the second item of Remark \ref{rem:RS} - we then observe that, for any integer $M \geq 1$,} 
\begin{equation*}
\begin{split}
&\sum_{m=0}^M 
\biggl( 
\int_0^t   \bigl\langle e^{\varepsilon \Delta} z_r,e_m \bigr\rangle  d \langle n_r,e_m^+ \rangle
- 
\int_0^t   \bigl\langle e^{\varepsilon \Delta} z_r,e_m \bigr\rangle  d \langle n_r,e_m^- \rangle
\biggr)
\\
&= \sum_{m=0}^M  \sum_{k \geq 0: t_k \leq t} 
   \bigl\langle e^{\varepsilon \Delta} z_{t_k},e_m \bigr\rangle  \Bigl[ 
   \Bigl( \langle n_{t_{k+1} \wedge t},e_m^+ \rangle
   -
   \langle n_{t_k},e_m^+ \rangle \Bigr) 
   - \Bigl( \langle n_{t_{k+1} \wedge t},e_m^- \rangle
   -
   \langle \eta_{t_k},e_m^- \rangle
   \Bigr)
   \Bigr]
\\
&= \sum_{m=0}^M  \sum_{k \geq 0: t_k \leq t}  
   \bigl\langle e^{\varepsilon \Delta} z_{t_k},e_m \bigr\rangle   
   \Bigl( \langle n_{t_{k+1} \wedge t},e_m \rangle
   -
   \langle n_{t_k},e_m \rangle \Bigr).
\end{split}
\end{equation*}
{Calling $I_M(t)$ the sum in the first line
and 
exchanging the  sums in the last line}, we get
\begin{equation*}
\begin{split}
&I_M(t) = \sum_{k \geq 0: t_k \leq t} \sum_{m=0}^M 
\biggl( 
   \bigl\langle e^{\varepsilon \Delta} z_{t_k},e_m \bigr\rangle   
   \Bigl( \langle n_{t_{k+1} \wedge t},e_m \rangle
   -
   \langle n_{t_k},e_m \rangle \Bigr)
\biggr).
\end{split}
\end{equation*}
Since, for each $k \in {\mathbb N}$, 
$n_{t_k}$ belongs to $H^{-2}_{\rm sym}({\mathbb S})$ 
and 
$e^{\varepsilon \Delta} z_{t_k}$ to $H^2_{\rm sym}({\mathbb S})$, we have 
\begin{equation*}
\forall k \in {\mathbb N}, 
\quad 
\lim_{M \rightarrow \infty} 
\sum_{m=0}^M 
\biggl( 
   \bigl\langle e^{\varepsilon \Delta} z_{t_k},e_m \bigr\rangle   
   \Bigl( \langle n_{t_{k+1} \wedge t},e_m \rangle
   -
   \langle \eta_{t_k},e_m \rangle \Bigr)
\biggr) =  \bigl\langle e^{\varepsilon \Delta} z_{t_k},  n_{t_{k+1} \wedge t}
   - n_{t_k} \bigr\rangle,
\end{equation*}
from which we easily deduce that 
\begin{equation*}
\begin{split}
\int_0^t   e^{\varepsilon \Delta} z_r \cdot d {n_r}
&:=
\lim_{M \rightarrow \infty} 
{I_M(t)}
= 
\sum_{k : t_k \leq t} \bigl\langle e^{\varepsilon \Delta} z_{t_k},  n_{t_{k+1} \wedge t}
   - n_{t_k} \bigr\rangle,
\end{split}
\end{equation*}
the convergence being uniform with respect to $t$ in compact subsets. 

{Whenever $z_{t_k}$ is in $U^2({\mathbb S})$,} 
so is 
$e^{\varepsilon \Delta} z_{t_k}$, see Lemma 
\ref{lem:heatpreservessymmetry}.
By the second item in Proposition 
\ref{prop:weak:limit:1}, we then obtain that 
$\langle e^{\varepsilon \Delta} z_{t_k},  n_{t_{k+1} \wedge t}
   - n_{t_k} \rangle \geq 0$. 
   \qed
\end{proof}

It remains to check 
Corollary \ref{cor:RS:3}.

\begin{proof}[Proof of Corollary \ref{cor:RS:3}]
The first claim in the statement of 
Corollary \ref{cor:RS:3}
is a consequence of \eqref{eq:integral:RS:bound:0:bis}, using the linearity of the integral,
 which says that 
\begin{equation*} 
\int_0^t e^{\varepsilon \Delta} z_r \cdot d n_r
- 
\int_0^t e^{\varepsilon \Delta} z^k_r \cdot d n_r
=
\int_0^t 
e^{\varepsilon \Delta} \bigl( z_r - z_r^k ) \cdot d n_r, \quad t \in [0,T],
\end{equation*}
together 
{with the fact 
that 
$\lim_{k \rightarrow \infty} \sup_{0 \leq t \leq T} \| z_t - z_t^k \|_2 = 0$.}
As for the second claim, it follows from Lemma 
\ref{lem:RS:3}. 
\qed
\end{proof}

\begin{remark}
\label{rem:RS:2epsilon}
{Notice that (with the presence of the factor $2$ in the exponential below)}
\begin{equation*}
\begin{split}
\int_0^t   e^{2 \varepsilon \Delta} z_r \cdot d n_r
=
\int_0^t   e^{\varepsilon \Delta} z_r \cdot d \bigl( e^{\varepsilon \Delta} n_r
\bigr)
&=
\lim_{M \rightarrow \infty} 
\sum_{m=0}^M 
\int_0^t   \bigl\langle e^{\varepsilon \Delta} z_r,e_m \bigr\rangle  d \langle e^{\varepsilon \Delta}  \eta_r,e_m \rangle
\\
&=
\lim_{M \rightarrow \infty} 
\sum_{m=0}^M 
\int_0^t   \bigl\langle e^{\varepsilon \Delta} z_r,e_m \bigr\rangle  d \langle \eta_r,e^{\varepsilon \Delta} e_m \rangle,
\end{split}
\end{equation*}
where it must be stressed that $(e^{\varepsilon \Delta} n_t)_{t \geq 0}$ on the first line satisfies 
({\bf E1}) and ({\bf E2}). While ({\bf E1}) follows from the contractive properties of the heat semigroup, 
({\bf E2})
follows from Lemma 
\ref{lem:heatpreservessymmetry} {(in words, $e^{\varepsilon \Delta} u \in U^2({\mathbb S})$ if $u \in U^2({\mathbb S})$)}. 

A proof of the above identity is as follows. 
By Corollary 
\ref{cor:RS:3} (and with the same notation), we can write
{
	$
\int_0^t   e^{2 \varepsilon \Delta} z_r \cdot d n_r
=
\lim_{k \rightarrow \infty}
\int_0^t   e^{2 \varepsilon \Delta} z_r^k \cdot d n_r.
$
}
Then, Lemma 
\ref{lem:RS:3} allows one to write 
the right-hand side as a Riemann sum. The proof is completed by expanding the terms in the Riemann sum in Fourier coefficients, exactly 
as in the proof of Lemma \ref{lem:RS:3}.
\end{remark}
\begin{remark}
\label{rem:integral:pathwise}
Definition 
\ref{def:integral:RS}
supplies us with the integral 
{
$
( \int_0^t e^{\varepsilon \Delta} z_r \cdot d n_r )_{t \geq 0}, 
$
}
when $(z_t)_{t \geq 0}$ is a deterministic continuous path with values in $L^2_{\rm sym}({\mathbb S})$
and $(n_t)_{t \geq 0}$ satisfies 
({\bf E1}) and ({\bf E2}). Importantly, one can replace $(z_t)_{t \geq 0}$ by the realisation of a (stochastic) 
continuous process $(Z_t)_{t \geq 0}$ 
with values in $L^2_{\rm sym}({\mathbb S})$ and 
$(n_t)_{t \geq 0}$ by 
 the same process $(\eta_t)_{t \geq 0}$ as in Proposition 
\ref{prop:weak:limit:1}.
The integral
is denoted 
{
$
( \int_0^t e^{\varepsilon \Delta} Z_r \cdot d \eta_r )_{t \geq 0}.
$
}
%
It is continuous in time. When $(Z_t)_{t \geq 0}$ is adapted to the filtration 
${\mathbb F}$ used in 
the statement of
Proposition 
\ref{prop:weak:limit:1}, the integral is also 
adapted to ${\mathbb F}$, due to  
Lemma
\ref{lem:RS:3}
and
Corollary
\ref{cor:RS:3}.
\end{remark}

\subsection{Orthogonality of the reflection}
\label{subse:3:3}
We now come to the last property in the description of the weak limits:

\begin{proposition}
\label{prop:reflection:ortho}
Let $(X_t,W_t)_{t \geq 0}$ be a weak limit
of the processes $\{ (\tilde X^h_t,W_t)_{t \geq 0}\}_{h >0}$ as $h$ tends to $0$, 
as given by Proposition \ref{prop:weak:limit:1}. 
Then, 
for any $t \geq s \geq 0$, 
\begin{equation}
	\label{eq zeroY}
	\lim_{\varepsilon\searrow0}\mathbb{E}\left[ \int_s^t  e^{\varepsilon\Delta}X_r   \cdot d  \eta_r \right]= 0.
\end{equation}
\end{proposition}

To appreciate the scope of the above statement, the reader should recall that 
$(X_t)_{t \geq 0}$ takes values in $U^2({\mathbb S})$. Therefore, 
Corollary 
\ref{cor:RS:3} yields, almost surely, 
{
	$
\int_s^t  e^{\varepsilon\Delta}X_r   \cdot d  \eta_r  \geq 0,
$
} 
for any $t \geq s \geq 0$. 
In particular, Fatou's lemma (together with the time continuity of the integral) implies that, 
with probability 1, for any $t \geq s \geq 0$, 
\begin{equation*}
	\liminf_{\varepsilon\searrow0}  \int_s^t  e^{\varepsilon\Delta}X_r   \cdot d  \eta_r = 0.
\end{equation*}
We regard this property as a (weak) form of orthogonality between $X_r$ and $d \eta_r$, recalling that the orthogonality property is standard in reflected equations (see for instance the seminal work \cite{Lions:Sznitman}). Moreover,  we obtain the following corollary
as an important by-product of the proof {of Proposition \ref{prop:reflection:ortho}.}

\begin{corollary}
\label{cor:reflection:ortho}
Let $(X_t,W_t)_{t \geq 0}$ be a weak limit of the processes $\{ (\tilde X^h_t,W_t)_{t \geq 0}\}_{h >0}$ as $h$ tends to $0$, 
as given by Proposition \ref{prop:weak:limit:1}. 
Then, 
 {with probability 1, 
 the process $(X_t)_{0 \leq t \leq T}$ 
has time-continuous trajectories with respect to 
 $\| \cdot \|_2$ and, for any $T>0$, it
 takes
 values in $L^2([0,T],H^1_{\rm sym}({\mathbb S}))$ with
 ${\mathbb E} \int_0^T \| D X_t \|^2_2 dt < \infty.$
 With probability 1, the process $(\eta_t)_{0 \le t \le T}$ has time-continuous trajectories with respect to $\| \cdot \|_{2,-2}$. }
 \end{corollary}

\begin{proof}[of Proposition \ref{prop:reflection:ortho} and Corollary \ref{cor:reflection:ortho}]
{ \ }
\vspace{5pt}

\noindent \textit{First step.} 
We first prove 
that 
$(X_t)_{0 \leq t \leq T}$ 
 takes
 values in $L^2([0,T],H^1_{\rm sym}({\mathbb S}))$
{(which is not a direct corollary of Proposition 
\ref{prop:reflection:ortho}, but which {comes as} a consequence of the global architecture of the proof)}.
 In order to do so, 
we return to the scheme 	\eqref{eq scheme n}. 
For a given $h \in (0,1]$ 
and for any integer $n \geq 0$,  
\begin{equation}\label{eq zeroYschemeRewrite}
\begin{split}
	{\E \bigl[\lVert X_{n+1}^h \rVert^2_2\bigr] \leq     \E \bigl[\lVert e^{h\Delta} X_{n}^h \rVert^2_2\bigr] + \E \left[\left\lVert  \int_0^{h} e^{(h-s)\Delta}dW^{n+1}_s\right\rVert^2_2\right]}.\\
\end{split}
\end{equation}
Then, using the fact that $(e^{t\Delta}X^h_n)_{0 \leq t \leq h}$ solves the heat equation, one has the equality
\begin{equation}
\label{eq heatestimate} 
\E \biggl[\lVert e^{h\Delta}X^h_n\rVert_2^2 + 2 \int_0^h \lVert De^{s\Delta} X^h_n\rVert_2^2ds \biggr]= \E \bigl[\lVert X^h_n \rVert_2^2 \bigr]. \\ 
\end{equation}
Considering $h\leq \varepsilon$ {for some $\varepsilon \in (0,1)$}, we have $\E[ \lVert e^{\varepsilon\Delta}DX^h_n\rVert_2^2  ] \leq \E[ \lVert De^{s\Delta} X^h_n\rVert_2^2 ]$
{for $s \in (0,h]$}.
{By \eqref{eq zeroYschemeRewrite}, \eqref{eq heatestimate}, and Lemma \ref{lem unif est noise}
(for the definition of $c_{1,\lambda}$)}, we get  %
\begin{equation*}%
	\begin{split}
		\E \bigl[\lVert X_{n+1}^h \rVert^2_2\bigr] - \E \bigl[\lVert X^h_n \rVert_2^2 \bigr]     + 2 h \E\left[    \lVert De^{\varepsilon\Delta} X^h_n\rVert_2^2 \right] \leq &   
		c_{1,\lambda} h .\\
	\end{split}
\end{equation*}
And then, for $t \geq s \geq 0$ and {$N_t:= \lfloor t/h \rfloor$ and 
$N_s:= \lfloor s/h \rfloor$}, we have 
\begin{equation*}
	\E \bigl[\lVert X_{N_t}^h \rVert^2_2\bigr] - \E \bigl[\lVert X^h_{N_s} \rVert_2^2 \bigr]     +2  h 
	\sum_{n=N_s}^{N_t-1} \E\left[    \lVert De^{\varepsilon\Delta} X^h_n\rVert_2^2 \right] \leq    c_{1,\lambda} h (N_t- N_s).
	\end{equation*} 
	Choosing $s=0$, lower bounding 
	$\| X_{N_t}^h \|_2^2$ by 
	$\| e^{\varepsilon \Delta} X^h_{N_t} \|^2_2$, 
	recalling the notation \eqref{eq interpolation} and combining tightness of the family 
$\{ (\tilde X^h_t)_{t \geq 0}\}_{h \in (0,1]}$ in $\cC([0,\infty),H_{\rm sym}^{-1}({\mathbb S}))$ with Corollary \ref{cor:2.5}
(which supplies us with uniform integrability properties),
we obtain
\begin{equation}
\label{eq zeroYschemeRewrite2}
	\E \bigl[\lVert e^{\varepsilon \Delta} \tilde X_{t}^h \rVert^2_2\bigr]  
	-
	\E \bigl[\lVert   X_{0} \rVert^2_2\bigr]    +  2
	\int_0^t \E\left[    \lVert De^{\varepsilon\Delta} \tilde X^h_r \rVert_2^2 \right] dr \leq    c_{1,\lambda}  t + \oh_h(1),
	\end{equation}
	with $\lim_{h \searrow 0} \oh_h(1)=0$ (the rate possibly depending on $\varepsilon$). 
Noticing that the function $z \mapsto D [e^{\varepsilon \Delta}  z]$ is continuous from $H^{-1}_{\rm sym}({\mathbb S})$ into $L^2_{\rm sym}({\mathbb S})$, we can easily 
take some weak limit as in the statement (as $h$ tends to $0$). We get 
\begin{equation}\label{eq zeroYschemeRewrite3}
	\begin{split}
		\E \bigl[\lVert  e^{\varepsilon \Delta} X_t \rVert^2_2\bigr] - \E \bigl[\lVert X_0 \rVert_2^2 \bigr]     +  2\int_0^t \E\left[    \lVert De^{\varepsilon\Delta} X_r \rVert_2^2 \right] dr \leq & c_{1,\lambda} t. 
	\end{split}
\end{equation}
Since $c_{1,\lambda}$ is independent of $\varepsilon$, this establishes   ${\mathbb E} \int_0^T \| D X_t \|^2_2 dt < \infty$. 
 \vskip 5pt
 
\noindent  \textit{Second step.}
Next, return to equation \eqref{eq RSHE}, with $u \in H^2_{\rm sym}({\mathbb S})$ replaced by $e^{\varepsilon\Delta}u$,
\begin{equation}\label{eq zeroYsecondPart2}
 	\begin{split}
 		\langle X_t-{X}_s ,e^{\varepsilon\Delta}u \rangle &  - \int_s^t \langle  {X}_s  ,\Delta e^{\varepsilon\Delta}u \rangle ds   =  \langle W_t-W_s  ,e^{\varepsilon\Delta}u \rangle +  \langle \eta_t-\eta_s, e^{\varepsilon\Delta}u \rangle.  
 	\end{split} 
 \end{equation}
The next step is to choose $u=e_m$ and then to apply Itô's formula
in order to expand $(\langle X_t,e^{\varepsilon\Delta} e_m \rangle^2)_{t \geq 0}$. 
To do so, it is worth recalling from 
\eqref{eq fouriersplit2}
 that 
{since $e_m=e^+_m-e^-_m$ with $e_m^{\pm} \in U^2({\mathbb S})$}, the process 
$(\langle \eta_t , e^{\varepsilon\Delta}e_m \rangle)_{t \geq 0}$ may be written as the difference of two non-decreasing processes and consequently has finite variation. 
Therefore, due to Itô's formula,  
 \begin{equation}\label{eq zeroYsecondPartIto2}
	\begin{split}
		d\langle X_t ,e^{\varepsilon\Delta}e_m \rangle^2 = &2  \langle e^{\varepsilon\Delta}X_t ,e_m \rangle\langle \Delta e^{\varepsilon\Delta}  {X}_t  ,  e_m \rangle dt + 2 \langle X_t ,e^{\varepsilon\Delta}e_m \rangle d  \langle \eta_t , e^{\varepsilon\Delta}e_m \rangle
		\\
		&\hspace{15pt} 
		 +  2 \langle X_t ,e^{\varepsilon\Delta}e_m \rangle d  \langle W_t  ,e^{\varepsilon\Delta}e_m \rangle
		 +    d \left[     \langle W_\cdot  ,e^{\varepsilon\Delta}e_m \rangle  \right]_t,
	\end{split} 
\end{equation}
where, as before, the symbol $[ \cdot ]_t$ is used to denote the bracket. Integrating over $[0,t]$, applying expectation, summing over $m \in {\mathbb N}_0$ and then integrating by parts,
\begin{equation}
   	\label{eq ineqEps1}
\begin{split}
		& \E \bigl[ \lVert e^{\varepsilon\Delta}X_t   \rVert_2^2 \bigr] + 2  \int_0^t \E\bigl[ \lVert De^{\varepsilon\Delta}X_r  \rVert_2^2 \bigr] dr  \\
		&=  \E \bigl[ \lVert e^{\varepsilon\Delta}X_0   \rVert_2^2 \bigr]
		 +2\E \biggl[ \sum_{m\in\bN_0} \int_0^t \langle e^{\varepsilon\Delta}X_r ,e_m \rangle d \langle \eta_r , e^{\varepsilon\Delta}e_m \rangle  \biggr]  
		   +  \E  \bigl[ \lVert e^{\varepsilon\Delta}W_t    \rVert_2^2 \bigr]
		   \\
		   &= {  \E \bigl[ \lVert e^{\varepsilon\Delta}X_0   \rVert_2^2 \bigr]   +  \E  \bigl[ \lVert e^{\varepsilon\Delta}W_t    \rVert_2^2 \bigr]   +2\E \left[ \int_0^t   e^{2\varepsilon\Delta}X_r \cdot  d \eta_r    \right],}  \\  
  	\end{split}
  \end{equation}  
  where we used Definition \ref{def:integral:RS} and Remark \ref{rem:RS:2epsilon} to get the last line. Combining with the inequality \eqref{eq zeroYschemeRewrite3}, 
recalling 
\eqref{eq:lem unif est noise:1}
 and passing to the limit as $\varepsilon\rightarrow 0$, this implies that 
\begin{equation*}
	\begin{split}
	\lim_{\varepsilon\searrow0} \E \left[ \int_0^t   e^{2\varepsilon\Delta}X_r \cdot  d \eta_r    \right] &=
	  	 	\E \bigl[ \lVert  X_t   \rVert_2^2 \bigr]+   2  \int_0^t \E\bigl[ \lVert D X_r  \rVert_2^2\bigr] dr -  \E \bigl[ \lVert  X_0   \rVert_2^2 \bigr]   -  \E  \bigl[ \lVert  W_t    \rVert_2^2 \bigr]
	\\
	&\leq c_{1,\lambda} t - \E  \bigl[ \lVert  W_t    \rVert_2^2 \bigr]
=0	,
	\end{split}
\end{equation*}  
which completes the proof of Proposition \ref{prop:reflection:ortho} (recall that 
$ \int_s^t  e^{\varepsilon\Delta}X_r   \cdot d  \eta_r \leq 
 \int_0^t  e^{\varepsilon\Delta}X_r   \cdot d  \eta_r$ for
 $s \in [0,t]$).  
 \vskip 5pt 

\noindent 
 \textit{Third step.}
 We now prove the first part of 
  Corollary \ref{cor:reflection:ortho}
  (time-continuity of the trajectories of $(X_t)_{t \geq 0}$ with respect to $\| \cdot \|_2$). 
  To do so, we come back to 
  \eqref{eq ineqEps1}, but without expectation. We have
  \begin{equation}
\label{eq zeroYsecondPartIto322}
\begin{split}
		&  \lVert e^{\varepsilon\Delta}X_t   \rVert_2^2   + 2  \int_0^t  \lVert De^{\varepsilon\Delta}X_r  \rVert_2^2  dr  
		\\
		&=    \lVert e^{\varepsilon\Delta}X_0   \rVert_2^2  
		 + \sum_{m \in {\mathbb N}_0} \lambda_m^2 e^{- 8 \pi^2 m^2 \varepsilon } t + 2
		 \int_0^t  e^{2\varepsilon\Delta}X_r \cdot d W_r  +2  \int_0^t   e^{2\varepsilon\Delta}X_r \cdot  d \eta_r. 
			\end{split}
  \end{equation}  
Fix $T>0$ as in the first step and assume that $t \in [0,T]$. Writing the same identity as above but at another time 
$s \in [0,T]$, for $s  \leq t $, we obtain 
  \begin{equation*}
\begin{split}
\label{eq zeroYsecondPartIto318}
		  \Bigl\vert \lVert e^{\varepsilon\Delta}X_t   \rVert_2^2  
		- \lVert e^{\varepsilon\Delta}X_s   \rVert_2^2    
		\Bigr\vert
		&\leq  2  \int_s^t  \lVert D X_r  \rVert_2^2  dr  
		 + \sum_{m \in {\mathbb N}_0} \lambda_m^2 e^{-8 \pi^2 m^2 \varepsilon } (t-s)
		\\
		&\hspace{15pt}  + 2
		\biggl\vert  \int_s^t  e^{2\varepsilon\Delta}X_r \cdot d W_r  \biggr\vert +2  \int_s^t   e^{2\varepsilon\Delta}X_r \cdot  d \eta_r. 
			\end{split}
  \end{equation*}  
Upper bounding $\int_s^t e^{2 \varepsilon \Delta} X_r \cdot d \eta_r$ by 
$\int_0^T e^{2 \varepsilon \Delta} X_r \cdot d \eta_r$, we observe from the second step that 
$\sup_{0 \leq s \leq t \leq T} \vert \int_s^t e^{2 \varepsilon \Delta} X_r \cdot  d \eta_r \vert$ tends to $0$ in probability. 
Similarly, we deduce from Doob's inequality that 
\begin{equation*} 
\sup_{0 \leq s \leq t \leq T} \biggl\vert 
\int_s^t  e^{2\varepsilon\Delta}X_r \cdot d W_r
-
\int_s^t   X_r \cdot d W_r
\biggr\vert
\end{equation*}
tends to $0$ in probability. Therefore, by extracting a subsequence $(\varepsilon_n)_{n \geq 0}$ that tends to $0$, 
we can easily pass to the limit in the right-hand 
side of 
\eqref{eq zeroYsecondPartIto318}, with probability 1, for all $0 \leq s \leq t \leq T$. 
As for the left-hand side, we can consider an event of probability 1, on which $\sup_{0 \le r \le T} 
\| X_r \|_2 < \infty$ (courtesy of Proposition 
\ref{prop:weak:limit:2}). On this event, we can pass to the limit for all $s,t \in [0,T]$ in the left-hand side of 
\eqref{eq zeroYsecondPartIto318}. 

We deduce that, ${\mathbb P}$-almost surely, for all $s,t \in [0,T]$ with $s \leq t$, 
\begin{equation*} 
  \Bigl\vert \lVert  X_t   \rVert_2^2  
		- \lVert  X_s   \rVert_2^2    
		\Bigr\vert
		\leq  2  \int_s^t  \lVert D X_r  \rVert_2^2  dr  
		 +  c_{1,\lambda} (t-s)
		   + 2
		\biggl\vert  \int_s^t   X_r \cdot d W_r  \biggr\vert. 
\end{equation*}
By the second step, with probability 1, the right-hand side
tends to $0$ as $t-s$ tends to $0$. This proves that there exists an event of probability 1 on 
which the trajectory $t \in [0,T] \mapsto \| X_t \|_2^2$ is continuous. 
Recalling that the trajectory $t \in [0,T] \mapsto X_t$ is already known to be continuous with 
respect to $\| \cdot \|_{2,-1}$, we easily deduce that it is continuous 
with respect to $\| \cdot \|_2$. 
By 
\eqref{eq RSHE}, we deduce that the trajectory $t \in [0,T] \mapsto \eta_t$ is continuous with
respect to $\| \cdot \|_{2,-2}$. 
\qed
\end{proof}

\subsection{Definition and uniqueness of solutions to the rearranged SHE}
\label{subse:3:4} 
 
{We now define a solution to the rearranged SHE studied in this paper.}

\begin{definition}
\label{def:existence}
On a given (filtered) probability space $(\Omega,{\mathcal A},{\mathbb F},{\mathbb P})$ equipped with 
a $Q$-Brownian motion $(W_t)_{t \geq 0}$ with values in $L^2_{\rm sym}({\mathbb S})$
(with respect to the filtration ${\mathbb F}$) and with 
an ${\mathcal F}_0$-measurable initial condition $X_0$ with values in $
U^2({\mathbb S})$ (see 
\eqref{def:1.2:U2}) and
 with finite moments 
of any order (see \eqref{eq:assumption:X0}), we say that a pair of processes  
$(X_t,\eta_t)_{t \geq 0}$ solves the 
rearranged SHE 
\eqref{eq reflSHE}
driven by 
$(W_t)_{t \geq 0}$ and $X_0$ if 
\begin{enumerate}
\item $(X_t)_{t \geq 0}$ is a continuous ${\mathbb F}$-adapted process with values in $
U^2({\mathbb S})$;
\item $(\eta_t)_{t \geq 0}$ is a continuous ${\mathbb F}$-adapted process with values in $H^{-2}_{\rm sym}({\mathbb S})$, starting from $0$ at $0$, such that, with probability 1, for any $u \in H^2_{\rm sym}({\mathbb S})$ that is non-increasing, the path
$(\langle \eta_t,u \rangle)_{t \geq 0}$ is non-decreasing;
\item with probability 1, for any $u \in H^{2}_{\rm sym}({\mathbb S})$, 
\begin{equation}\label{def RSHE 1}
	\begin{split}
		\forall t \geq 0,\quad 
		\langle  X_t,u \rangle    =     \int_0^t \langle   {X}_r  ,\Delta     u \rangle dr +\langle W_t  ,u \rangle + \langle \eta_t ,u\rangle.
\end{split} 
\end{equation}
\item for any $t \geq 0$, 
\begin{equation}
	\label{def RSHE 2}
	\lim_{\varepsilon\searrow0}\mathbb{E}\left[ \int_0^t  e^{\varepsilon\Delta}X_r   \cdot d  \eta_r \right]= 0.
\end{equation}
\end{enumerate}
\end{definition}
Of course, the definition of the integral in \eqref{def RSHE 2} is understood as in Definition 
\ref{def:integral:RS}.

{We now address pathwise uniqueness to the rearranged SHE.} 

\begin{proposition}
\label{prop:strong:!}
Given 
$(\Omega,{\mathcal A},{\mathbb F},{\mathbb P})$ equipped with 
a $Q$-Brownian motion $(W_t)_{t \geq 0}$ with values in $L^2_{\rm sym}({\mathbb S})$
(with respect to the filtration ${\mathbb F}$) and with 
an ${\mathcal F}_0$-measurable initial condition $X_0$ with values in $U^2({\mathbb S})$
and with finite moments of any order,
there exists at most one solution $(X_t,\eta_t)_{t \geq 0}$ to 
 \eqref{eq reflSHE} that satisfies 
Definition 
\ref{def:existence}.
\end{proposition}

\begin{proof}
 Consider two candidate solutions  $(X^1_t,\eta^1_t)_{t \geq 0}$ and $(X^2_t,\eta^2_t)_{t \geq 0}$. 
 Then, by  {Itô's formula}, 
 for any $\varepsilon >0$, 
 $\P$-a.s.,
 for any $m \in {\mathbb N}_0$, 
\begin{align}
\label{eq comparison1}
&d\langle e^{\varepsilon\Delta} (X^1_t-X_t^2), e_m \rangle^2 
 \\
&= 2\langle e^{\varepsilon\Delta} ( X_t^1-X_t^2), e_m \rangle \left[ \langle \Delta e^{\varepsilon\Delta} ( X^1_t-X^2_t), e_m \rangle dt + d \langle \eta^1_t-\eta^2_t, e^{\varepsilon\Delta} e_m \rangle  \right], \quad t \geq 0. \nonumber
\end{align}
{
Summing over $m \in {\mathbb N}_0$ and integrating by parts, one obtains
($\mathbb{P}$-a.s)
\begin{align}
\label{eq comparison2} 
&\lVert e^{\varepsilon\Delta}  (X_t^1-X_t^2) \rVert_2^2 
+ 2   \int_0^t \lVert De^{\varepsilon\Delta}( X^1_r  -X^2_r) \rVert_2^2 dr  = 2 \int_0^t e^{2\varepsilon\Delta}(X^1_r-  {X_r^2} ) \cdot d (\eta^1_r - \eta^2_r) \nonumber \\
&\hspace{15pt} \leq   2 \left( \int_0^t e^{2\varepsilon\Delta} X^1_r  \cdot d \eta^1_r +  \int_0^t e^{2\varepsilon\Delta}  X^2_r   \cdot d \eta^2_r \right), 
\end{align}
}
where we used 
Corollary 
\ref{cor:RS:3} to establish the last inequality. 

Applying expectation and 
setting $\varepsilon\rightarrow 0 $, we obtain 
from 
	\eqref{def RSHE 2} 
that $X_t^1=X_t^2$, $\bP-a.s.$ for any $t\geq 0$, which {shows that $X^1$ are $X^2$  are indistinguishable}. Since $\eta^1$ and $\eta^2$ are defined via  {$X^1$ and $X^2$}, there exists a unique solution to \eqref{eq reflSHE}.
\qed
\end{proof}

{By combining} Propositions 
\ref{prop:weak:limit:1}
and \ref{prop:strong:!} {with Corollary \ref{cor:reflection:ortho}},
we deduce from an obvious adaptation of Yamada-Watanabe argument, {the first main result  of the paper}:
\begin{theorem}
\label{thm:main:existence:uniqueness}
Given 
$(\Omega,{\mathcal A},{\mathbb F},{\mathbb P})$ equipped with 
a $Q$-Brownian motion $(W_t)_{t \geq 0}$ with values in $L^2_{\rm sym}({\mathbb S})$
(with respect to the filtration ${\mathbb F}$) and with 
an ${\mathcal F}_0$-measurable initial condition $X_0$ with values in $L^2_{\rm sym}({\mathbb S})$,
there exists a unique solution $(X_t,\eta_t)_{t \geq 0}$ to the rearranged SHE \eqref{eq reflSHE} that satisfies 
Definition 
\ref{def:existence}.

Moreover, the processes $\{ (\tilde X^h_t,W_t)_{t \geq 0}\}_{h >0}$, as defined in 
\eqref{eq scheme n}, 
are convergent
in law
(over 
$\cC([0,\infty),{H_{\rm sym}^{-1}({\mathbb S}) \times L_{\rm sym}^2({\mathbb S})})$)  equipped with the topology of uniform convergence on compact subsets) and the limit is 
the law of $(X_t,W_t)_{t \geq 0}$. 
\end{theorem}
Notice that, as in 
Proposition \ref{prop:strong:!}, {we use} the same noise 
for the scheme and for the limiting equation. 
However, in contrast to
Proposition \ref{prop:strong:!}, 
there is no abuse in doing so: the first part of the statement allows us to construct the solution to 
the rearranged equation on the same filtered probability space (equipped with the same noise) as the scheme. It shall prove useful to note that item 4 in 
Definition \ref{def:existence} may be strengthened:
\begin{proposition}
\label{prop:item:iv:higher:moments}
Let 
$(X_t,\eta_t)_{t \geq 0}$ satisfy 
 Definition \ref{def:existence}
 except item 4 therein. 
 Then, item 4  holds true if and only if one of the following two properties below 
 is satisfied: 
 \begin{itemize}
\item[4'] For any $t >0$, there exists a sequence of positive reals $(\varepsilon_q)_{q \geq 1}$, with $0$ as limit, 
such that, in ${\mathbb P}$-probability, 
$	\lim_{q \rightarrow \infty}  \int_0^t  e^{\varepsilon_q \Delta}X_r   \cdot d  \eta_r = 0$.
 \item[4''] It holds that, for any $p \geq 1$, for any $t>0$, $
	\lim_{\varepsilon\searrow0}\mathbb{E}[ ( \int_0^t  e^{\varepsilon\Delta}X_r   \cdot d  \eta_r )^p ]= 0.$
	\end{itemize} 
\end{proposition}

\begin{proof}
We proceed as follows. 
We consider a process
$(X_t,\eta_t)_{t \geq 0}$ satisfying 
 Definition \ref{def:existence}
 except item \textit{4} therein. Obviously \textit{4''} implies \textit{4}, which implies in turn \textit{4'}. The only difficulty is to prove that \textit{4'} implies 
 \textit{4''}. 

Assuming  
\textit{4'}, 
we recall that, by construction, the argument inside the power function in \textit{4''}  is non-negative. 
We then show that, for any $t \geq 0$, 
\begin{equation}
\label{eq:bound:suppl:item4''}
\forall p \geq 1, \quad 
\sup_{\varepsilon >0}
\mathbb{E}\left[ \biggl( \int_0^t  e^{\varepsilon\Delta}X_r   \cdot d  \eta_r \biggr)^p \right] < \infty. 
\end{equation}
We restart 
from 
\eqref{eq zeroYsecondPartIto2}
and we follow the derivation {of}   \eqref{eq ineqEps1}, but without taking the expectation therein. 
For a given $\varepsilon >0$, we 
restart from 
\eqref{eq zeroYsecondPartIto322}
(which holds true under items 1, 2 and 
3 in Definition 
\ref{def:existence}, even though item 4 is not known yet).
Recall from the contractivity of the heat semigroup 
that the function 
\begin{equation*}
\varepsilon \in (0,\infty) \mapsto  \int_0^t   \lVert De^{\varepsilon \Delta}X_r  \rVert_2^2 
 dr
 \end{equation*}
 is non-increasing (for any given realisation). 
{Choosing} $\varepsilon = \varepsilon_q/2$ for some $q \in {\mathbb N}_0$ in 
\eqref{eq zeroYsecondPartIto322}
and
{taking} the limit (in probability) as $q$ tends to $\infty$, we deduce 
from   \textit{4'} that
\begin{equation*}
\begin{split}
 \lVert X_t   \rVert_2^2 
 +    2\sup_{\varepsilon >0}  \int_0^t   \lVert e^{\varepsilon \Delta} D X_r  \rVert_2^2 
 dr =
  \lVert X_0   \rVert_2^2 
     +  \lVert W_t    \rVert_2^2  
+ 2 
 \int_0^t X_s \cdot d  W_s .
 \end{split}
\end{equation*}
{By taking power $p$ and recalling
	\eqref{eq:lem unif est noise:1} 
	and Proposition 
	\ref{prop:weak:limit:2}
together with the fact that $\| X_t \|_2$ has finite moments of any order for any $t$ (see  Proposition \ref{prop:tightness}), 
we  deduce  that}
\begin{equation}
\label{eq:gronwall:lp}
{\mathbb E} \biggl[ \biggl( \sup_{\varepsilon >0} \int_0^t   \lVert D e^{\varepsilon \Delta}  X_r  \rVert_2^2
\biggr)^p \biggr]
< \infty.
\end{equation}
{Back to \eqref{eq zeroYsecondPartIto322}, we can express 
$\int_0^t   e^{2\varepsilon\Delta}X_r \cdot  d \eta_r$
in terms of all the other terms. Using the Burkolder-Davis-Gundy inequality to handle the 
stochastic integral,   
\eqref{eq:bound:suppl:item4''}
 follows.} 

Now that we have 
\eqref{eq:bound:suppl:item4''}, 
it suffices to prove item \textit{4''} with $p=1$. 
The result for $p>1$ then follows from a standard uniform integrability argument. 
In fact, by combining
\eqref{eq:bound:suppl:item4''}
and 
 item \textit{4'}, we already know that, for any $t >0$, 
 \begin{equation}
 \label{eq:proof:4'':existence:lim:subseq}
 \lim_{q \rightarrow \infty} 
{\mathbb E} \left[ \int_0^t  e^{\varepsilon_q \Delta}X_r   \cdot d  \eta_r \right] = 0. 
 \end{equation}
 It remains to observe that the limit 
 $\lim_{\varepsilon \searrow 0} 
{\mathbb E}  \int_0^t  e^{\varepsilon  \Delta}X_r   \cdot d  \eta_r$
 exists necessarily. 
   Indeed, 
 taking expectation in 
 \eqref{eq zeroYsecondPartIto322}, we observe that
 \begin{equation*}
\begin{split}
 \lim_{\varepsilon \searrow 0} 
{\mathbb E}  \left[ \int_0^t  e^{\varepsilon  \Delta}X_r   \cdot d  \eta_r \right]
& =  {\mathbb E} 
 \bigl[ \lVert X_t   \rVert_2^2 \bigr]
 +     2 \sup_{\varepsilon >0} {\mathbb E} \left[  \int_0^t   \lVert De^{\varepsilon\Delta}X_r  \rVert_2^2 
 dr \right] 
 \\
&\hspace{15pt}  -
{\mathbb E} \bigl[  \lVert X_0   \rVert_2^2 \bigr]
     -
     {\mathbb E} \bigl[  \lVert  W_t    \rVert_2^2 \bigr],  
 \end{split}  
  \end{equation*}
  where we used once again the contractivity of the heat semigroup to write 
$  \lim_{\varepsilon \rightarrow 0} 
  {\mathbb E} \left[  \int_0^t   \lVert De^{\varepsilon\Delta}X_r  \rVert_2^2 
 dr\right]
=
  \sup_{\varepsilon >0} 
  {\mathbb E} \left[ \int_0^t   \lVert De^{\varepsilon\Delta}X_r  \rVert_2^2 
 dr\right]$. The proof is complete.  \qed
\end{proof}

\subsection{Lipschitz regularity of the flow}
\label{subse:3:5}
{We conclude this section with the following result, crucial for the rest of the paper. The proof is the same as that of uniqueness; we just need to retain the difference of the initial conditions in the argumentation.}

\begin{proposition}
\label{prop:flow} 
Given 
$(\Omega,{\mathcal A},{\mathbb F},{\mathbb P})$ equipped with 
a $Q$-Brownian motion $(W_t)_{t \geq 0}$ with values in $L^2_{\rm sym}({\mathbb S})$
(with respect to  ${\mathbb F}$),
consider $(X^x,\eta^x)$ and $(X^y,\eta^y)$ the solutions to 
the \eqref{eq reflSHE}  with $X_0^x=x \in U^2({\mathbb S})$ and $X_0^y=y \in U^2({\mathbb S})$ as initial conditions. 
Then,  
\begin{equation}
	\label{eq lip}
	\begin{split}
		&
		{\mathbb P}\textrm{\rm -a.s.},
	\quad
		 \lVert X^{x }_t-X^y_t \rVert_2^2 
		+
		2 \int_0^t \lVert De^{\varepsilon\Delta}( X^x_r-X^y_r) \rVert_2^2 dr
		\leq  \lVert  x-y \rVert_2^2,\quad t \geq 0.
	\end{split}
\end{equation} 
\end{proposition}
As an obvious (but very useful) consequence of 
Proposition 
\ref{prop:flow}, we have, for all $T >0$, 
\begin{equation}
\label{eq:reg:flow:eta}
\sup_{0 \leq t \leq T} 
\Bigl\| 
 \eta_t^x - \eta_t^y
 \Bigr\|_{2,-2} \leq c_T \| x - y \|_2,
\end{equation}
for a constant $c_T$ only depending on $T$. 
The proof follows from 
the identity 
\eqref{def RSHE 1}. 

\begin{remark}
\label{rem:flow:existence}
{Inequalities}
	\eqref{eq lip}
	and
	\eqref{eq:reg:flow:eta} say that, for each $x \in L^2_{\rm sym}({\mathbb S})$, we can find versions of $(X_t^x)_{t \geq 0}$ and $(\eta^x_t){_t \geq 0}$ such that, for any $T>0$, the mappings
	$(\omega,t,x) \in \Omega \times [0,T] \times L^2_{\rm sym}({\mathbb S}) \mapsto X_t^x(\omega) 
	\in L^2_{\rm sym}({\mathbb S})$ and 
	$(\omega,t,x) \in \Omega \times [0,T] \times L^2_{\rm sym}({\mathbb S}) \mapsto \eta_t^x(\omega) 
	\in H^{-2}_{\rm sym}({\mathbb S})$ 
are measurable, {continuous in $t$ and Lipschitz in $x$}. 
\end{remark}

\section{Smoothing Effect}
\label{sec:4}


It is demonstrated below that the semigroup -  
$$\{P_t\}_{t\geq 0}\quad \text{with}\quad P_tf(x):= \bE[f(X^x_t)], \quad {t \geq 0}, \ x \in U^2({\mathbb S}),$$
for $f$ within the set of bounded measurable functions on $U^2({\mathbb S})$ - maps 
bounded measurable functions into Lipschitz continuous functions on 
$U^2({\mathbb S})$, at least when the 
parameter $\lambda$ in \eqref{def:tilde:noise}  belongs to $(\nicefrac12,1)$. Importantly, we prove that the rate at which the Lipschitz constant of $P_t f$ blows up as $t$ decreases to $0$ is integrable, 
see Theorem \ref{thm:main:lipschitz} below together with Remark 
\ref{rem:4:13} for possible applications to infinite dimensional PDEs. 

We first consider a finite dimensional reduction of the problem. For a given truncation level $M \in {\mathbb N}_0$
and for any $v$ in $L^2_{\rm sym}(\bS)$, we let
\begin{equation*}
v^{M} :=  \sum_{m =0}^M \langle v,e_m\rangle e_m(\cdot), \quad v^{*,M} := \biggl( \sum_{m =0}^M \langle v,e_m\rangle e_m(\cdot) \biggr)^*. 
\end{equation*}
Clearly, $v^{*,M}$ is an element of $U^2(\bS)$ parametrised by {the (finite) vector of} Fourier modes 
$(\langle v,e_m \rangle)_{m=0,\cdots,N}$. We let 
$E^M := \{ v^{M}, \ v \in L^2_{\rm sym}({\mathbb S})\}$
and 
$E^{*,M} := \{v^{*,M}, \ v \in L^2_{\rm sym}({\mathbb S})\}$. Obviously, $E^M \cong \bR^{M+1}$.
The point is to prove that, for any $t>0$, the mapping 
{
	$x \in E^M \mapsto P_tf (x^{*,M} )$
}
is Lipschitz continuous, with a Lipschitz constant independent of $M$. Reducing the dimensionality allows us to use many tools from finite dimensional analysis, notably 
{Rademacher's theorem}. 
Together with 
Proposition 
\ref{prop:flow}, the latter says
 that, for a given $t>0$, the flow
 {
 	$
x \in E^M \mapsto X_t^{x^*}
%
$
}
is almost everywhere (for the Lebesgue measure on $E^M$) differentiable. 
{In this way, we 
avoid having to establish the everywhere differentiability of the flow of solutions to \eqref{eq reflSHE} 
with respect to the initial condition, a property that is not clear to us at this stage. (See
however, the references 
 \cite{Andres1,Andres2,deuschelZambotti2005bismutElworthySDErefl,LipschutzRamanan1,LipschutzRamanan2}
cited in Subsection 
\ref{subse:1.3}
for positive results in this direction when the reflected dynamics take values in a finite dimensional space.)}

The second step is to consider, for 
$x,v \in L^2_{\rm sym}({\mathbb S})$, $\delta \in {\mathbb R}$ and $T>0$,
the difference
\begin{equation}\label{eq:exp:semigroup}
	\begin{split}
	P_Tf\bigl( (x+\delta v)^{*,M} \bigr) - P_Tf\bigl(x^{*,M}\bigr) = \bE \Bigl[ f\bigl(X^{(x +\delta v)^{*,M}}_T\bigr) \Bigr]- 
	 \bE \Bigl[ f\bigl(X^{x^{*,M}}_T\bigr) \Bigr],
	\end{split} 
\end{equation}
and  to represent it \textit{via} use of a Girsanov transformation. This adapts earlier arguments from 
Malliavin calculus, {see 
  \cite{Bismut,norris1986simplifiedMallCalc}}, and
from the proof of the so-called Bismut-Elworthy{-Li} formula, 
{see  
 \cite{ElworthyLi,Thalmaier}}. The key idea is to consider the \emph{shifted process} 
$$\Bigl( X^{(x+\delta\frac{T-t}{T}v)^{*,M}}_t\Bigr)_{0 \le t \le T},$$
which satisfies 
$X^{(x+\delta\frac{T-t}{T}v)^{*,M}}_0=(x+\delta v)^{*,M}$
and $X^{(x+\delta\frac{T-t}{T}v)^{*,M}}_T=X_T^{x^{*,M}}.$
{It is  shown that, under a particular change of measure, the shifted process is the unique solution to 
\eqref{eq reflSHE}  started from $(x+\delta v)^{*,M}$. 
To guarantee the  conditions of Girsanov transformation, 
we need   to localise the dynamics and to  
enact the \textit{shifting} up to a stopping time, the form of which is clarified in  
	\eqref{eq:q:y,tauy}.}
{For $t \in [0,T],  y,v \in E^M, \delta \in {\mathbb R}$, we let}
\begin{equation}
\label{eq:notation:shift}
\begin{split}
&y_t(v,\delta) :=   y + \delta \frac{T-t}{T} v,\quad\text{and}\quad
y_t^*(v,\delta) := \Bigl( y + \delta \frac{T-t}{T} v\Bigr)^*.
\end{split}
\end{equation}
{The time} horizon $T$ is implicitly understood (and omitted) in the two left-hand sides
and   the function $y$ is manifested by the notations $y_t$ and $y_t^*$. For say, $x \in E^M$, we write 
$x_t(v,\delta)$ and $x_t^*(v,\delta)$. 
Moreover, frequently we will take derivatives with respect to the finite-dimensional variable $y \in E^M$. The gradient is denoted $\partial_y$. 
Notice   that for $y \in E^M$, 
both $y$ and $y^*$ can be regarded as elements 
of $L^2_{\rm sym}({\mathbb S})$. 
{By Lemma \ref{lem:isometry} and Parseval identity, 
$\| y^* \|_2=\| y \|_2 = \vert y \vert$ where
$\vert \cdot \vert$ stands for the   Euclidean norm on $E^M$}.

\subsection{Shifted state process and tilted reflection process}
\label{subse:5.1}

{In this subsection, we address  the dynamics of the shifted pair $(X_t^{y_t^*(v,\delta)},\eta_t^{y_t^*(v,\delta)})_{0 \leq t \leq T}$, for $y$ and $v$ in $E^M$ and $\delta \in {\mathbb R}$. Until further notice, the number $M$ of low frequency modes to which the initial condition is truncated is   fixed and  the indexing with respect to $M$ is omitted. To identify the dynamics, we expand the time evolution of the Fourier modes of the shifted process integrated against a test function $\varphi\in \cC^\infty_0( E^M)$:}
\begin{equation*}
\biggl( \int_{E^M} \bigl\langle  X^{y_t^*(v,\delta)}_t , e_m \bigr\rangle \varphi(y)  dy \biggr)_{0 \leq t \leq T},
\end{equation*}
with 
$\langle X^{y_t^*(v,\delta)}_t , e_m \rangle$
standing for the $m^{\rm th}$ Fourier mode of $X_t^{y_t^*(v,\delta)}$. 

Changing variables and recalling the dynamics \eqref{eq RSHE},
we have
\begin{align}
	\label{eq expansion0}
		&\int_{E^M} \bigl\langle   X^{y_t^*(v,\delta)}_t , e_m \bigr\rangle \varphi(y)  dy 
		= { \int_{E^M} \bigl\langle  X^{y^*}_t , e_m \bigr\rangle \varphi\bigl(y_t(v,-\delta) \bigr)  dy } \nonumber  \\
		&=     \int_{E^M} \bigl\langle X^{y^* }_0 , e_m  \bigr\rangle \varphi\bigl(y- \delta  v\bigr)    dy  +\int_{E^M}    \int_0^t \bigl\langle    X^{y^*}_s  ,\Delta e_m \bigr\rangle \varphi\bigl(
		y_s(v,-\delta) \bigr) ds \, dy \nonumber
		\\ 
		&\hspace{5pt} + \int_{E^M} \left( \int_0^t \varphi\bigl(y_s(v,- \delta) \bigr) d \bigl\langle \eta^{y^*}_s, e_m \bigr\rangle \right) dy   +
		\int_{E^M}
		\left(
		 \int_0^t \varphi\bigl(y_s(v,- \delta)\bigr)d \bigl\langle W_s  ,e_m \bigr\rangle \right)  dy \nonumber\\
		&\hspace{5pt} +  \tfrac{\delta}{T} \int_{E^M} \int_0^t \bigl\langle X^{y^*}_s, e_m \bigr\rangle  \partial_y \varphi\bigl(y_s(v, - \delta) \bigr) \cdot v  \, ds \, dy,
\end{align}
where 
$\partial_y \varphi(y_s( v,- \delta)) \cdot v$ represents the gradient of $\varphi$ in the direction of $v$ at point 
$y_s(v, - \delta)$.
{Notice that the well-posedness of the second integral in the penultimate line is guaranteed by the 
stochastic version of Fubini's theorem, \cite[Theorem 4.33]{daPratoZabczyk2014stochEqnsInfDim}.}

Ideally, we would like to revert back the variables in the various integrals appearing in the expansion \eqref{eq expansion0} and hence to compute the test function 
$\varphi$ at the generic point $y$ instead of $y_s(v,-\delta)$.
{The main difficulty is to handle the integral}  
\begin{equation}
\label{eq:sec:4:integral:we:want:to:define}
 \int_{E^M} \biggl( \int_0^t \varphi\bigl(y_s(v,-\delta)\bigr) d \bigl\langle \eta^{y^*}_s, e_m \bigr\rangle \biggr) dy.
 \end{equation} 
{The existence of the time integral in the right-hand side 
follows from 
\eqref{eq:integral:RS:bound:0:bis} 
and
\eqref{eq:reg:flow:eta}, 
the latter ensuring 
in particular 
the measurability of the mapping 
$(\omega,t,y)  \mapsto \langle \eta_t^{y^*}(\omega),e_m \rangle$.
The combination of both guarantees that, almost surely, } 
\begin{equation}
\label{eq:growth:eta:y}
\forall R >0, \quad 
\sup_{0 \leq t \leq T} 
\sup_{\vert y \vert \leq R} 
\bigl\| \eta_t^{y^*} \bigr\|_{2,-2} 
\leq  
\sup_{0 \leq t \leq T} 
\bigl\| \eta_t^0 \bigr\|_{2,-2} 
+ c_T R < \infty, 
\end{equation}
for a constant $c_T$ only depending on $T$. 

{Before formulating a convenient change of variables for 
\eqref{eq:sec:4:integral:we:want:to:define} in the forthcoming Proposition \ref{prop:4.4:rewrite}, we introduce  
some useful ingredients. This includes defining a so-called 
\textit{tilted} version 
$(\tilde \eta_t^{y^*})_{0 \leq t \leq T}$
of the reflection process $( \eta_t^{y^*})_{0 \leq t \leq T}$. 
We proceed below as in Subsection 
\ref{subse:3:2} and formally replace (at least for the first result) the realisation of 
the field $((\eta_t^y)_{0 \leq t \leq T})_{y \in E^{*,M}}$ (here, $y$ is directly assumed to be in $E^{*,M}$) by a {\bf deterministic} 
flow  
$((n_t^y)_{0 \le t \le T})_{y \in E^{*,M}}$  
satisfying the following two properties:}
\begin{itemize}
\item[({\bf F1})]
For any $y \in E^{*,M}$, 
the trajectory $t \in [0,T] \mapsto n_t^y$
satisfies ({\bf E1}) and ({\bf E2}) in 
Subsection 
\ref{subse:3:2}, restricted in an obvious manner to the interval 
$[0,T]$;
\item[({\bf F2})] 
The flow satisfies the Lipschitz condition
$\displaystyle \sup_{0 \leq t \leq T} \bigl\| n_t^x - n_t^y \bigr\|_{2,-2} \leq c_{{T}} \, \vert x-y \vert$,
for any $x,y \in E^{*,M}$. Equivalently, $ \sup_{0 \leq t \leq T} \bigl\| n_t^{x^*} - n_t^{y^*} \bigr\|_{2,-2} \leq c_{{T}} \, \vert x-y \vert$,
for any $x,y \in E^{M}$.  
\end{itemize}
The reader may reformulate
\eqref{eq:growth:eta:y}
accordingly. 
The following definition 
clarifies the form of the corrected (or \textit{tilted}) reflection term: 
\begin{definition}
\label{def:4.2}
Let $v \in   {E^M}$ and $\delta \in {\mathbb R}$ be given and $M$ be  
{as in 
\eqref{eq:exp:semigroup}}.
Moreover, 
let
$((n_t^y)_{0 \leq t \leq T})_{y \in E^{*,M}}$
satisfy 
{\rm ({\bf F1})}--{\rm ({\bf F2})}, and
 $((z_t^y)_{0 \leq t \leq T})_{y \in E^{*,M}}$ be a deterministic {jointly continuous $L^2_{\rm sym}({\mathbb S})$-valued flow}, i.e., the map 
$(t,y) \in [0,T] \times {E^{*,M}} \mapsto z_t^y \in L^2_{\rm sym}({\mathbb S})$ is continuous. 
Then, 
for any $m \in {\mathbb N}_0$,  
$y \in {E^M}$
and $t \in [0,T]$,  {define}
\begin{align}
\label{eq:def:4.2}
&\partial_y^*   n_t^{m,y} :=
\left\{
\begin{array}{ll}
\partial_y \bigl\langle n_t^{y^*},e_m \bigr\rangle
\quad 
&\text{\begin{tabular}{l}whenever the ${\mathbb R}^M$-valued gradient at
\\
point $y \in E^M$ in the right-hand side exists,
\end{tabular}}
\vspace{5pt}
\\
0
\quad 
&\text{\, otherwise},
\end{array}
\right.
\\
\label{eq:Fourier:tilde:n}
 {\text{and}}\quad 
&\bigl\langle \tilde{n}^{y,(v,\delta)}_t,e_m \bigr\rangle := 
\bigl\langle n_t^{y_t^*(v,\delta)}, e_m \bigr\rangle + \frac{\delta}{T}
\int_0^t \Bigl[ \partial_w^* n_s^{m,w} \cdot v \Bigr]_{\vert w= y_s(v,\delta)} \, ds.
\end{align}
\end{definition}
Measurability of 
$(t,y) \mapsto {\partial_y^*} n_t^{m,y}$ is obvious and, in fact,  
({\bf F2}) says that the derivative in  {Definition \ref{def:4.2}} exists for any $t \in [0,T]$, 
for almost every $y \in E^M$, and thus for almost every $(t,y) \in [0,T] 
\times E^M$.  
When
$((n_t^y)_{0 \leq t \leq T})_{y \in E^M}$ is
replaced by 
$((\eta_t^y)_{0 \leq t \leq T})_{y \in E^{*,M}}$,  
{the mapping $
%
(t,\omega,y) \in [0,T] \times E^M \times \Omega \mapsto  {\partial_y^*} \eta_t^{m,y}
$
}
 is measurable with respect to 
${\mathcal P} \times {\mathcal B}(E^M)$, where 
${\mathcal P}$ is the progressive $\sigma$-algebra on 
$[0,T] \times \Omega$ 
 {($\Omega$ equipped with the same ${\mathbb F}$ as before)}
and 
${\mathcal B}(E^M)$ is the Borel $\sigma$-algebra on $E^M$. 
{In the random setting, 
we define $\bigl\langle \tilde{\eta}^{y,(v,\delta)}_t,e_m \bigr\rangle$ by replacing the letter $n$ with $\eta$ throughout \eqref{eq:Fourier:tilde:n}.
}

It is easy to check that 
${\partial_y^*} n_t^{y}:=\sum_{m \in {\mathbb N}_0} {\partial_{{y}}^*} n_t^{m,y} e_m$ as defined through the Fourier modes  
\eqref{eq:def:4.2}
is an element of $H_{\rm sym}^{-2}({\mathbb S})$, 
since
 the series 
$\sum_{m \in {\mathbb N}_0} m^{-4} \vert {\partial^*_y} n_t^{m,y} \vert^2
$ is bounded. 
In particular, we have, {for a constant $C_c$ depending on 
$c$ in 
({\bf F2})
}
\begin{equation}
\label{eq:lem:4.3:main:ineq}
\bigl\|{ \partial^*_y} n_t^{y}
\bigr\|_{2,-2} \leq C_c.
\end{equation}
As a corollary of 
\eqref{eq:lem:4.3:main:ineq}, we get the following statement, which allows us to regard 
$(\tilde n_t^y)_{0 \leq t \leq T}$ as a continuous path with values in 
$H^{-2}_{\rm sym}({\mathbb S})$:
\begin{corollary}
\label{corollary:lem:4.3}
Within the framework of Definition 
\ref{def:4.2}, for any $R>0$, 
\begin{equation}
\label{eq:growth:tilde:eta:y}
\forall R >0, \quad 
\sup_{0 \leq t \leq T} 
\sup_{\vert y \vert \leq R} 
\Bigl\| \tilde n_t^{{y,(v,\delta)}} \Bigr\|_{2,-2} 
\leq  
\sup_{0 \leq t \leq T} 
\bigl\| n_t^0 \bigr\|_{2,-2} 
+ c \bigl( R + \delta \vert v\vert \bigr) < \infty, 
\end{equation}
for the same constant $c$ as in {\rm ({\bf F2})}.
\end{corollary}

Based on the  process $(\tilde \eta_t^{{y},(v,\delta)})_{0 \leq t \leq T}$
introduced prior, we provide here  a useful change of variables for 
\eqref{eq:sec:4:integral:we:want:to:define}.  It goes through the following 
notation. 
For $\varphi \in \cC^\infty_0( E^M)$, 
$\delta \in {\mathbb R}$ and $y,v \in E^M$,
we let 
\begin{equation}
\label{eq:tilde:eta:varphi}
\tilde n_t^{\varphi,{(v,\delta)}} := \int_{E^M}
\varphi(y) \tilde{n}_t^{y,{(v,\delta)}}
dy, 
\quad t \in [0,T], 
\end{equation}
which is regarded as a path with values in ${H}_{\rm sym}^{-2}({\mathbb S})$. 
The following statement 
will be proven in Subsection 
\ref{subse:technical}:
\begin{lemma}
\label{cor:4.5}
Assume that the function $\varphi$ is positive-valued. 
Then, 
for any %
$z \in U^2({\mathbb S}) \cap H^2_{\rm sym}({\mathbb S})$, 
the process $(\langle \tilde n_t^{\varphi,{(v,\delta)}},z \rangle)_{0 \leq t \leq T}$ is non-decreasing (in time). 
\end{lemma}

The above corollary 
shows that $(\tilde n_t^{\varphi,{(v,\delta)}})_{0 \le t \le T}$ satisfies 
Assumption ({\bf E2}) in Subsection 
\ref{subse:3:2}. Since 
({\bf E1})
 {(time continuity with respect to $\| \cdot \|_{2,-3}$)}
follows quite obviously from the joint continuity of the flow 
$((n_t^y)_{0 \leq t \leq T})_{y \in E^{*,M}}$, 
we can invoke Definition 
\ref{def:integral:RS}
to give meaning to the integral
\begin{equation*}
\biggl( \int_0^t 
\bigl\langle e^{\varepsilon \Delta} z_s,d \tilde{n}_s^{\varphi, {(v,\delta)}}
\bigr\rangle
\biggr)_{0 \leq t \leq T},
\end{equation*}
for $\varepsilon >0$ and 
for 
a continuous 
path
$(z_t)_{0 \le t \le T}$ with values in 
$L^2_{\rm sym}({\mathbb S})$, at least when 
$\varphi$ takes non-negative values, which will suffice for our purpose. 
In order to state our change of variable in a convenient way, we also let 
\begin{equation}
\label{eq:tilde:z:varphi}
\tilde{z}_t^{y,\varphi,{(v,\delta)}} := \varphi(y_t(v,-\delta))z_t, \quad t \in [0,T]. 
\end{equation}
We claim
(the proof, which is technical, is also deferred 
to Subsection 
\ref{subse:technical}) 
\begin{proposition}
\label{prop:4.4:rewrite}
Let
$((n_t^y)_{0 \leq t \leq T})_{y \in E^{*,M}}$
satisfy 
{\rm ({\bf F1})}--{\rm ({\bf F2})}, 
$\varphi$ be a non-negative valued
test function in 
$\cC^\infty_0( E^M)$
and
$(z_t)_{0 \le t \le T}$ is  
be a continuous 
path
with values in 
$L^2_{\rm sym}({\mathbb S})$. Then, one has
\begin{equation}
\label{eq:rewriting:propr:4.4}
		\int_E 
\biggl(		\int_0^T 
		\Bigl\langle e^{\varepsilon\Delta }\tilde z_t^{y,\varphi, {(v,\delta)}} , d  n^{{y}^*}_t \Bigr\rangle   \biggr) dy
		=
		 \int_0^T
\bigl\langle e^{\varepsilon \Delta} z_t,d \tilde{n}_t^{\varphi, {(v,\delta)}}
\bigr\rangle, \quad t \in [0,T]. 
\end{equation}
\end{proposition}
\begin{remark}
\label{rem:integral:pathwise:y}
Following Remark 
\ref{rem:integral:pathwise}, all results {given} in this subsection can be applied when 
$((n_t^y)_{0 \leq t \leq T})_{y \in E^{*,M}}$ is the realisation of 
$((\eta_t^y)_{0 \leq t \leq T})_{y \in E^{*,M}}$
(which satisfies ({\bf F1}) and ({\bf F2}) thanks to 
Remark 
\ref{rem:flow:existence}).
\end{remark}
\color{black}

\subsection{Integrating in time the shifted tilted reflection process}
{We here achieve two objectives; not only do we explicit the dynamics of the  process $(X^{y_t^*(v,\delta)}_t)_{0 \le t \le T}$ (for a given $y \in E^M$), but we also clarify the notion of a
pathwise integral with respect to 
the \textit{tilted} 
reflection process  
$(\tilde{\eta}_t^{y,(v,\delta)})_{0 \leq t \leq T}$
defined in 
\eqref{eq:Fourier:tilde:n}. 
In comparison, the integral with respect to 
$(\tilde{\eta}_t^{\varphi,(v,\delta)})_{0 \leq t \leq T}$
(see \eqref{eq:tilde:eta:varphi}, replacing   
$(\tilde{n}_t^{y,(v,\delta)})_{0 \leq t \leq T}$
by
$(\tilde{\eta}^{y,(v,\delta)}_t)_{0 \leq t \leq T}$) is just constructed for $\varphi$ smooth.  Somehow, we want   
to choose $\varphi$ as a Dirac mass.}
\begin{proposition}
\label{prop:4:8}
Fix $\delta \in {\mathbb R}$
and
$v \in  {E^M}$.
Then, with probability 1, 
for  
almost every $y \in E^M$, 
the path 
$(\tilde{\eta}_t^{ {y,(v,\delta)}})_{0 \leq t \leq T}$ satisfies  
{\rm ({\bf E1})} and {\rm ({\bf E2})} in Subsection 
\ref{subse:3:2}, {so that Definition 
\ref{def:integral:RS}
can be invoked}
 in order to define an integral with respect to 
$(\tilde{\eta}_t^{ y_t^*(v,\delta)})_{0 \leq t \leq T}$. 

Moreover, with probability 1, 
for 
almost every $y \in E^M$, for any $u \in H^2_{\rm sym}({\mathbb S})$,  
\begin{equation}
\label{eq shiftedDynamics}
	\begin{split}
		\bigl\langle    X^{y_t^*(v,\delta)}_t , u  \bigr\rangle    	=  &   \bigl\langle    X^{(y+\delta v)^*}_0 , u \bigr\rangle + 
		\int_0^t \bigl\langle    X^{y^*_s(v,\delta)}_s  ,\Delta u \bigr\rangle   ds 
		 +  \int_0^t   d \bigl\langle \tilde{\eta}^{{y,(v,\delta)}}_s, u \bigr\rangle   
		 \\
		 &+ \int_0^t d\langle  {W}_s  ,u\rangle   -    \tfrac{\delta}{T}  \int_0^t \partial_y \bigl\langle  X^{y^*_s(v,\delta)}_s ,u\bigr\rangle \cdot v    \,   ds.
	\end{split} 
\end{equation}
With probability 1, for almost every
$y \in E^M$,  almost every $t \in [0,T]$ and  any $u \in H^2_{\rm sym}({\mathbb S})$,  
the derivative $\partial_y \langle X^{y_s^*(v,\delta)}_s,u \rangle$
appearing in 
\eqref{eq shiftedDynamics}
 exists
and is jointly measurable  on $\Omega \times [0,T] \times E^M$.
\end{proposition}

\begin{proof}
{The existence of the derivatives, as stated in 
the last sentence, follows from 
Rademacher's theorem. Indeed, by 
\cite[Theorem 4]{Bongiorno},
with probability 1, 
for any $t \in [0,T]$, the map $y \in E^M \mapsto 
X^{y^*}_t \in L^2_{\rm sym}({\mathbb S})$ is almost everywhere differentiable}.
By Fubini's theorem, 
the 
map $(\omega,t,y) \mapsto \partial_y X^{y^*}_t(\omega)$ is 
hence defined up to a negligible subset of 
$\Omega \times [0,T] \times E^M$ and 
induces a jointly measurable mapping on $\Omega \times [0,T] \times E^M$. Since, for any $t \in [0,T]$, 
the map 
$y \mapsto y_t(v,\delta)=y + \delta (T-t)/T \, v$ preserves the Lebesgue measure, 
we deduce that 
$(\omega,t,y) \mapsto \partial_y X^{y_t^*(v,\delta)}_t$ is 
also defined up to a negligible subset of 
$\Omega \times [0,T] \times E^M$ and 
also induces a jointly measurable mapping on $\Omega \times [0,T] \times E^M$.
\vskip 4pt

\noindent \textit{First step.}
{Notice that, 
once 
\eqref{eq shiftedDynamics}
has been proven to hold true, for a given $u$, with probability 1 and for almost every $y$, it is easy 
 to get the result with probability 1, for almost every $y$ and  any $u \in H^2_{\rm sym}({\mathbb S})$. 
 It suffices to 
use the separability of $H^2_{\rm sym}({\mathbb S})$.}

In order to prove \eqref{eq shiftedDynamics}, we consider a function
$\varphi \in \cC^\infty_0(E^M)$ with  values in $[0,+\infty)$. 
{Repeating \eqref{eq expansion0}},
we get, for any $u \in H^2_{\rm sym}({\mathbb S})$, with probability 1,
\begin{equation} \notag
	\begin{split}
		&\int_{E^M} \bigl\langle X^{y_t^*(\delta,v)}_t,u \bigr\rangle \varphi(y)  dy %
		\\
		&= 
		    \int_{E^M} \bigl\langle X^{y^* }_0, u\bigr\rangle \varphi(y- \delta  v)    dy   
		    + \int_{E^M}    \int_0^t 
		\bigl\langle    X^{y^*}_s  ,\Delta u \bigr\rangle \varphi\bigl(y_s(v,-\delta)\bigr) ds dy
		\\
		& 
		\hspace{5pt} + \int_{E^M} \Bigl( \int_0^t \varphi\bigl(y_s(v,-\delta)\bigr) d \bigl\langle \eta^{y^*}_s, u \bigr\rangle \Bigr) dy  + 
		\int_{E^M} \Bigl( 
		\int_0^t \varphi\bigl(y_s({v},-\delta)\bigr)
		d\bigl\langle W_s  ,u \bigr\rangle \Bigr) dy 
		\\
		&\hspace{5pt}  + \tfrac{\delta}{T} \int_{E^M} \int_0^t \bigl\langle   X^{y^*}_s, u \bigr\rangle \, 
		 \partial_y \varphi \bigl(y_s(v, - \delta)\bigr) \cdot  v  \, ds dy.
\end{split}
\end{equation}
Here is the key point. By 
\eqref{eq:tilde:z:varphi}
and
\eqref{eq:rewriting:propr:4.4}, we can perform a change of variable in the penultimate line
(with $z_t=u$). As for the last line, 
we can make an integration by parts, recalling the  Lipschitz property of the flow, see 
Proposition \ref{prop:flow}. We get
\begin{equation}
\label{eq expansion}
\begin{split}
	&\int_{E^M} \bigl\langle X^{y_t^*(v,\delta)}_t,u \bigr\rangle \varphi(y)  dy 
	\\
&		=     \int_{E^M} \bigl\langle X^{(y+\delta  v)^* }_0,u  \bigr\rangle \varphi(y) dy  +
	 \int_{E^M}   
		\int_0^t \bigl\langle    X^{y_s^*(v,\delta)}_s  , \Delta u \bigr\rangle \varphi(y) ds dy 
		+   \bigl\langle \tilde{\eta}^{\varphi,{(v,\delta)}}_t,u \bigr\rangle  
		\\ 
		&\hspace{5pt} +   
		\bigl\langle W_t  ,u \bigr\rangle
\int_{E^M} \varphi(y)  dy  -  \tfrac{\delta}{T} \int_{E^M}     \int_0^t  \varphi(y) 
		\partial_y \bigl\langle   X_s^{\color{black} y^*_s(v,\delta)} , u \bigr\rangle \cdot v      \,  ds dy.
		\end{split}
\end{equation}

\noindent \textit{Second step.}
Assume $\varphi \geq 0$. 
 By invoking Lemma 
\ref{lem:4:9} below and then by taking the supremum over $u \in  {H^2_{\rm sym}({\mathbb S})}$ in 
\eqref{eq expansion} right above, we claim
\begin{equation*}
\bigl\| \tilde \eta_t^{\varphi,{(v,\delta)}} - \tilde \eta_s^{\varphi,{(v,\delta)}}
\bigr\|_{2,-2}
 \leq \zeta_R(\vert t-s \vert) \int_{E^M} \varphi(y) dy,
\quad (s,t) \in [0,T],
\end{equation*}
for a 
random field $\zeta_R$, with values in $(0,\infty)$ and with $\lim_{\rho \rightarrow 0} \zeta_R(\rho) = 0$ almost surely. 
{By the definition \eqref{eq:tilde:eta:varphi} of 
$\tilde \eta^{\varphi,(v,\delta)}$, we get that, for any $u \in H^2_{\rm sym}({\mathbb S})$,  
with $\| u \|_{2,2} \leq 1$}, 
\begin{equation*}
\begin{split}
\biggl\vert 
 \int_{E^M}
\varphi(y)\Bigl\langle \tilde{\eta}_t^{{y,(v,\delta)}}
-
 \tilde{\eta}_s^{{y,(v,\delta)}},u 
\Bigr\rangle dy
\biggr\vert \leq 
  \zeta_R(\vert t-s \vert) \int_{E^M} \varphi(y) dy,
\quad (s,t) \in [0,T]^2.
\end{split}
\end{equation*}
 We deduce that, with probability 1, for almost every $y \in E^M$, 
%
 for $s,t$ in a dense countable subset of $[0,T]$, 
 for $u$ in a dense countable subset of the unit ball of {$H^2_{\rm sym}({\mathbb S})$}, 
 \begin{equation*}
 \Bigl\vert 
  \bigl\langle \tilde{\eta}_t^{{y,(v,\delta)}} ,u \bigr\rangle -
  \bigl\langle \tilde{\eta}_s^{{y,(v,\delta)}} ,u \bigr\rangle
  \Bigr\vert \leq  \zeta_R(\vert t-s \vert){.}
 \end{equation*}
 Therefore, with probability 1, for almost every $y \in E^M$, for $s,t$ in a dense countable subset of $[0,T]$,
 \begin{equation*}
 \Bigl\|
 \tilde{\eta}_t^{{y,(v,\delta)}} -
  \tilde{\eta}_s^{{y,(v,\delta)}} 
  \Bigr\|_{2,-2} 
 \leq  \zeta_R(\vert t-s \vert){,} 
 \end{equation*}
 and then we have a continuous extension to the whole $[0,T]$. 
\vskip 4pt

\noindent
\textit{Third step.}
We observe 
from 
Lemma 
\ref{cor:4.5} that, with probability 1, for almost every $y$, 
for $s,t$ in a dense countable subset of $[0,T]$, for $u$ in a countable subset of 
$U^2({\mathbb S})$, {$\langle
 \tilde{\eta}_t^{ {y,(v,\delta)}} -
  \tilde{\eta}_s^{ {y,(v,\delta)}},u
  \rangle \geq 0$.  By density, the continuous extension satisfies 
the same inequality} for all $s,t \in [0,T]$ and ${u \in H^2_{\rm sym}({\mathbb S})}$. 
  This completes the proof of 
  the first part of the statement in 
  Proposition 
\ref{prop:4:8}. 
In particular, we can construct, almost surely, for almost every $y \in E^M$, an integral with respect to 
the process 
       $( \tilde{\eta}_t^{{y,(v,\delta)}} )_{0 \leq t \leq T}$
       as the latter satisfies ({\bf E1}) and ({\bf E2}) in Subsection 
       \ref{subse:3:2}.  
{Combining 
               Lemma 
       \ref{lem:RS:3}
 and Corollary 
 \ref{corollary:lem:4.3}, we deduce that  the identity 
                \eqref{eq:tilde:eta:varphi} is preserved, despite the additional extension by continuity. Proposition \ref{prop:4:8} then follows by inserting 
  \eqref{eq:tilde:eta:varphi}
 into 
\eqref{eq expansion}.} \qed
  \end{proof}

In the proof of Proposition 
\ref{prop:4:8}, we made use of the following statement:
\begin{lemma}
\label{lem:4:9}
Let $\delta \in {\mathbb R}$ and $v \in E^M$. 
Then, for any $R>0$, 
there exists a random field $\zeta_R$, with values in $(0,\infty)$ and with 
{$\lim_{\rho \rightarrow 0} \zeta_R(\rho)= 0$}, such that, for $\vert y \vert \leq R$, 
\begin{equation*}
\bigl\| X_t^{y_t^*(v,\delta)} - X_s^{y_s^*(v,\delta)}
\bigr\|_{2} \leq \zeta_R(\vert t-s \vert),
\quad (s,t) \in [0,T].
\end{equation*}
\end{lemma}

\begin{proof} 
{By Remark \ref{rem:flow:existence},
the flow
$((X_t^y)_{0 \leq t \leq T})_{y \in E^{*,M}}$ is jointly continuous in 
$(t,y)$. }
\qed
\end{proof}

\subsection{Dynamics under new probability measure}

For $\delta \in {\mathbb R}$ and $v \in E^M$ as before (with $M \in {\mathbb N}$ fixed), Proposition
\ref{prop:4:8}
prompts us to implement a Girsanov transformation in such a way that, 
for almost every $y \in E^M$, 
under a new probability measure
${\mathbb Q}^{{y,(v,\delta)}}$ depending on $y$, 
the process
$(\widetilde{W}_t^{{y,(v,\delta)}})_{t \geq 0}$ defined via Fourier modes by
{(recalling the definition \eqref{def:tilde:noise} of the 
weights 
$(\lambda_m)_{m \in {\mathbb N}_0}$):}
$$\bigl\langle \widetilde{W}_t^{{y,(v,\delta)}},e_m \bigr\rangle := \langle  {W}_t  ,e_m \rangle - \tfrac{\delta}{T}  \int_0^t \lambda_m^{-1} 
\partial_y \bigl\langle   X^{y_t^*(v,\delta)}_t ,e_m\bigr\rangle \cdot v \, ds, \quad t \geq 0, \quad m \in {\mathbb N}_0,$$
becomes a $Q$-Wiener process. Of course, it must be stressed that, in the definition 
of $((\langle \widetilde{W}_t^{{y,(v,\delta)}} ,e_m \rangle)_{0 \le t \le T})_{y \in E^M}$, 
 the integral process
only exists for almost every $\omega \in \Omega$ and 
for almost every $y \in E^M$, see 
the last line in Proposition 
\ref{prop:4:8}. In order to remedy this issue, {denote}
\begin{equation*}
\chi^{m,{y,(v,\delta)}}_t := \left\{
\begin{array}{ll}
\partial_y \bigl\langle  e_m, X^{y_t^*(v,\delta)}_t \bigr\rangle \cdot v
 \quad &\  \ \text{if the derivative exists},
\\
0 &\  \ \text{otherwise},
\end{array}\right.
\quad m \in {\mathbb N}_0, \quad y, v \in E^M, 
\end{equation*}
which allows one to extend the derivative when it does not exist. 

{Intuitively, one expects that, for almost every $y \in E^M$, 
the process $(X_t^{y_t^*(v,\delta)})_{0 \leq t \leq T}$ 
solves under the new probability measure, the conditions of Definition
\ref{def:existence}. As we will see next, the main challenge is in fact to verify the orthogonality \eqref{def RSHE 2}, but a 
first difficulty is to define the new probability measure. Whilst 
we wish to let}  
\begin{equation}
	\begin{split}
	\label{eq shiftedRNderiv}
		\frac{d\bQ^{{y,(v,\delta)}}}{d\bP} := \mathcal{E}_T\left\{  \tfrac{\delta}{T}   \sum_{m \in {\mathbb N}_0} \int_0^\cdot \lambda_m^{-1} \chi_s^{m,{y,(v,\delta)}}   d{B}^m_s \right\} ,
%
	\end{split}
\end{equation}
where ${\mathcal E}_T$ is a shorter notation for the Doléans-Dade exponential at time $T$, {it is however not immediate} that 
this defines indeed a new probability measure. 
For this reason, we employ a localisation argument by introducing the stopping time
\begin{equation}
\label{eq:tau:y}
	\begin{split}
		\tau_{{y,(v,\delta)}} := \inf \biggl\{ t \geq 0 \,:\, 
		\biggl\vert \int_0^t \sum_{m\in {\mathbb N}_0}
		 \lambda_m^{-1} \chi_s^{m,{y,(v,\delta)}}   d{B}^m_s \biggr\vert \geq \frac{T}{\delta}   \biggr\} \wedge T.
	\end{split}
\end{equation}
{%
Therefore, 
\begin{equation}
	\begin{split}
		\frac{d\bQ^{ {y,(v,\delta)},\tau}}{d\bP} := \mathcal{E}_T\biggl\{  \frac{\delta}{T}  \int_0^{\cdot\wedge\tau_{ {y,(v,\delta)}} } \sum_{m\in {\mathbb N}_0} \lambda_m^{-1} \chi_s^{m, {y,(v,\delta)}}   d {B}^m_s \biggr\}.
	\end{split}
	\label{eq:q:y,tauy}
\end{equation}
is a probability density 
(where we have omitted the subscript $(y,(v,\delta))$ in the index $\tau$ in the left-hand side) 
and the process 
$(\widetilde{W}_t^{y, {(v,\delta)},\tau})_{t \geq 0}$ defined in Fourier modes by
\begin{equation}
\label{eq:widetildeW:y}
\bigl\langle \widetilde{W}_t^{y, {(v,\delta)},\tau},e_m\bigr\rangle := \langle  {W}_t  ,e_m \rangle - \tfrac{\delta}{T}  \int_0^{t \wedge \tau_{{y,(v,\delta)}}} \lambda_m^{-1} \chi_s^{m,{y,(v,\delta)}}   ds, \quad t \geq 0
\quad m \in {\mathbb N}_0,
\end{equation}
is a $Q$-Wiener process under $\bQ^{{y,(v,\delta)},\tau}$.} Here is the main statement of this subsection:
\begin{proposition}
\label{prop:solution:under:new:proba:measure}
Let $y,v \in E^M$ and $\delta \in {\mathbb R}$. {For
$(\tilde{\eta}_t^{{y,(v,\delta)}})_{0 \le t \leq T}$ as
in
\eqref{eq:Fourier:tilde:n}}
and
 $\tau_{{y,(v,\delta)}}$ as in 
\eqref{eq:tau:y}, let (omitting the parameters $(v,\delta)$ in the notation $y^*(v,\delta)$)
\begin{equation*}
\begin{split}
&\widetilde{X}_t^{y,{(v,\delta)},\tau} := \left\{\begin{array}{l}
X_t^{ y_t^*}, \quad t \in [0,\tau_{y,{(v,\delta)}}]
\\
X_t^{y_{\tau_{y,{(v,\delta)}}}^*}, \quad t \in [\tau_{y,{(v,\delta)}},T]
\end{array}
\right., 
\vspace{5pt}
\\
&\tilde{\eta}_t^{y,{(v,\delta)},\tau} := \left\{\begin{array}{l}
\tilde{\eta}_t^{{y,(v,\delta)}}, \quad t \in [0,\tau_{y,{(v,\delta)}}]
\\
{\eta_t^{y_{\tau_{y,{(v,\delta)}}}^*}
-
\eta_{\tau_{y,{(v,\delta)}}}^{y_{\tau_{y,{(v,\delta)}}}^*}
+
\tilde{\eta}_{\tau_{y,{(v,\delta)}}}^{y,{(v,\delta)},\tau}}, \quad t \in [\tau_{y,{(v,\delta)}},T]
\end{array}
\right..
\end{split}
\end{equation*}
Then, for almost every $y \in E^M$, the process
$(\widetilde{X}_t^{y,{(v,\delta)},\tau},\tilde{\eta}_t^{y,{(v,\delta)},\tau})_{0 \leq t \leq T}$ satisfies 
Definition 
\ref{def:existence}
under the probability measure ${\mathbb Q}^{y,(v,\delta),\tau}$.
\end{proposition}

While the notation looks complicated,   
 $(\widetilde{X}_t^{y,{(v,\delta)},\tau},\tilde{\eta}_t^{y,{(v,\delta)},\tau})_{0 \leq t \leq T}$
has   in fact a quite simple interpretation: 
the shifting $y_t(v,\delta)$ is enacted up until time $t=\tau_{y,{(v,\delta)}}$. 

\begin{proof}
Item 1 in Definition 
\ref{def:existence} is easily checked by means of 
Proposition 
\ref{prop:flow}
and 
Remark 
\ref{rem:flow:existence}.
Items 2 and 3 follow from 
Proposition 
\ref{prop:4:8} and 
Remark 
\ref{rem:flow:existence}.
In both cases, 
the properties 
are proved on $[\tau_{y,{(v,\delta)}},T]$ by applying 
Definition \ref{def:existence} itself for the solution 
restarted from the random initial condition 
$X_{\tau_{y,{(v,\delta)}}}^{y^*_{\tau_{y,{(v,\delta)}}}}$
(omitting the notation $(v,\delta)$ in 
$y(v,\delta)$) and driven by the 
shifted version 
$(W_{t+\tau_{y,{(v,\delta)}}}-W_{\tau_{y,{(v,\delta)}}})_{t \geq 0}$
of the noise. 

The main difficulty is to check item 4 in 
Definition \ref{def:existence}.
As above, 
it is   easily verified on $[\tau_{y,{(v,\delta)}},T]$ by applying 
Definition \ref{def:existence} for the solution 
restarted from the random initial condition 
$X_{\tau_{y,{(v,\delta)}}}^{y^*_{\tau_{y,{(v,\delta)}}}}$ 
(for instance, we may invoke item 4 for the restarted solution on 
 $[\tau_{y,{(v,\delta)}},T+\tau_{y,{(v,\delta)}}]$ and then use
the non-decreasing property of the integral to 
get the result on 
$[\tau_{y,{(v,\delta)}},T]$). 
The key point is thus to prove that, for almost every $y \in E^M$,
$$\lim_{\varepsilon\searrow0} \E_{\bQ^{y,{(v,\delta)},\tau}} \left[ \int_0^{\tau_{y,{(v,\delta)}}}   e^{\varepsilon\Delta}X^{y^*_s(v,\delta)}_s \cdot  d \tilde\eta_s^{{y,(v,\delta)}}   \right] = 0.$$
 
By 
Proposition 
\ref{prop:item:iv:higher:moments}, it suffices to prove that the above convergence 
holds, for almost every $y \in E^M$, in ${\bQ}^{y,{(v,\delta)},\tau}$ probability, along a subsequence. 
Moreover, by the localisation 
procedure 
\eqref{eq:tau:y},
we have a bound on the moments of the density 
$d {\mathbb Q}^{y, {(v,\delta)},\tau}/d {\mathbb P}$ and 
it suffices to establish the convergence in ${\bP}$ probability only. Actually, since the integral is non-decreasing in time, 
it is sufficient to address the convergence for the integral on the entire $[0,T]$.  
The argument is as follows. 
We claim that, 
for $\varphi \in \cC^\infty_0(E^M)$ with non-negative values, 
 with probability 1, 
  \begin{equation}
 \label{eq:main:prop:14} 
  \begin{split}
	&\sum_{m \in {\mathbb N}_0}\int_{E^M}\int_0^T \Bigl\langle e^{\varepsilon\Delta }X_t^{y^*},e_m\Bigr\rangle 
		\varphi\bigl( y_t(v,-\delta)\bigr)d \bigl\langle \eta^{{y^*}}_t,e_m \bigr\rangle    dy
\\
	&=  \lim_{N\rightarrow\infty}\int_{E^M} \varphi(y )\sum_{i=0}^{N-1}     \left[  \biggl\langle \tilde{\eta}^{{y,(v,\delta)}}_{r_{i+1}},e^{\varepsilon\Delta }
	X_{r_i}^{{y^*_{r_i}(v,\delta)}}
	\biggr\rangle-  \biggl\langle \tilde{\eta}^{{y,(v,\delta)}}_{r_{i}},e^{\varepsilon\Delta }X_{r_i}^{{y^*_{r_i}(v,\delta)}} \biggr\rangle \right]dy, 
  \end{split}
  \end{equation}
  which is a straightforward consequence of 
 the forthcoming Proposition 
  \ref{prop:4.4} 
  (with $z_t^y = X_t^y$ therein).
  (Proposition \ref{prop:4.4} is a technical result, established in the final 
Subsection  
  \ref{subse:technical}
  dedicated to the proofs of 
Lemma \ref{cor:4.5} and Proposition \ref{prop:4.4:rewrite}.)  
Now, one can use 
\eqref{eq:growth:eta:y}
to exchange the order of 
summation and integration in 
the left-hand side 
of \eqref{eq:main:prop:14}. Thus, the latter is equal to 
$\int_{E^M} ( \int_0^T 
	  [ \varphi\bigl( y_t^*(v,-\delta)\bigr)
	e^{\varepsilon \Delta}
	X_t^{y^*}]	 \cdot d  \eta^{{y}^*}_t )    dy.$
As for the right-hand side of 
\eqref{eq:main:prop:14}, we can invoke Proposition \ref{prop:4:8}
 and 
 Corollary 
\ref{cor:RS:3}, 
and regard the sum therein as a Riemann sum 
associated with 
{$\int_0^T 
 (e^{\varepsilon\Delta }
X_{t}^{{y^*_t(v,\delta)}} ) 
\cdot 
d
\tilde{\eta}^{{{y,(v,\delta)}}}_{t}$}. 
Recalling 
the inequality 
\eqref{eq:integral:RS:bound:0:bis} 
	together with 
	Corollary 
	\ref{corollary:lem:4.3}
	and
	the fact that, with probability 1, the flow
	$((X_t^y)_{0 \leq t \leq T})_{y \in E^{*,M}}$ 
is jointly continuous (with values
in $L^2_{\rm sym}({\mathbb S})$), we can exchange the limit (over $N$) 
and the integral (in $y$). 
Therefore,  with probability 1, 
  \begin{equation}
  \label{eq:item:4:shifted:tilted}
	\begin{split}
	&\int_{E^M} \biggl( \int_0^T 
	  \Bigl[ \varphi\bigl( y_t^*(v,-\delta)\bigr)
	e^{\varepsilon \Delta}
	X_t^{y^*}
\Bigr]	 \cdot d  \eta^{{y}^*}_t \biggr)    dy
\\
& =  \int_{E^M} \varphi(y )
	\biggl( 
	\int_0^T 
	\bigl(e^{\varepsilon\Delta }
X_{t}^{{y_t^*(v,\delta)}}\bigr)
\cdot 
d
\tilde{\eta}^{{{y,(v,\delta)}}}_{t}
\biggr) dy. 
\end{split}
  \end{equation}
By 
following the proof of 
Proposition \ref{prop:item:iv:higher:moments}
(and in particular the proof of 
\eqref{eq:bound:suppl:item4''}),
we can 
easily have a bound for 
${\mathbb E} [
\int_0^T 
	[ \varphi\bigl( y_t(v,-\delta)\bigr)
	e^{\varepsilon \Delta}
	X_t^{y^*}
]	\cdot d  \eta^{{y}^*}_t  ]$
that is uniform with respect to 
{$\varepsilon \in (0,1)$
and to}
$y$ in compact subsets of $E^M$.
Therefore, by item 4 in Definition  
 \ref{def:existence}, we deduce that 
the expectation of the left-hand side in 
\eqref{eq:item:4:shifted:tilted}
tends to $0$ (with $\varepsilon$). Then, 
the expectation of the right-hand side 
in 
\eqref{eq:item:4:shifted:tilted} 
also tends to $0$ 
(with $\varepsilon$). 
Recalling that $\varphi$ is non-negative valued and assuming that 
$\varphi$ matches 1 on a given compact subset of $E^M$, we deduce 
from Fatou's lemma
that, for any $R>0$,
\begin{equation*}
 \int_{E^M} {\mathbf 1}_{\{ \vert y \vert \leq R \}}
 \liminf_{\varepsilon \searrow 0} 
	{\mathbb E} \biggl[
	\int_0^T 
	\bigl(e^{\varepsilon\Delta }
X_{t}^{{y_t^*(v,\delta)}}\bigr)
\cdot 
d
\tilde{\eta}^{{{y,(v,\delta)}}}_{t}
\biggr] dy = 0.
\end{equation*}
Therefore, 
 for almost every $y \in E^M$, 
there exists a subsequence $(\varepsilon_q)_{q \geq 1}$ (possibly depending on $y$), with $0$ as limit, such that 
 \begin{equation*}
\lim_{q \rightarrow \infty} 
	{\mathbb E} \biggl[
	\int_0^T 
	\Bigl(e^{\varepsilon_q \Delta }
X_{t}^{y^*_t(v,\delta)}\bigr)
\cdot 
d
\tilde{\eta}^{{{y,(v,\delta)}}}_{t}
\biggr]=0,
\end{equation*}
which implies convergence in probability, as we claimed. 
\color{black}
\qed
\end{proof}
\subsection{Regularity}

We arrive at the main statement of this section, 
which asserts that the semigroup generated by 
{the solution to \eqref{eq reflSHE}}
maps bounded functions into Lipschitz functions:
\begin{theorem}
\label{thm:main:lipschitz}
Assume that 
$\lambda$ in
\eqref{def:tilde:noise} 
is in $(\nicefrac12,1)$. 
Let $((X_t^{x^*})_{t \geq 0})_{x \in L^2_{\rm sym}({\mathbb S})}$
be the flow generated by \eqref{eq reflSHE}, as defined in Remark 
\ref{rem:flow:existence}. 
Then, 
there exists a constant $c_{\lambda}$, only depending on  
$\lambda$
such that, 
for any $t>0$ and any bounded (measurable) function $f : L^2_{\rm sym}({\mathbb S}) \rightarrow {\mathbb R}$, {
the function 
%
$
x \in L^2_{\rm sym}({\mathbb S}) \mapsto 
{\mathbb E} [ f\bigl(X_t^{x^*}\bigr) ]
$
}
 is Lipschitz continuous with $c_{\lambda} t^{-(1+\lambda)/2}$ as Lipschitz constant. 
\end{theorem} 

\begin{remark}
\label{rem:4:13}
{Notice that, for $\lambda \in (\nicefrac12,1)$, the exponent $(1+\lambda)/2$ is strictly less than 1. 
This guarantees that the rate at which the Lipschitz constant  {blows up} when time becomes small is integrable. 
This   is expected to have important applications for the analysis of   PDEs
on ${\mathcal P}({\mathbb R})$  
driven by the generator of 
the process 
$((X_t^x)_{t \geq 0})_{x \in L^2_{\rm sym}({\mathbb S})}$.}
\end{remark}

\begin{proof}  
\noindent \textit{First step.} 
We start with a bounded measurable function 
$f : L^2_{\rm sym}({\mathbb S}) \rightarrow {\mathbb R}$. 
We are also given a threshold $M$ as 
in \eqref{eq:exp:semigroup}, a time horizon $T>0$ and a {non-zero} element $v \in E^M$. By 
Proposition 
\ref{prop:solution:under:new:proba:measure}, we know that for almost every $y \in E^M$, under $\bQ^{y,{(v,\delta)},\tau}$, the 
process $(\widetilde X^{y,{(v,\delta)},\tau}_t,\tilde \eta^{y,{(v,\delta)},\tau}_t)_{0 \leq t \leq T}$ is the unique solution to the rearranged SHE started from $y + \delta v$
and driven by 
the tilted noise
\eqref{eq:widetildeW:y}. 
By an obvious adaptation of the Yamada-Watanabe theorem ({to} which we already alluded before the proof of 
Theorem \ref{thm:main:existence:uniqueness}),
 we 
deduce that 
not only 
 uniqueness holds
in the strong sense (as guaranteed by  
Theorem 
\ref{thm:main:existence:uniqueness})
but it also 
holds in the weak sense. 
Therefore, 
\begin{equation}\label{eq stopSplit}
	\begin{split}
		\E \Bigl[ f\bigl( X^{(y+\delta v)^*}_T\bigr) \Bigr] = 
		\E_{{\mathbb Q}^{y,(v,\delta),\tau}} \Bigl[ f\bigl( \widetilde X^{y,{(v,\delta)},\tau}_T\bigr) \Bigr],		
\end{split}
\end{equation}
where we recall the notations 
\eqref{eq:tau:y}  
and
\eqref{eq:q:y,tauy}.
\vskip 4pt   
   
\noindent
\textit{Second step.}  
\begin{equation*}
\begin{split}
{\mathbb Q}^{y,(v,\delta),\tau}
\Bigl( \bigl\{ \tau_{y,(v,\delta)} < T \bigr\} \Bigr) 
&=
{\mathbb Q}^{y,(v,\delta),\tau}
\biggl(
\biggl\{
\sup_{0 \le t \le T}
\biggl\vert \int_0^t \sum_{m\in {\mathbb N}_0}
		 \lambda_m^{-1} \chi_s^{m,{y,(v,\delta)}}   d {B}^m_s \biggr\vert   \geq \frac{T}{\delta}   \biggr\}
\\
&\leq \frac{\delta^2}{T^2} 
{\mathbb E}^{y,(v,\delta),\tau}
\int_0^T \sum_{m\in {\mathbb N}_0} \bigl\lvert \lambda_m^{-1}  \chi_s^{m, {y,(v,\delta)}}   \bigr\rvert^2 ds.
\end{split} 
\end{equation*} 
Using the fact that 
$d {\mathbb Q}^{y,(v,\delta),\tau}/ d {\mathbb P} \leq e$, 
we can rewrite this as 
\begin{equation*}
\begin{split}
&{\mathbb Q}^{y,(v,\delta),\tau}
\Bigl( \bigl\{ \tau_{y,(v,\delta)} < T \bigr\} \Bigr) 
\leq \frac{\delta^2 e}{T^2} 
{\mathbb E} 
\int_0^T \sum_{m\in {\mathbb N}_0} \bigl\lvert \lambda_m^{-1} 
\chi_s^{m,{y,(v,\delta)}} \bigr\rvert^2 ds.
\end{split} 
\end{equation*}
\textit{Third step.} Returning to equation \eqref{eq stopSplit}, one has that, for almost every $y \in E^M$,
\begin{align}
		&\E \Bigl[ f\bigl( X^{(y+\delta v)^*}_T\bigr) \Bigr] \nonumber
		\\
	&=
\E_{{\mathbb Q}^{y,{(v,\delta)},\tau}} \Bigl[ f \bigl(  \widetilde X^{y,{(v,\delta)},\tau}_T\bigr) {\mathbf 1}_{\{ T = \tau_{y,{(v,\delta)}} \}}\Bigr] \nonumber
+ 
\E_{{\mathbb Q}^{y,{(v,\delta)},\tau}} \Bigl[ f \bigl(  \widetilde X^{y,{(v,\delta)},\tau}_T\bigr) {\mathbf 1}_{\{  \tau_{y,{(v,\delta)}} < T\}}\Bigr] \nonumber
\\
&= \E_{{\mathbb Q}^{y,{(v,\delta)},\tau}} \Bigl[ f\bigl( X_T^{y^*} \bigr)  \Bigr]+{\Oh}\Bigl(
{\mathbb Q}^{y,(v,\delta),\tau}
\bigl( \bigl\{ \tau_{y,(v,\delta)} < T \bigr\} \bigr) 
\Bigr),
\label{eq stopSplit2}
\end{align}
where the {(big)} Landau symbol $\Oh(\cdot)$
is uniform with respect to $y$, $M$, $\delta$ and $v$.  
And then, for $v \in E^M$, 
and
for almost every $y \in E^M$,
\begin{equation*}
\begin{split}
&\Bigl\vert \E \Bigl[ f \bigl( X^{(y+\delta v)^*}_T\bigr) \Bigr] - \E \Bigl[ f\bigl( X^{y^*}_T \bigr) \Bigr] \Bigr\vert 
\\
&\leq  \| f \|_\infty 
{\rm d}_{TV} \bigl({\mathbb Q}^{y,{(v,\delta)},\tau},{\mathbb P}\bigr) +{\Oh}\Bigl(
{\mathbb Q}^{y,(v,\delta),\tau}
\bigl( \bigl\{ \tau_{y,(v,\delta)} < T \bigr\} \bigr) 
\Bigr).
\end{split}
\end{equation*}
where 
${\rm d}_{TV}$ is the distance in total variation (see \cite[p. 22]{villani})
By Pinsker's inequality (see \cite[Eq. (22.25)]{villani}), 
we get 
\begin{equation*}
\begin{split}
&\Bigl\vert \E \Bigl[ f\bigl( X^{(y+\delta v)^*}_T\bigr) \Bigr] - \E \Bigl[ f\bigl( X^{y^*}_T\bigr) \Bigr] \Bigr\vert 
\\
&\leq 
\| f \|_\infty 
\sqrt{2 \, {\mathbb E}_{{\mathbb Q}^{y, {(v,\delta)},\tau}} \Bigl[ \ln \Bigl( \frac{d{\mathbb Q}^{y,{(v,\delta)},\tau}}{d {\mathbb P}} \Bigr) \Bigr]} +  {\Oh}\Bigl(
{\mathbb Q}^{y,(v,\delta),\tau}
\bigl( \bigl\{ \tau_{y,(v,\delta)} < T \bigr\} \bigr) 
\Bigr).
\end{split}
\end{equation*}
By
\eqref{eq:tau:y}
and
	\eqref{eq:q:y,tauy},
\begin{equation*}
\begin{split}
\Bigl\vert \E \Bigl[ f\bigl( X^{(y+\delta v)^*}_T\bigr) \Bigr] - \E \Bigl[ f\bigl( X^{y^*}_T\bigr) \Bigr] \Bigr\vert
&\leq  
\frac{\delta \sqrt{e} \| f \|_\infty}{T} 
{\mathbb E} \biggl[  
\int_0^{\tau_{y,{(v,\delta)}}}  \sum_{m\in {\mathbb N}_0} \Bigl\vert   \lambda_m^{-1}   
\chi_s^{m,{y,(v,\delta)}}  \Bigr\vert^2 ds \biggr]
^{1/2}
\\
&\hspace{5pt} 
+  {\Oh}\Bigl(
{\mathbb Q}^{y,(v,\delta),\tau}
\bigl( \bigl\{ \tau_{y,(v,\delta)} < T \bigr\} \bigr) 
\Bigr).
\end{split}
\end{equation*}
By the second step, we end up with 
\begin{align}
\Bigl\vert \E \Bigl[ f\bigl( X^{(y+\delta v)^*}_T\bigr) \Bigr] - \E \Bigl[ f\bigl( X^{y^*}_T\bigr) \Bigr] \Bigr\vert
&\leq  
\frac{\delta \sqrt{e} \| f \|_\infty}{T} 
{\mathbb E} \biggl[  
\int_0^{T}  \sum_{m\in {\mathbb N}_0} \Bigl\vert   \lambda_m^{-1}   
\chi_s^{m,{y,(v,\delta)}}  \Bigr\vert^2 ds\biggr]
^{1/2}
\nonumber
\\
&\hspace{-5pt} + {\Oh}\biggl(
\frac{\delta^2 e}{T^2} 
{\mathbb E} \biggl[
\int_0^T \sum_{m\in {\mathbb N}_0} \bigl\lvert \lambda_m^{-1} 
\chi_s^{m,{y,(v,\delta)}} \bigr\rvert^2 ds \biggr]
\biggr).
\label{eq:4:33:aa}
\end{align}
%

\noindent \textit{Fourth Step.}
By 
H\"older's inequality, 
\begin{align*}
		 &{\mathbb E} \biggl[
\int_0^T  \sum_{m\in {\mathbb N}_0} \Bigl\vert   \lambda_m^{-1}   
 \chi_s^{m,{y,(v,\delta)}} \Bigr\vert^2 ds  \biggr] \nonumber
			 = 
		{\mathbb E}  \biggl[
\int_0^T  \sum_{m\in {\mathbb N}_0} \Bigl\vert   (1 \vee m)^{\lambda}  
\chi_s^{m,{y,(v,\delta)}}   \Bigr\vert^2 ds \biggr] \nonumber
		\\
&\leq   
 \E  \biggl[ \int_0^{T}   \sum_{m\in {\mathbb N}_0} (1 \vee m)^2 \Bigl\vert  
\chi_s^{m,{y,(v,\delta)}} \Bigr\vert^2 ds \biggr]^{\lambda}
\nonumber
 \E \biggl[ \int_0^{T}   \sum_{m\in{\mathbb N}_0} \Bigl\vert  
\chi_s^{m,{y,(v,\delta)}} \Bigr\vert^2 ds \biggr]^{1-\lambda}.
\end{align*}
To estimate the above, return
to Proposition 
\ref{prop:flow}.
Changing therein $(x,y)$ for $((y_t(\delta,v)+\mu v)^*,y_t^*(\delta,v))$ for $\mu >0$,
dividing by $\mu$ and letting $\mu$ tend to $0$, we observe that 
$ \sum_{m\in{\mathbb N}_0}
\vert  %
\chi_s^{m,{y,(v,\delta)}}  \vert^2 \leq \lVert v \rVert_2^2$. Therefore, 
\begin{equation}
\label{eq:4:33:bb}
\begin{split} 
{\mathbb E} \left[
\int_0^T  \sum_{m\in {\mathbb N}_0} \Bigl\vert   \lambda_m^{-1}   
 \chi_s^{m,{y,(v,\delta)}} \Bigr\vert^2 ds \right]
&\leq (T \lVert v \rVert_2^2  )^{(1-\lambda)}
\bigl[
\Xi^{y,(v,\delta)}_T \bigr]^\lambda
\\
%
\textrm{\rm with} \quad 
\Xi^{y,(v,\delta)}_T &:=\E \biggl[ \int_0^{T}   \sum_{m\in {\mathbb N}_0} (1 \vee m)^2 \Bigl\vert  
\chi_s^{m,{y,(v,\delta)}} \Bigr\vert^2 ds \biggr].
\end{split}
\end{equation}
\vskip 4pt

\noindent \textit{Fifth step.}
By
\eqref{eq:4:33:aa},
and 
\eqref{eq:4:33:bb},  there exists a constant $C$, independent of $M$, $\delta$, $v$ and $y$, such that
\begin{equation*}
\begin{split} 
&\Bigl\vert \E \Bigl[ f\bigl( X^{(y+ \delta v)^*}_T\bigr) \Bigr] - \E \Bigl[ f\bigl( X^{y^*}_T\bigr) \Bigr] \Bigr\vert
\\
&\leq
\frac{\delta \sqrt{e} \lVert v \rVert_2^{1-\lambda} \| f \|_\infty }{T^{(1+\lambda)/2}} 
  \bigl[
\Xi^{y,(v,\delta)}_T \bigr]^{\lambda/2}
 + 
\frac{C \delta^2\lVert v \rVert_2^{2(1-\lambda)}}{T^{1+\lambda}}
  \bigl[
\Xi^{y,(v,\delta)}_T \bigr]^\lambda.
\end{split}
\end{equation*}
Integrating with respect to $y \in E^M$
with respect to a smooth compactly supported density 
$\psi$ on $E^M$, using Jensen's inequality and dividing through by $\delta$ yields
\begin{equation}
\label{eq:conclusion:third:step:lipschitz}  
 \begin{split}
&\int_{E^M} 
\frac{\psi(y) }{\delta} \Bigl\vert \E \Bigl[ f\bigl( X^{(y+ \delta v)^*}_T\bigr) \Bigr] - \E \Bigl[ f\bigl( X^{y^*}_T\bigr) \Bigr] 
 \Bigr\vert
dy 
\\
&\leq
\frac{ \sqrt{e}\lVert v \rVert_2^{1-\lambda} \| f \|_\infty }{T^{(1+\lambda)/2}} 
 \bigl[
 \Xi^{\psi,(v,\delta)}_T 
 \big]^{\lambda/2}
+
\frac{C \delta \lVert v \rVert_2^{2(1-\lambda)}}{T^{1+\lambda}}
 \bigl[ 
  \Xi^{\psi,(v,\delta)}_T  \bigr]^\lambda,
\end{split} 
\end{equation} 
where
$\displaystyle 
 \Xi^{\psi,(v,\delta)}_T := \int_{E^M}  \Xi^{y,(v,\delta)}_T \psi(y) dy$, that is
\begin{equation*} 
\begin{split} 
 \Xi^{\psi,(v,\delta)}_T &=\int_0^{T}  \int_{E^M} 
  \sum_{m\in {\mathbb N}_0} (1 \vee m)^2 \Bigl\vert  
\partial_y \bigl\langle  e_m, X^{y_s^*(v,\delta)}_s \bigr\rangle \cdot v \Bigr\vert^2 
\psi(y) dy ds
\\
&= \int_0^{T}  \int_{E^M} 
  \sum_{m\in {\mathbb N}_0} (1 \vee m)^2 \Bigl\vert  
\partial_y \bigl\langle  e_m, X^{y^*}_s \bigr\rangle \cdot v \Bigr\vert^2 
\psi\bigl(y_s(v,-\delta)\bigr) dy ds,
\end{split}
\end{equation*} 
with the last line following from a change of variable (as done in \eqref{eq expansion}). 

Assume for a while (the proof is given right below) that we have a deterministic bound for 
$\int_0^T
 \sum_{m\in {\mathbb N}_0} (1 \vee m)^2 \vert  
\partial_y \bigl\langle  e_m, X^{y^*}_t \bigr\rangle \cdot v 
\vert^2 dt$, independently of $y$ (in a full subset of $E^M$). Then, by Lebesgue's dominated convergence theorem, we can let 
$\delta$ to $0$ in \eqref{eq:conclusion:third:step:lipschitz}:
\begin{equation*} 
\begin{split} 
&\liminf_{\delta \rightarrow 0} \int_{E^M}  
\frac{\psi(y)}{\delta} \Bigl\vert \E \Bigl[ f\bigl( X^{(y+ \delta v)^*}_T\bigr) \Bigr] - \E \Bigl[ f\bigl( X^{y^*}_T\bigr) \Bigr] 
  \Bigr\vert
dy
\leq
\frac{ \sqrt{e} \lVert v \rVert_2^{1-\lambda} \| f \|_\infty }{T^{(1+\lambda)/2}} 
 \bigl[ \Xi_T^{\psi,(v,0)} \bigr]^{\lambda/2}.
\end{split}
\end{equation*} 
To estimate the right-hand side, we proceed as in the 
derivation of 
\eqref{eq:4:33:bb}. We return
back to
\eqref{eq lip}
in
 Proposition 
\ref{prop:flow}, this time swapping $(x,y)$ therein for $((y+\mu v)^*,y^*)$ for $\mu >0$,
divide by $\mu$ and let  $\mu$ tend to $0$.
Subsequently, we let $\varepsilon$ in
\eqref{eq lip}
tend to $0$ and deduce by Fatou's lemma that 
\begin{equation*} 
\Xi_T^{y,(v,0)}
 = \int_0^{T}  
  \sum_{m\in {\mathbb N}_0} (1 \vee m)^2 \Bigl\vert  
\partial_y \bigl\langle  e_m, X^{y^*}_s \bigr\rangle \cdot v \Bigr\vert^2 
 ds
 \leq \lVert v \rVert_2^2, 
 \end{equation*}
 for almost every $y \in E^M$, 
which gives 
\begin{equation}
\label{eq:conclusion:third:step:lipschitz:2}  
\begin{split} 
&\liminf_{\delta \rightarrow 0} \int_{E^M} 
\frac{\psi(y) }{\delta} \Bigl\vert \E \Bigl[ f\bigl( X^{(y+ \delta v)^*}_T\bigr) \Bigr] - \E \Bigl[ f\bigl( X^{y^*}_T\bigr) \Bigr] 
  \Bigr\vert
dy
 \leq
\frac{ \sqrt{e}  \lVert v \rVert_2 \| f \|_\infty }{T^{(1+\lambda)/2} }.
\end{split}
\end{equation} 
\textit{Last Step.} We now
assume that $f$ itself is Lipschitz continuous. By Lipschitz continuity of the flow $(X_T^{y^*})_{y \in E^M}$, 
the mapping $P_T^M f : y \in E^M \mapsto 
{\mathbb E}[f(X_T^{y^*})]$ is Lipschitz {continuous} 
{(with respect to the Euclidean norm on $E^M$, which coincides with the $L^2({\mathbb S})$-norm)}
and thus almost everywhere differentiable.
Then,  
the left-hand side in \eqref{eq:conclusion:third:step:lipschitz:2}
is equal to $\int_{E^M} \psi(y) \vert \nabla_y P_T^M f(y) \cdot v \vert dy$.
Since the bound is true for any density $\psi$, we get that the almost everywhere gradient of 
$y \in E^M \mapsto 
{\mathbb E}[f(X_T^{y^*})]$ 
is less than $\sqrt{e}\| f \|_\infty T^{-(1+\lambda)/2}$.
Therefore, the Lipschitz constant of $y \in E^M \mapsto 
{\mathbb E}[f(X_T^{y^*})]$ is less than $\sqrt{e} \| f \|_\infty T^{-(1+\lambda)/2}$, 
 {with 
$E^M$ being equipped with the $L^2_{\rm sym}({\mathbb S})$-norm}.  
 {Approximating any $y \in L^2_{\rm sym}({\mathbb S})$ by a sequence 
$(y^M \in E^M)_{M \geq 1}$, 
we see that 
the result remains true when the function 
$y \mapsto 
{\mathbb E}[f(X_T^{y^*})]$
is considered on the entire $L^2_{\rm sym}({\mathbb S})$.}

It remains to pass from a Lipschitz function 
$f$ to a merely bounded (measurable) function, but this may be regarded as a consequence of 
standard results in measure theory, see 
for instance
\cite[Lemma 2.2 p.160]{peszatZabczyk1995strongFellerIrredHilbertDiff}.
\color{black}
\qed
\end{proof}

\subsection{Proofs of Lemma \ref{cor:4.5} and Proposition \ref{prop:4.4:rewrite}}
\label{subse:technical}

Throughout, we use the notations introduced in Subsection 
\ref{subse:5.1}. 
We then start with the following lemma {that exposes a subtlety when changing variables} in the integral  \eqref{eq:sec:4:integral:we:want:to:define}.

\begin{lemma}
\label{lem:4.1}
Let
$((n_t^y)_{0 \leq t \leq T})_{y \in E^{*,M}}$
satisfy 
{\rm{({\bf F1})}}--{\rm{({\bf F2})}}
and
 $((z_t^y)_{0 \leq t \leq T})_{y \in E^{*,M}}$ be a deterministic {jointly continuous $L^2_{\rm sym}({\mathbb S})$-valued flow}, i.e., the map 
$(t,y) \in [0,T] \times E^{*,M} \mapsto z_t^y \in L^2_{\rm sym}({\mathbb S})$ is continuous. 
Then, for 
any $\varepsilon>0$,  
for any family $\{r_i^N\}_{i=0,\cdots,N}$ of subdivision points of $[0,T]$ with   
{$ 
\lim_{N \rightarrow \infty} 
\sup_{i=1,\cdots,N} \bigl\vert r_i^N - r_{i-1}^N \bigr\vert = 0,  
$ 
}
the following identity holds true (for convenience we merely write $r_i$ for $r^N_i$):
\begin{align}
&{\mathcal I}_{\varepsilon,\varphi}= \lim_{N\rightarrow\infty}\int_{E^M} \varphi(y )\sum_{i=0}^{N-1}    \biggl[  \Bigl\langle n^{y^*_{r_i}(v,\delta)}_{r_{i+1}},e^{\varepsilon\Delta } z_{r_i}^{y^*_{r_i}(v,\delta)}\Bigr\rangle-  \Bigl\langle n^{y^*_{r_i}(v,\delta)}_{r_{i}},e^{\varepsilon\Delta } z_{r_i}^{y^*_{r_i}(v,\delta)} \Bigr\rangle \biggr]   dy, 
\nonumber
\\
\textrm{with} \  &{\mathcal I}_{\varepsilon,\varphi} := 
\sum_{m \in {\mathbb N}_0}
\int_{E^M} \left\{ \int_0^T  \bigl\langle e^{\varepsilon\Delta }z_t^{y^*},e_m\bigr\rangle \varphi\bigl(y_t(v,-\delta) \bigr)d \bigl\langle n^{y^*}_t,e_m\bigr\rangle  \right\}  dy.
\label{eq:lem:4.1}
\end{align}
\color{black}
\end{lemma}
{We remark that the first superscript in the first line above is not $y^*_{r_{i+1}}(v,\delta)$ as one might expect.}

\begin{proof}
{We let 
$${\mathcal I}_{\varepsilon,\varphi}^m:=\int_{E^M} \left\{ \int_0^T  \bigl\langle e^{\varepsilon\Delta }z_t^{y^*},e_m\bigr\rangle \varphi\bigl(y_t(v,-\delta) \bigr)d \bigl\langle n^{y^*}_t,e_m\bigr\rangle  \right\}  dy.$$}
	By Corollary 
	\ref{cor:RS:3},
	one has for a fixed value of $m \in {\mathbb N}_0$: 
	\begin{equation*}
		\begin{split}
	{{\mathcal I}_{\varepsilon,\varphi}^m }		&=  \int_{E^M}\lim_{N\rightarrow\infty} \sum_{i=0}^{N-1} \bigl\langle e^{\varepsilon\Delta }z_{r_i}^{y^*},e_m \bigr\rangle \varphi\bigl(
			y_{r_i}(v,- \delta)\Bigr)\left[  \bigl\langle n^{{y^*}}_{r_{i+1}},e_m \bigr\rangle-  \langle n^{{y^*}}_{r_{i}},e_m\bigr\rangle \right]   dy \\ 
			&=  \lim_{N\rightarrow\infty}\int_{E^M} \sum_{i=0}^{N-1} \bigl\langle e^{\varepsilon\Delta }z_{r_i}^{y^*},e_m\bigr\rangle \varphi\bigl(y_{r_i}(v,- \delta)\bigr)\left[  \bigl\langle n^{{y^*}}_{r_{i+1}},e_m \bigr\rangle-  \bigl\langle n^{{y^*}}_{r_{i}},e_m \bigr\rangle \right]   dy,
		\end{split}
	\end{equation*}
	the argument for exchanging the limit and the sum following from 
	Lebesgue's dominated convergence theorem. From  
	Lemma \ref{lem:integral:RS:bound:0}, it is indeed clear
	that the sum over $i$ on the second line is uniformly bounded 
	in $N$. 	
	Therefore, performing for each $i \in \{0,\cdots,N-1\}$ an obvious change of variable for the integral in $y$, we get 
	\begin{equation}\notag
		\begin{split}
			{{\mathcal I}_{\varepsilon,\varphi}^m}  &=  \lim_{N\rightarrow\infty}\int_{E^M} \varphi(y )\sum_{i=0}^{N-1} \Bigl\langle e^{\varepsilon\Delta }z_{r_i}^{{y^*_{r_i}(v,\delta)}},e_m\Bigr\rangle \biggl[  
			\Bigl\langle \eta^{{y_{r_i}^*(v,\delta)}}_{r_{i+1}},e_m\Bigr\rangle-  \Bigl\langle \eta^{{y_{r_i}^*(v,\delta)}}_{r_{i}},e_m \Bigr\rangle \biggr]   dy. \\ 
		\end{split} 
	\end{equation}
	In fact, Lemma 
	\ref{lem:integral:RS:bound:0} says more: the argument inside the limit decays polynomially fast with $m$, uniformly in $N$. 
	In particular, summing over $m \in {\mathbb N}_0$, one can exchange the sum over $m$ and the limit over $N$. 
	{Since 
	$\sum_{m \in {\mathbb N}_0} 
	{\mathcal I}_{\varepsilon,\varphi}^m=
	{\mathcal I}_{\varepsilon,\varphi}$,}
	we obtain
	\begin{equation*}
		\begin{split}
%
			{{\mathcal I}_{\varepsilon,\varphi}}  &= 
			\lim_{N\rightarrow\infty}\int_{E^M} \varphi(y )\sum_{i=0}^{N-1} \sum_{m \in {\mathbb N}_0} 
			\Bigl\langle e^{\varepsilon\Delta }z_{r_i}^{{y^*_{r_i}(v,\delta)}},e_m\Bigr\rangle \biggl[  \Bigl\langle \eta^{{y^*_{r_i}(v,\delta)}}_{r_{i+1}},e_m\Bigr\rangle-  
			\Bigl\langle \eta^{{y^*_{r_i}}(v,\delta)}_{r_{i}},e_m \Bigr\rangle \biggr]   dy
			\\
			&=
			\lim_{N\rightarrow\infty}\int_{E^M} \varphi(y )\sum_{i=0}^{N-1} \biggl[ 
			\Bigl\langle e^{\varepsilon\Delta }z_{r_i}^{{y^*_{r_i}(v,\delta)}}, \eta^{y^*_{r_i}(v,\delta)}_{r_{i+1}}
			\Bigr\rangle- 
			\Bigl\langle e^{\varepsilon\Delta }z_{r_i}^{{y_{r_i}^*(v,\delta)}}, \eta^{y^*_{r_i}(v,\delta)}_{r_{i}}
			\Bigr\rangle
			\biggr]   dy,
		\end{split}
	\end{equation*}
	which is the desired result. 
	\qed
\end{proof}

The bulk of the analysis carried out in this subsection is the following statement: 
  \begin{proposition}
  \label{prop:4.4}
Let
$((n_t^y)_{0 \leq t \leq T})_{y \in E^{*,M}}$
satisfy 
{\rm ({\bf F1})}--{\rm ({\bf F2})}, and
 $((z_t^y)_{0 \leq t \leq T})_{y \in E^{*,M}}$ be a deterministic {$L^2_{\rm sym}({\mathbb S})$-valued flow such that $(t,y) \in [0,T] \times E^{*,M} \mapsto z_t^y \in L^2_{\rm sym}({\mathbb S})$ is continuous. 
Then, with the same notation 
as in \eqref{eq:lem:4.1}, for $\varepsilon >0$ and for 
$\varphi \in \cC^\infty_0( E^M)$,}
\begin{equation}
\label{eq:prop:4.4:main:step}
\begin{split}
&{{\mathcal I}_{\varepsilon,\varphi}} 
\\
& =  \lim_{N\rightarrow\infty}\int_{E^M} \varphi(y )\sum_{i=0}^{N-1}     \left[  \biggl\langle \tilde{n}^{{y,(v,\delta)}}_{r_{i+1}},e^{\varepsilon\Delta }
	z_{r_i}^{{y^*_{r_i}(v,\delta)}}
	\biggr\rangle-  \biggl\langle \tilde{n}^{{y,(v,\delta)}}_{r_{i}},e^{\varepsilon\Delta }z_{r_i}^{y_{r_i}^*(v,\delta)} \biggr\rangle \right]dy.    \end{split}
	\end{equation}
  \end{proposition}

\begin{proof}[of Proposition \ref{prop:4.4}]
	{
		It suffices to prove  
		\eqref{eq:prop:4.4:main:step}
		for a flow  $((z_t^{y^*})_{0 \leq t \leq T})_{y \in E^{M}}$, differentiable in $y$, with derivative   jointly continuous in $(t,y)$.  
		Indeed, by a mollification argument in the variable 
		$y$,
		we can approximate any $((z_t^{y^*})_{0 \leq t \leq T})_{y \in E^M}$ 
		that is only continuous in $(t,y)$
		by a flow that is regular in $y$ (with jointly continuous derivatives)
		and use
		\eqref{eq:integral:RS:bound:0:bis}
		in order to pass to the limit in 
		\eqref{eq:prop:4.4:main:step}. We thus assume below that 
		$((z_t^{y^*})_{0 \leq t \leq T})_{y \in E^M}$ 
		is differentiable in $y$, with derivative jointly continuous in $(t,y)$. 
	}
	
	{
		Another key observation is that
		$(t,y) \in [0,T] \times E^M \mapsto n_t^{y^*} \in {H}_{\rm sym}^{-2}({\mathbb S})$
		is jointly continuous in $(t,y)$ and thus uniformly continuous on 
		$[0,T] \times {\rm Supp}(\varphi)$, with 
		${\rm Supp}(\varphi)$ denoting the support of $\varphi$.
		This follows from ({\bf F2}) and the fact that the map $t \in [0,T] \mapsto n_t^{y^*} \in {\mathbb H}_{\rm sym}^{-2}({\mathbb S})$
		is continuous for each $y \in E^M$. By Lemma 
		\ref{lem:4.1}, one has}
	\begin{align*}
		{\mathcal I}_{\varepsilon,\varphi} 
		&=  \lim_{N\rightarrow\infty}\int_{E^M} \varphi(y ) \biggl\{ \sum_{i=0}^{N-1}     \biggl[ 
		\Bigl\langle n^{y_{r_i}^*(v,\delta)}_{r_{i+1}},e^{\varepsilon\Delta }z_{r_i}^{{y^*_{r_i}(v,\delta)}}\Bigr\rangle - 
		\Bigl\langle n^{y_{r_{i+1}}^*(v,\delta)}_{r_{i+1}},e^{\varepsilon\Delta }z_{r_i}^{y^*_{r_{i+1}}(v,\delta)}\Bigr\rangle
		\nonumber 
		\\ 
		&\hspace{80pt} +       \Bigl\langle n^{y^*_{r_{i+1}}(v,\delta)}_{r_{i+1}},e^{\varepsilon\Delta }z_{r_i}^{{y^*_{r_{i+1}}(v,\delta)}}\Bigr\rangle - 
		\Bigl\langle n^{{y_{r_{i+1}}^*(v,\delta)}}_{r_{i+1}},e^{\varepsilon\Delta }z_{r_i}^{{y_{r_i}^*(v,\delta)}}\Bigr\rangle  \nonumber
		\\ 
		&\hspace{80pt} +   		 \Bigl\langle n^{{y^*_{r_{i+1}}(v,\delta)}}_{r_{i+1}},e^{\varepsilon\Delta }z_{r_i}^{{y^*_{r_i}(v,\delta)}}
		\Bigr\rangle-  \Bigl\langle n^{{y_{r_i}^*(v,\delta)}}_{r_{i}},e^{\varepsilon\Delta }z_{r_i}^{y^*_{r_i}(v,\delta)} \Bigr\rangle   \biggr]  
		\biggr\} dy,
		\end{align*}
		{\text{which, by exchanging the first and third lines in the summand, can be rewritten} }	
		\begin{align*}
		{\mathcal I}_{\varepsilon,\varphi}&=  \lim_{N\rightarrow\infty} \biggl\{ \int_{E^M} \varphi(y )\sum_{i=0}^{N-1}     \biggl[  \Bigl\langle n^{y^*_{r_{i+1}}(v,\delta)}_{r_{i+1}},e^{\varepsilon\Delta }
		z^{{y_{r_i}^*(v,\delta)}}_{r_i}\Bigr\rangle-  \Bigl\langle n^{y_{r_i}^*(v,\delta)}_{r_{i}},e^{\varepsilon\Delta }z^{{y^*_{r_i}(v,\delta)}}_{r_i}\Bigr\rangle \biggr]   dy  \nonumber
		\\ 
		&\hspace{45pt} +     \int_{E^M} \varphi(y )\sum_{i=0}^{N-1}   \Bigl\langle n^{y^*_{r_{i+1}}(v,\delta)}_{r_{i+1}},e^{\varepsilon\Delta }
		\Bigl( z_{r_i}^{{y^*_{r_{i+1}}(v,\delta)}}
		- 
		z_{r_i}^{{y_{r_i}^*(v,\delta)}} \Bigr) \Bigr\rangle dy  \nonumber
		\\ 
		&\hspace{45pt} +\sum_{i=0}^{N-1} \int_{E^M}   \left[ \varphi(y )-\varphi\Bigl(y+\delta\tfrac{r_{i+1}-{r_{i}}}{T} v\Bigr) \right]    \Bigl\langle 
		n^{{y^*_{r_i}(v,\delta)}}_{r_{i+1}},e^{\varepsilon\Delta }z_{r_i}^{y_{r_i}^*(v,\delta)}\Bigr\rangle  dy \biggr\}  \nonumber
		\\
			&
		=:  \lim_{N\rightarrow\infty} \Bigl\{ T_1^N + T_2^N + T_3^N \Bigr\}.	\nonumber
		\label{eq:proof:tilde:eta:4:007}
	\end{align*}

	\noindent \textit{Analysis of $T_1^N$.}
	By applying Definition 
	\ref{def:4.2} (at point $y_t(v,\delta)$ instead of $y$)
	{and by using the fact that 
		$(\tilde n_t^y)_{0 \leq t \leq T}$ takes values in 
		$H^{-2}_{\rm sym}({\mathbb S})$}, we get 
	\begin{equation}\notag
		\begin{split}
			& T_1^N
			=   \int_{E^M} \varphi(y )\sum_{i=0}^{N-1}     \left[  \Bigl\langle \tilde{n}^{{y,(v,\delta)}}_{r_{i+1}},e^{\varepsilon\Delta }
			z_{r_i}^{y_{r_i}^*(v,\delta)}
			\Bigr\rangle-  \Bigl\langle \tilde{n}^{{y,(v,\delta)}}_{r_{i}},e^{\varepsilon\Delta }z_{r_i}^{y_{r_i}^*(v,\delta)}\Bigr\rangle \right]dy 
			\\ 
			&\hspace{2pt}  -  \tfrac{\delta}{T}   \int_{E^M} \varphi(y )\sum_{i=0}^{N-1} \sum_{m \in {\mathbb N}_0}  \left[ \int_{r_i}^{r_{i+1}} 
			\Bigl(
			\partial_w \Bigl[ \Bigl\langle n^{w^*}_s,  e_m \Bigr\rangle \Bigr]_{\vert w=y_s(v,\delta)}  \cdot v \Bigr) \, \Bigl\langle  z_{r_i}^{y_{r_i}^*(v,\delta)} ,e^{\varepsilon\Delta }e_m\Bigr\rangle  \, ds	  \right]  dy
			\\
			&{=   \int_{E^M} \varphi(y )\sum_{i=0}^{N-1}     \left[  \Bigl\langle \tilde{n}^{{y,(v,\delta)}}_{r_{i+1}},e^{\varepsilon\Delta }
				z_{r_i}^{y_{r_i}^*(v,\delta)}
				\Bigr\rangle-  \Bigl\langle \tilde{n}^{{y,(v,\delta)}}_{r_{i}},e^{\varepsilon\Delta }z_{r_i}^{y_{r_i}^*(v,\delta)}\Bigr\rangle \right]dy}
			\\ 
			&\hspace{2pt}  {-  \tfrac{\delta}{T}    \int_{E^M} \varphi(y )\sum_{i=0}^{N-1} \sum_{m \in {\mathbb N}_0}  \left[ \int_{r_i}^{r_{i+1}} 
				\Bigl(
				\partial_y \Bigl[ \Bigl\langle n^{y_s^*(v,\delta)}_s,  e_m \Bigr\rangle \Bigr]  \cdot v \Bigr) \, \Bigl\langle  z_{r_i}^{y_{r_i}^*(v,\delta)} ,e^{\varepsilon\Delta }e_m\Bigr\rangle  \, ds	  \right]  dy.}
		\end{split}
	\end{equation}
	{Exchanging the integral in $y$ and the sum over $m$ (which is possible thanks to
		Lemma \ref{lem:integral:RS:bound:0})} and then performing an integration by parts in the last line, we obtain
	\begin{align} 
		 &T_1^N \nonumber	
		=  \int_{E^M} \varphi(y )\sum_{i=0}^{N-1}  
		\left[  \Bigl\langle \tilde{n}^{{{y,(v,\delta)}}}_{r_{i+1}},e^{\varepsilon\Delta }
		z_{r_i}^{y_{r_i}^*(v,\delta)}
		\Bigr\rangle-  \Bigl\langle \tilde{n}^{{y,(v,\delta)}}_{r_{i}},e^{\varepsilon\Delta }z_{r_i}^{y_{r_i}^*(v,\delta)}\Bigr\rangle \right]
		dy \nonumber
		\\ 
		&\hspace{2pt} +   \tfrac{\delta}{T} \sum_{i=0}^{N-1}  \int_{r_i}^{r_{i+1}} \int_{E^M} \Bigl( \partial_y \varphi(y )\cdot v \Bigr)   \Bigl\langle n^{
			y_s^*(v,\delta)}_s,  e^{\varepsilon\Delta }z_{r_i}^{{y^*_{r_i}(v,\delta)}} \Bigr\rangle  	     dy ds \nonumber
		\\	
		&\hspace{2pt} +
		\tfrac{\delta}T
		\int_{E^M} \varphi(y )\sum_{i=0}^{N-1} \sum_{m \in {\mathbb N}_0}  \left[ \int_{r_i}^{r_{i+1}} 
		\Bigl\langle n^{y^*_s(v,\delta)}_s, e_m \Bigr\rangle
		\Bigl( \partial_y 
		{ \Bigl[ \Bigl\langle  z_{r_i}^{y^*_{r_i}(v,\delta)} ,e^{\varepsilon\Delta } e_m\Bigr\rangle \Bigr]} \cdot v \Bigr)  \, ds	  \right]  dy \nonumber
		\\
		&=:  T_{1,1}^N +  T_{1,2}^N + T_{1,3}^N.
		\label{eq:T11+T12+T13}
	\end{align}
	
	\noindent \textit{Analysis of $T_{1,3}^N+T_2^N$.}
	Using the regularity of the flow 
	$((z_t^y)_{0 \leq t \leq T})_{y \in E^{*,M}}$, 
	we write
	\begin{equation*}
		\begin{split}
			&T_2^N =     \int_{E^M} \varphi(y )\sum_{i=0}^{N-1}   \Bigl\langle n^{y^*_{r_{i+1}}(v,\delta)}_{r_{i+1}},e^{\varepsilon\Delta }
			\Bigl( z_{r_i}^{{y^*_{r_{i+1}}(v,\delta)}}
			- 
			z_{r_i}^{{y_{r_i}^*(v,\delta)}} \Bigr) \Bigr\rangle dy 
			\\
			&=-  \tfrac{\delta}T
			\int_{E^M} \varphi(y )\sum_{i=0}^{N-1} \bigl(r_{i+1} - r_i\bigr)  
			\sum_{m \in {\mathbb N}_0}
			\Bigl\langle n^{{y_{r_i+1}^*(v,\delta)}}_{r_{i+1}},e_m 
			\Bigr\rangle
			\Bigl(	 { \partial_y    \Bigl[ \Bigl\langle  z_{r_i}^{y_{r_i}^*(v,\delta)} ,    e^{\varepsilon\Delta } e_m \Bigr\rangle \Bigr] }\cdot v \Bigr) dy  
			\\
			&\hspace{15pt} 
			+
			\sum_{i=0}^{N-1}  \oh (r_{i+1}-r_i),
		\end{split}
	\end{equation*}
	where $\oh$ is the  little Landau symbol (and is here implicitly understood to be uniform in 
	$N$ and $i$). 
	And, then using the joint regularity of 
	$((n_t^y)_{0 \leq t \leq T})_{y \in E^{*,M}}$, 
	we get 
	\begin{equation*}
		\begin{split}
			\lim_{N \rightarrow \infty}
			\bigl( T_2^N + T_{1,3}^N \bigr)&  =  \lim_{N\rightarrow\infty}\int_{E^M} \varphi(y )\sum_{i=0}^{N-1} \sum_{m \in {\mathbb N}_0}  
			\biggl[ \int_{r_i}^{r_{i+1}} 
			\Bigl\langle n^{y_s^*(v,\delta)}_s
			-
			n_{r_{i+1}}^{y_{r_{i+1}}^*(v,\delta)}	  
			, e_m \Bigr\rangle 
			\\
			&\hspace{100pt} \times	  \Bigl({ \partial_y \Bigl[ \Bigl\langle  z_{r_i}^{y_{r_i}^*(v,\delta)} ,e^{\varepsilon\Delta } e_m\Bigr\rangle \Bigr]} \cdot v \Bigr)  \, ds	  \biggr]  dy
			=0.
		\end{split}
	\end{equation*}

	\noindent \textit{Analysis of $T_{1,2}^N+T_3^N$.}
	Adding and subtracting 
	the quantity 
	$$  \tfrac{\delta}{T} \sum_{i=0}^{N-1} \int_{E^M}   \bigl( \partial_y \varphi(y)\cdot v \bigr) \bigl(r_{i+1}-r_i\bigr)   \Bigl\langle 
	n^{y_{r_i}^*(v,\delta)}_{r_{i+1}},e^{\varepsilon\Delta }z_{r_i}^{y_{r_i}^*(v,\delta)} \Bigr\rangle  dy,$$
	we 
	have  
	\begin{align}
		& T_{1,2}^N + T_3^N  \nonumber
		\\   
		&=  \sum_{i=0}^{N-1}  \tfrac{\delta}{T} \int_{r_i}^{r_{i+1}} \hspace{-10pt}\int_{E^M} \Bigl( \partial_y \varphi(y )\cdot v  \Bigr)  \biggl[ 
		\Bigl\langle 
		n_s^{y^*_s(v,\delta)},  e^{\varepsilon\Delta }z^{{y_{r_i}^*(v,\delta)}}_{r_i} \Bigr\rangle
		-
		\Bigl\langle n^{{y_{r_i}^*(\delta,v)}}_{r_{i+1}},e^{\varepsilon\Delta }
		z^{y_{r_i}^*(v,\delta)}_{r_i}\Bigr\rangle 
		\biggr]  dy ds \nonumber
		\\ 
		&\hspace{15pt} + \sum_{i=0}^{N-1}  \int_{E^M} \biggl\{    \biggl[ \varphi(y)+ 
		\Bigl( \tfrac{\delta}{T}( {r_{i+1}}-r_i) \Bigr)
		\partial_y \varphi(y)\cdot v  \nonumber
		-\varphi\Bigl(y+\tfrac{\delta}{T}(r_{i+1}-r_i)v\Bigr) \biggr]   
		\\
		&\hspace{30pt} \times 
		 \Bigl\langle n^{y_{r_i}^*(v,\delta)}_{r_{i+1}},e^{\varepsilon\Delta }
		z_{r_i}^{{y_{r_i}^*(v,\delta)}} \Bigr\rangle  \biggr\} dy\nonumber
		\\
		&=:   T_{(1,2,3),1}^N + T_{(1,2,3),2}^N.
		\label{eq:proof:tilde:eta:4}
	\end{align}
	
	\noindent	\textit{Analysis of $T_{(1,2,3),1}^N+T_{(1,2,3),2}^N$.}
	We claim that the limits of the two terms in the above argument are $0$ as 
	we can write 
	both of them in the form 
	$\sum_{i=0}^{N-1}  \oh (r_{i+1}-r_i)$.

The limit of $(T^N_{(1,2,3),2})_{N \geq 1}$ is easily handled by using the fact that $\varphi$ is smooth and by 
{invoking the duality between $H^2_{\rm sym}({\mathbb S})$ and 
	$H^{-2}_{\rm sym}({\mathbb S})$  
	(to handle terms of the form 
	$\langle n^{{y_{r_i}^*(v,\delta)}}_{r_{i+1}},e^{\varepsilon\Delta }z_{r_i}^{{y^*_{r_i}(v,\delta)}}\rangle$)}. 

The limit of $(T^N_{(1,2,3),1})_{N \geq 1}$
is shown to be $0$ by invoking 
the fact that 
the mapping 
$(t,y) \in [0,T] \times E^M \mapsto \eta_t^y \in {\mathbb H}^{-2}_{\rm sym}({\mathbb S})$
is jointly continuous in $(t,y)$ and thus uniformly continuous on 
$[0,T] \times {\rm Supp}(\varphi)$. 
\vskip 5pt

\noindent \textit{Conclusion.}
Back to 
\eqref{eq:proof:tilde:eta:4}, we deduce from the above analysis that 
\begin{equation*}
	\lim_{N\rightarrow\infty} \Bigl\{ T_1^N + T_2^N + T_3^N \Bigr\} 
	=
	\lim_{N\rightarrow\infty} \Bigl\{ T_{1,1}^N + T_{1,2}^N + T_{1,3}^N + T_2^N + T_3^N \Bigr\} 
	=
	\lim_{N\rightarrow\infty} T_{1,1}^N.
\end{equation*} 
This completes the proof. 
\qed
\end{proof}

We now apply Proposition 
  \ref{prop:4.4}
  to the proofs
  of Lemma \ref{cor:4.5} and Proposition \ref{prop:4.4:rewrite}. 
  In order to do so, we assume that the flow $((z_t^y)_{0 \leq t \le T})_{y \in E^{*,M}}$ 
in
Proposition 
  \ref{prop:4.4}
reduces to one single trajectory 
$(z_t)_{0 \leq t \leq T}$. 
Recalling the notation
\eqref{eq:tilde:z:varphi}
and following the derivation of 
  \eqref{eq:item:4:shifted:tilted}, 
 we observe that 
 the left-hand side in 
 \eqref{eq:prop:4.4:main:step} 
 (whose explicit form is given in 
 \eqref{eq:lem:4.1})
 can be rewritten as
%
\begin{equation}
\label{eq:prop:4.4:main:step:LHS}
\begin{split}
&{{\mathcal I}_{\varepsilon,\varphi}}
		= 
		\int_{E^M} 
\biggl(		\int_0^T 
		\bigl\langle e^{\varepsilon\Delta }\tilde z_t^{y,\varphi,{(v,\delta)}} , d  n^{{y^*}}_t \bigr\rangle   \biggr) dy.
		\end{split}
		\end{equation}
In order to handle the right-hand side of 
\eqref{eq:prop:4.4:main:step}, we 
recall 
\eqref{eq:tilde:eta:varphi}. We observe that the argument in the limit appearing in the right-hand side of 
\eqref{eq:prop:4.4:main:step} can be rewritten
\begin{equation}
\label{eq:integral:varphi}
\begin{split}
&\int_{E^M} \varphi(y )\sum_{i=0}^{N-1}     \left[  \Bigl\langle \tilde{n}^{{{y,(v,\delta)}}}_{r_{i+1}},e^{\varepsilon\Delta }z_{r_i}\rangle-  \langle \tilde{n}^{{{y,(v,\delta)}}}_{r_{i}},e^{\varepsilon\Delta }z_{r_i} \Bigr\rangle \right]dy
\\
&= \sum_{i=0}^{N-1}
\biggl[ \Bigl\langle 
\tilde n_{r_{i+1}}^{\varphi,{(v,\delta)}},
e^{\varepsilon\Delta }z_{r_i}
\Bigr\rangle
-
\Bigl\langle 
\tilde n_{r_{i}}^{\varphi,{(v,\delta)}},
e^{\varepsilon\Delta }z_{r_i}
\Bigr\rangle \biggr].
\end{split}
\end{equation}

\begin{proof}[of Lemma \ref{cor:4.5} and Proposition \ref{prop:4.4:rewrite}]
With $z \in U^2({\mathbb S})$, we apply 
Proposition 
\ref{prop:4.4}
with $(z_t:={\mathbf 1}_{[r,s]}(t) z)_{0 \leq t \leq T}$ for 
a given pair $(r,s) \in [0,T]^2$ satisfying $r<s$. 
By 
\eqref{eq:tilde:z:varphi}, 
\eqref{eq:prop:4.4:main:step:LHS}
and
\eqref{eq:integral:varphi}, we obtain
\begin{equation}
	\label{eq:cor:4.5:1}
	\int_{E^M} 
	\biggl(		\int_0^T 
	\Bigl\langle e^{\varepsilon\Delta }\tilde z_t^{y,\varphi,{(v,\delta)}} , d  n^{{y}^*}_t \Bigr\rangle   \biggr) dy
	=
	\lim_{N \rightarrow \infty} 
	\sum_{i=0}^{N-1}
	\biggl[ \Bigl\langle 
	\tilde n_{r_{i+1}}^{\varphi,{(v,\delta)}},
	e^{\varepsilon\Delta }z_{r_i}
	\Bigr\rangle
	-
	\Bigl\langle 
	\tilde n_{r_{i}}^{\varphi,{(v,\delta)}},
	e^{\varepsilon\Delta }z_{r_i}
	\Bigr\rangle \biggr].
\end{equation}
Regardless the choice of the subdivision 
$\{r^N_i\}_{i=0,\cdots,N}$ (recall that we omit the superscript $N$ in the various equations), we 
have
\begin{equation*}
	\begin{split}
		&\sum_{i=0}^{N-1}
		\biggl[ \Bigl\langle 
		\tilde n_{r_{i+1}}^{\varphi,{(v,\delta)}},
		e^{\varepsilon\Delta }z_{r_i}
		\Bigr\rangle
		-
		\Bigl\langle 
		\tilde n_{r_{i}}^{\varphi,{(v,\delta)}},
		e^{\varepsilon\Delta }z_{r_i}
		\Bigr\rangle \biggr]
		\\
		&
		= 
		\sum_{i=0}^{N-1}
		\biggl\{
		\biggl[ \Bigl\langle 
		\tilde n_{r_{i+1}}^{\varphi,{(v,\delta)}},
		e^{\varepsilon\Delta }z
		\Bigr\rangle
		-
		\Bigl\langle 
		\tilde n_{r_{i}}^{\varphi,{(v,\delta)}},
		e^{\varepsilon\Delta }z
		\Bigr\rangle \biggr]
		{\mathbf 1}_{[r,s]}\bigl(r_i\bigr)
		\biggr\}.
	\end{split}
\end{equation*}
If we assume that $r$ and $s$ belong to the collection $\{r_i^N\}_{i=0,\cdots,N}$, which can be done 
without any loss of generality, we get 
\begin{equation*}
	\begin{split}
		\sum_{i=0}^{N-1}
		\biggl[ \Bigl\langle 
		\tilde n_{r_{i+1}}^{\varphi,{(v,\delta)}},
		e^{\varepsilon\Delta }z_{r_i}
		\Bigr\rangle
		-
		\Bigl\langle 
		\tilde n_{r_{i}}^{\varphi,{(v,\delta)}},
		e^{\varepsilon\Delta }z_{r_i}
		\Bigr\rangle \biggr]
		&=  \Bigl\langle 
		\tilde n_{s}^{\varphi,{(v,\delta)}},
		e^{\varepsilon\Delta }z
		\Bigr\rangle
		-
		\Bigl\langle 
		\tilde n_{r}^{\varphi,{(v,\delta)}},
		e^{\varepsilon\Delta }z
		\Bigr\rangle.
	\end{split}
\end{equation*} 
Then, \eqref{eq:cor:4.5:1}
yields
\begin{equation*}
	\Bigl\langle 
	\tilde n_{s}^{\varphi,{(v,\delta)}},
	e^{\varepsilon\Delta }{z}
	\Bigr\rangle
	-
	\Bigl\langle 
	\tilde n_{r}^{\varphi,{(v,\delta)}},
	e^{\varepsilon\Delta }{z}
	\Bigr\rangle
	=
	\int_{E^M} 
	\biggl(		\int_0^T 
	\Bigl\langle e^{\varepsilon\Delta }\tilde z_t^{y,\varphi,{(v,\delta)}} , d  n^{{y^*}}_t \Bigr\rangle   \biggr) dy.
\end{equation*}
Recalling the definition 
\eqref{eq:tilde:z:varphi}
and using the 
fact that 
$\varphi \geq 0$, 
we observe 
that 
$(\tilde z_t^{y,\varphi,{(v,\delta)}})_{0 \leq t \leq T}$
takes values in $U^2({\mathbb S})$. 
By Corollary 
\ref{cor:RS:3}, 
the right-hand side is non-negative.
Assuming that $z \in H^2_{\rm sym}({\mathbb S})$ and letting 
$\varepsilon$ tend to $0$, we complete the proof
of 
Lemma \ref{cor:4.5}.
In turn, this permits us to invoke 
Lemma \ref{lem:RS:3}
to
let $N$ tend to $\infty$
in the right-hand side of 
\eqref{eq:integral:varphi}. This gives the right-hand side in 
\eqref{eq:rewriting:propr:4.4}. 
As for the left-hand side in 
\eqref{eq:integral:varphi}, the limit is given by 
Proposition 
  \ref{prop:4.4} and identifies with 
  the right-hand side of 
  \eqref{eq:prop:4.4:main:step:LHS}. 
  This gives the left-hand side in 
\eqref{eq:rewriting:propr:4.4}
  and
  proves 
  Proposition 
    \ref{prop:4.4:rewrite}. 
\qed
\end{proof}

\appendix
 
\noindent {\bf Acknowledgement.} 
We
are very grateful to the two anonymous referees for their valuable comments and suggestions which 
clearly helped us to improve the article.
 

 
\bibliographystyle{spbasic} 
\bibliography{FullBibliographyWRPH1p}

\begin{thebibliography}{78}
\providecommand{\natexlab}[1]{#1}
\providecommand{\url}[1]{{#1}}
\providecommand{\urlprefix}{URL }
\expandafter\ifx\csname urlstyle\endcsname\relax
  \providecommand{\doi}[1]{DOI~\discretionary{}{}{}#1}\else
  \providecommand{\doi}{DOI~\discretionary{}{}{}\begingroup
  \urlstyle{rm}\Url}\fi
\providecommand{\eprint}[2][]{\url{#2}}

\bibitem[{Ambrosio et~al(2005)Ambrosio, Gigli, and Savar\'{e}}]{AGS}
Ambrosio L, Gigli N, Savar\'{e} G (2005) Gradient flows in metric spaces and in
  the space of probability measures. Lectures in Mathematics ETH Z\"{u}rich,
  Birkh\"{a}user Verlag, Basel

\bibitem[{Andersson(2013)}]{Andersson}
Andersson D (2013) Estimates of the spherical and ultraspherical heat kernel.
  master thesis. Department of Mathematical Sciences, Chalmers University of
  Technology and Göteborg University, Sweden
  \eprint{https://hdl.handle.net/20.500.12380/182086}

\bibitem[{Andres(2009)}]{Andres1}
Andres S (2009) Pathwise differentiability for {SDE}s in a convex polyhedron
  with oblique reflection. Ann Inst Henri Poincar\'{e} Probab Stat
  45(1):104--116, \urlprefix\url{https://doi.org/10.1214/07-AIHP151}

\bibitem[{Andres(2011)}]{Andres2}
Andres S (2011) Pathwise differentiability for {SDE}s in a smooth domain with
  reflection. Electron J Probab 16:no. 28, 845--879,
  \urlprefix\url{https://doi.org/10.1214/EJP.v16-872}

\bibitem[{Andres and von Renesse(2012)}]{AndresvRenesse}
Andres S, von Renesse MK (2012) Uniqueness and regularity for a system of
  interacting {B}essel processes via the {M}uckenhoupt condition. Trans Amer
  Math Soc 364(3):1413--1426,
  \urlprefix\url{https://doi.org/10.1090/S0002-9947-2011-05457-7}

\bibitem[{Baernstein(1995)}]{Baernstein_correction}
Baernstein A II (1995) Correction to: ``{C}onvolution and rearrangement on the
  circle'' [{C}omplex {V}ariables {T}heory {A}ppl. {\bf 12} (1989), no. 1-4,
  33--37]. Complex Variables Theory Appl 26(4):381--382,
  \urlprefix\url{https://doi.org/10.1080/17476939508814799}

\bibitem[{Baernstein~II(1989)}]{baernstein1989convRearrOnCirc}
Baernstein~II A (1989) Convolution and rearrangement on circle. Complex
  Variables, Theory and Application: An International Journal 12(1-4):33--37,
  \eprint{https://doi.org/10.1080/17476938908814351}

\bibitem[{Baernstein~II(2019)}]{baernstein2019symmetrizationInAnalysis}
Baernstein~II A (2019) Symmetrization in Analysis. New Mathematical Monographs,
  Cambridge University Press, \doi{10.1017/9781139020244}

\bibitem[{Barbu et~al(2009)Barbu, Da~Prato, and Tubaro}]{Barbu3}
Barbu V, Da~Prato G, Tubaro L (2009) Kolmogorov equation associated to the
  stochastic reflection problem on a smooth convex set of a {H}ilbert space.
  Ann Probab 37(4):1427--1458,
  \urlprefix\url{https://doi.org/10.1214/08-AOP438}

\bibitem[{Barbu et~al(2011)Barbu, Da~Prato, and Tubaro}]{Barbu2}
Barbu V, Da~Prato G, Tubaro L (2011) Kolmogorov equation associated to the
  stochastic reflection problem on a smooth convex set of a {H}ilbert space
  {II}. Ann Inst Henri Poincar\'{e} Probab Stat 47(3):699--724,
  \urlprefix\url{https://doi.org/10.1214/10-AIHP381}

\bibitem[{Barbu et~al(2012)Barbu, Da~Prato, and Tubaro}]{Barbu1}
Barbu V, Da~Prato G, Tubaro L (2012) The stochastic reflection problem in
  {H}ilbert spaces. Comm Partial Differential Equations 37(2):352--367,
  \urlprefix\url{https://doi.org/10.1080/03605302.2011.596878}

\bibitem[{Bismut(1981)}]{Bismut}
Bismut JM (1981) Martingales, the {M}alliavin calculus and hypoellipticity
  under general {H}\"{o}rmander's conditions. Z Wahrsch Verw Gebiete
  56(4):469--505, \urlprefix\url{https://doi.org/10.1007/BF00531428}

\bibitem[{Bongiorno(2017)}]{Bongiorno}
Bongiorno D (2017) Absolutely continuous functions with values in a {B}anach
  space. J Math Anal Appl 451(2):1216--1223,
  \urlprefix\url{https://doi.org/10.1016/j.jmaa.2017.02.067}

\bibitem[{Brenier(1991)}]{BrenierPolar}
Brenier Y (1991) Polar factorization and monotone rearrangement of
  vector-valued functions. Comm Pure Appl Math 44(4):375--417,
  \urlprefix\url{https://doi.org/10.1002/cpa.3160440402}

\bibitem[{Brenier(2004)}]{Brenier2}
Brenier Y (2004) Order preserving vibrating strings and applications to
  electrodynamics and magnetohydrodynamics. Methods Appl Anal 11(4):515--532,
  \urlprefix\url{http://projecteuclid.org/euclid.maa/1144939945}

\bibitem[{Brenier(2009)}]{Brenier1}
Brenier Y (2009) {$L^2$} formulation of multidimensional scalar conservation
  laws. Arch Ration Mech Anal 193(1):1--19,
  \urlprefix\url{https://doi.org/10.1007/s00205-009-0214-0}

\bibitem[{Cardaliaguet and Souganidis(2022)}]{CardaliaguetSouganidis}
Cardaliaguet P, Souganidis P (2022) Regularity of the value function and
  quantitative propagation of chaos for mean field control problems. arXiv
  \doi{10.48550/ARXIV.2204.01314}

\bibitem[{Cardaliaguet et~al(2019)Cardaliaguet, Delarue, Lasry, and
  Lions}]{CDLL}
Cardaliaguet P, Delarue F, Lasry JM, Lions PL (2019) The master equation and
  the convergence problem in mean field games, Annals of Mathematics Studies,
  vol 201. Princeton University Press, Princeton, NJ,
  \urlprefix\url{https://doi.org/10.2307/j.ctvckq7qf}

\bibitem[{Carmona and Delarue(2018)}]{CarmonaDelarueII}
Carmona R, Delarue F (2018) Probabilistic theory of mean field games with
  applications. {II}, Probability Theory and Stochastic Modelling, vol~84.
  Springer, Cham, mean field games with common noise and master equations

\bibitem[{Cecchin and Delarue(2022)}]{CecchinDelarue}
Cecchin A, Delarue F (2022) Weak solutions to the master equation of potential
  mean field games. arXiv \doi{10.48550/ARXIV.2204.04315}

\bibitem[{Cerrai(2001)}]{cerrai}
Cerrai S (2001) Second order {PDE}'s in finite and infinite dimension, Lecture
  Notes in Mathematics, vol 1762. Springer-Verlag, Berlin,
  \urlprefix\url{https://doi.org/10.1007/b80743}, a probabilistic approach

\bibitem[{Da~Prato(2006)}]{daprato}
Da~Prato G (2006) An introduction to infinite-dimensional analysis.
  Universitext, Springer-Verlag, Berlin,
  \urlprefix\url{https://doi.org/10.1007/3-540-29021-4}

\bibitem[{Da~Prato and Zabczyk(2014)}]{daPratoZabczyk2014stochEqnsInfDim}
Da~Prato G, Zabczyk J (2014) Stochastic Equations in Infinite Dimensions, 2nd
  edn. Encyclopedia of Mathematics and its Applications, Cambridge University
  Press, \doi{10.1017/CBO9781107295513}

\bibitem[{Da~Prato et~al(1995)Da~Prato, Elworthy, and Zabczyk}]{DPEZ}
Da~Prato G, Elworthy KD, Zabczyk J (1995) Strong {F}eller property for
  stochastic semilinear equations. Stochastic Anal Appl 13(1):35--45,
  \urlprefix\url{https://doi.org/10.1080/07362999508809381}

\bibitem[{Dawson and Vaillancourt(1995)}]{dawsonVaillancourt1995}
Dawson D, Vaillancourt J (1995) Stochastic {M}c{K}ean-{V}lasov equations. NoDEA
  Nonlinear Differential Equations Appl 2(2):199--229,
  \urlprefix\url{https://doi.org/10.1007/BF01295311}

\bibitem[{Dawson and March(1995)}]{DawsonMarch}
Dawson DA, March P (1995) Resolvent estimates for {F}leming-{V}iot operators
  and uniqueness of solutions to related martingale problems. J Funct Anal
  132(2):417--472, \urlprefix\url{https://doi.org/10.1006/jfan.1995.1111}

\bibitem[{Dawson et~al(1993)Dawson, Maisonneuve, and Spencer}]{Dawson}
Dawson DA, Maisonneuve B, Spencer J (1993) \'{E}cole d'\'{E}t\'{e} de
  {P}robabilit\'{e}s de {S}aint-{F}lour {XXI}---1991, Lecture Notes in
  Mathematics, vol 1541. Springer-Verlag, Berlin,
  \urlprefix\url{https://doi.org/10.1007/BFb0084189}, papers from the school
  held in Saint-Flour, August 18--September 4, 1991, Edited by P. L. Hennequin

\bibitem[{Dello~Schiavo(2022)}]{DelloSchiavo0}
Dello~Schiavo L (2022) The {D}irichlet-{F}erguson diffusion on the space of
  probability measures over a closed {R}iemannian manifold. Ann Probab
  50(2):591--648, \urlprefix\url{https://doi.org/10.1214/21-aop1541}

\bibitem[{Deuschel and
  Zambotti(2005)}]{deuschelZambotti2005bismutElworthySDErefl}
Deuschel JD, Zambotti L (2005) Bismut–elworthy's formula and random walk
  representation for sdes with reflection. Stochastic Processes and their
  Applications 115(6):907--925,
  \urlprefix\url{https://www.sciencedirect.com/science/article/pii/S0304414905000128}

\bibitem[{Ding(2022)}]{Ding}
Ding H (2022) A new particle approximation to the diffusive dean-kawasaki
  equation with colored noise. arXiv 2204.11309,
  \doi{10.48550/ARXIV.2204.11309}

\bibitem[{Donati-Martin and Pardoux(1993)}]{Donati-Martin}
Donati-Martin C, Pardoux E (1993) White noise driven {SPDE}s with reflection.
  Probab Theory Related Fields 95(1):1--24,
  \urlprefix\url{https://doi.org/10.1007/BF01197335}

\bibitem[{D\"{o}ring and Stannat(2009)}]{DoringvRenesse}
D\"{o}ring M, Stannat W (2009) The logarithmic {S}obolev inequality for the
  {W}asserstein diffusion. Probab Theory Related Fields 145(1-2):189--209,
  \urlprefix\url{https://doi.org/10.1007/s00440-008-0166-6}

\bibitem[{Dym and McKean(1972)}]{dymMcKean1972bookFourierSeriesAndInts}
Dym H, McKean HP (1972) Fourier series and integrals. Academic Press New York

\bibitem[{Elworthy(1992)}]{Elworthy92}
Elworthy KD (1992) Stochastic flows on {R}iemannian manifolds. In: Diffusion
  processes and related problems in analysis, {V}ol. {II} ({C}harlotte, {NC},
  1990), Progr. Probab., vol~27, Birkh\"{a}user Boston, Boston, MA, pp 37--72

\bibitem[{Elworthy and Li(1994)}]{ElworthyLi}
Elworthy KD, Li XM (1994) Formulae for the derivatives of heat semigroups. J
  Funct Anal 125(1):252--286,
  \urlprefix\url{https://doi.org/10.1006/jfan.1994.1124}

\bibitem[{Gangbo and Tudorascu(2019)}]{gangbo}
Gangbo W, Tudorascu A (2019) On differentiability in the {W}asserstein space
  and well-posedness for {H}amilton-{J}acobi equations. J Math Pures Appl (9)
  125:119--174, \urlprefix\url{https://doi.org/10.1016/j.matpur.2018.09.003}

\bibitem[{Gangbo et~al(2021)Gangbo, Mayorga, and Swiech}]{GangboMayorgaSwiech}
Gangbo W, Mayorga S, Swiech A (2021) Finite dimensional approximations of
  {H}amilton-{J}acobi-{B}ellman equations in spaces of probability measures.
  SIAM J Math Anal 53(2):1320--1356,
  \urlprefix\url{https://doi.org/10.1137/20M1331135}

\bibitem[{Hambly and Ledger(2017)}]{hambly-ledger}
Hambly B, Ledger S (2017) A stochastic mckean–vlasov equation for absorbing
  diffusions on the half-line. The Annals of Applied Probability
  27(5):2698--2752, \urlprefix\url{http://www.jstor.org/stable/26361426}

\bibitem[{Jordan et~al(1998)Jordan, Kinderlehrer, and Otto}]{JKO}
Jordan R, Kinderlehrer D, Otto F (1998) The variational formulation of the
  {F}okker-{P}lanck equation. SIAM J Math Anal 29(1):1--17,
  \urlprefix\url{https://doi.org/10.1137/S0036141096303359}

\bibitem[{Kac(1956)}]{Kac56}
Kac M (1956) Foundations of kinetic theory. In: Proceedings of the {T}hird
  {B}erkeley {S}ymposium on {M}athematical {S}tatistics and {P}robability,
  1954--1955, vol. {III}, University of California Press, Berkeley and Los
  Angeles, pp 171--197

\bibitem[{Kolokoltsov(2010)}]{Kolokoltsov_book}
Kolokoltsov VN (2010) Nonlinear {M}arkov processes and kinetic equations,
  Cambridge Tracts in Mathematics, vol 182. Cambridge University Press,
  Cambridge, \urlprefix\url{https://doi.org/10.1017/CBO9780511760303}

\bibitem[{Konarovskyi(2017{\natexlab{a}})}]{Konarovskyi}
Konarovskyi V (2017{\natexlab{a}}) On asymptotic behavior of the modified
  {A}rratia flow. Electron J Probab 22:Paper No. 19, 31,
  \urlprefix\url{https://doi.org/10.1214/17-EJP34}

\bibitem[{Konarovskyi(2017{\natexlab{b}})}]{Konarovsky0}
Konarovskyi V (2017{\natexlab{b}}) A system of coalescing heavy diffusion
  particles on the real line. Ann Probab 45(5):3293--3335,
  \urlprefix\url{https://doi.org/10.1214/16-AOP1137}

\bibitem[{Konarovskyi(2020)}]{Konarovskyi2}
Konarovskyi V (2020) On number of particles in coalescing-fragmentating
  {W}asserstein dynamics. Theory Stoch Process 25(2):74--80

\bibitem[{Konarovskyi and von Renesse(2019)}]{KonarovskyivRenesse}
Konarovskyi V, von Renesse MK (2019) Modified massive {A}rratia flow and
  {W}asserstein diffusion. Comm Pure Appl Math 72(4):764--800,
  \urlprefix\url{https://doi.org/10.1002/cpa.21758}

\bibitem[{Konarovskyi et~al(2019)Konarovskyi, Lehmann, and von
  Renesse}]{KonarovskyivRenesse3}
Konarovskyi V, Lehmann T, von Renesse MK (2019) Dean-{K}awasaki dynamics:
  ill-posedness vs. triviality. Electron Commun Probab 24:Paper No. 8, 9,
  \urlprefix\url{https://doi.org/10.1214/19-ECP208}

\bibitem[{Konarovskyi et~al(2020)Konarovskyi, Lehmann, and von
  Renesse}]{KonarovskyivRenesseTobias}
Konarovskyi V, Lehmann T, von Renesse M (2020) On {D}ean-{K}awasaki dynamics
  with smooth drift potential. J Stat Phys 178(3):666--681,
  \urlprefix\url{https://doi.org/10.1007/s10955-019-02449-3}

\bibitem[{Kotelenez(1982)}]{kotelenez1982submartIneqSPDE}
Kotelenez P (1982) A submartingale type inequality with applicatinos to
  stochastic evolution equations. Stochastics 8(2):139--151,
  \urlprefix\url{https://doi.org/10.1080/17442508208833233}

\bibitem[{Kurtz and Xiong(1999)}]{KurtzXiong}
Kurtz TG, Xiong J (1999) Particle representations for a class of nonlinear
  {SPDE}s. Stochastic Process Appl 83(1):103--126,
  \urlprefix\url{https://doi.org/10.1016/S0304-4149(99)00024-1}

\bibitem[{Kurtz and Xiong(2004)}]{KurtzXiong2}
Kurtz TG, Xiong J (2004) A stochastic evolution equation arising from the
  fluctuations of a class of interacting particle systems. Commun Math Sci
  2(3):325--358, \urlprefix\url{http://projecteuclid.org/euclid.cms/1109868725}

\bibitem[{Lions(2006)}]{LionsVideo}
Lions PL (2006) Cours du coll{\`e}ge de france.
  https://wwwcollege-de-francefr/site/pierre-louis-lions/

\bibitem[{Lions and Sznitman(1984)}]{Lions:Sznitman}
Lions PL, Sznitman AS (1984) Stochastic differential equations with reflecting
  boundary conditions. Comm Pure Appl Math 37(4):511--537,
  \urlprefix\url{https://doi.org/10.1002/cpa.3160370408}

\bibitem[{Lipshutz and Ramanan(2018)}]{LipschutzRamanan1}
Lipshutz D, Ramanan K (2018) On directional derivatives of {S}korokhod maps in
  convex polyhedral domains. Ann Appl Probab 28(2):688--750,
  \urlprefix\url{https://doi.org/10.1214/17-AAP1299}

\bibitem[{Lipshutz and Ramanan(2019)}]{LipschutzRamanan2}
Lipshutz D, Ramanan K (2019) Pathwise differentiability of reflected diffusions
  in convex polyhedral domains. Ann Inst Henri Poincar\'{e} Probab Stat
  55(3):1439--1476, \urlprefix\url{https://doi.org/10.1214/18-aihp924}

\bibitem[{Marx(2018)}]{Marx1}
Marx V (2018) A new approach for the construction of a {W}asserstein diffusion.
  Electron J Probab 23:Paper No. 124, 54,
  \urlprefix\url{https://doi.org/10.1214/18-EJP254}

\bibitem[{Marx(2020)}]{Marx2}
Marx V (2020) A bismut-elworthy inequality for a wasserstein diffusion on the
  circle. arXiv 2005.04972, \doi{10.48550/ARXIV.2005.04972}

\bibitem[{McKean~Jr.(1966)}]{McKean66}
McKean~Jr HP (1966) A class of {M}arkov processes associated with nonlinear
  parabolic equations. Proc Nat Acad Sci USA 56:1907--1911

\bibitem[{Norris(1986)}]{norris1986simplifiedMallCalc}
Norris JR (1986) Simplified malliavin calculus. S\'eminaire de probabilit\'es
  de Strasbourg 20:101--130,
  \urlprefix\url{http://www.numdam.org/item/SPS_1986__20__101_0/}

\bibitem[{Nowak et~al(2019)Nowak, Sj\"{o}gren, and Szarek}]{Nowak}
Nowak A, Sj\"{o}gren P, Szarek TZ (2019) Sharp estimates of the spherical heat
  kernel. J Math Pures Appl (9) 129:23--33,
  \urlprefix\url{https://doi.org/10.1016/j.matpur.2018.10.002}

\bibitem[{Nualart(2006)}]{Nualartbook}
Nualart D (2006) The {M}alliavin calculus and related topics, 2nd edn.
  Probability and its Applications (New York), Springer-Verlag, Berlin

\bibitem[{Nualart and Pardoux(1992)}]{NualartPardoux}
Nualart D, Pardoux E (1992) White noise driven quasilinear {SPDE}s with
  reflection. Probab Theory Related Fields 93(1):77--89,
  \urlprefix\url{https://doi.org/10.1007/BF01195389}

\bibitem[{Otto(1999)}]{Otto1}
Otto F (1999) Evolution of microstructure in unstable porous media flow: a
  relaxational approach. Comm Pure Appl Math 52(7):873--915,
  \urlprefix\url{https://doi.org/10.1002/(SICI)1097-0312(199907)52:7<873::AID-CPA5>3.3.CO;2-K}

\bibitem[{Otto(2001)}]{Otto2}
Otto F (2001) The geometry of dissipative evolution equations: the porous
  medium equation. Comm Partial Differential Equations 26(1-2):101--174,
  \urlprefix\url{https://doi.org/10.1081/PDE-100002243}

\bibitem[{Peszat and
  Zabczyk(1995)}]{peszatZabczyk1995strongFellerIrredHilbertDiff}
Peszat S, Zabczyk J (1995) {Strong Feller Property and Irreducibility for
  Diffusions on Hilbert Spaces}. The Annals of Probability 23(1):157 -- 172,
  \urlprefix\url{https://doi.org/10.1214/aop/1176988381}

\bibitem[{Ren and Wang(2022)}]{RenWang}
Ren P, Wang FY (2022) Ornstein-uhlenbeck type processes on wasserstein space.
  arXiv 2206.05479, \doi{10.48550/ARXIV.2206.05479}

\bibitem[{von Renesse and Sturm(2009)}]{vRenesseSturm}
von Renesse MK, Sturm KT (2009) Entropic measure and {W}asserstein diffusion.
  Ann Probab 37(3):1114--1191,
  \urlprefix\url{https://doi.org/10.1214/08-AOP430}

\bibitem[{Revuz and Yor(1999)}]{revuzYor1999ctsMartAndBM}
Revuz D, Yor M (1999) Continuous martingales and Brownian motion, 3rd edn. No.
  293 in Grundlehren der mathematischen Wissenschaften, Springer

\bibitem[{Röckner et~al(2012)Röckner, Zhu, and
  Zhu}]{rocknerZhuZhu2012reflectInfDimConv}
Röckner M, Zhu RC, Zhu XC (2012) {The stochastic reflection problem on an
  infinite dimensional convex set and BV functions in a Gelfand triple}. The
  Annals of Probability 40(4):1759 -- 1794,
  \urlprefix\url{https://doi.org/10.1214/11-AOP661}

\bibitem[{Salavati and
  Zangeneh(2016)}]{salavatiZangeneh2016pthPowerMaxIneqStochConv}
Salavati E, Zangeneh BZ (2016) A maximal inequality for pth power of stochastic
  convolution integrals. Journal of Inequalities and Applications

\bibitem[{Stannat(2002)}]{Stannat}
Stannat W (2002) Long-time behaviour and regularity properties of transition
  semigroups of {F}leming-{V}iot processes. Probab Theory Related Fields
  122(3):431--469, \urlprefix\url{https://doi.org/10.1007/s004400100166}

\bibitem[{Sturm(2014)}]{Sturm}
Sturm KT (2014) A monotone approximation to the {W}asserstein diffusion. In:
  Singular phenomena and scaling in mathematical models, Springer, Cham, pp
  25--48, \urlprefix\url{https://doi.org/10.1007/978-3-319-00786-1_2}

\bibitem[{Thalmaier(1997)}]{Thalmaier}
Thalmaier A (1997) On the differentiation of heat semigroups and {P}oisson
  integrals. Stochastics Stochastics Rep 61(3-4):297--321,
  \urlprefix\url{https://doi.org/10.1080/17442509708834123}

\bibitem[{Vaillancourt(1988)}]{vaillancourt1988}
Vaillancourt J (1988) On the existence of random mckean–vlasov limits for
  triangular arrays of exchangeable diffusions. Stochastic Analysis and
  Applications 6(4):431--446

\bibitem[{Villani(2009)}]{villani}
Villani C (2009) Optimal Transport, Old and New. Springer Verlag

\bibitem[{Zambotti(2002)}]{Zambotti1}
Zambotti L (2002) Integration by parts formulae on convex sets of paths and
  applications to {SPDE}s with reflection. Probab Theory Related Fields
  123(4):579--600, \urlprefix\url{https://doi.org/10.1007/s004400200203}

\bibitem[{Zambotti(2004)}]{Zambotti2}
Zambotti L (2004) Occupation densities for {SPDE}s with reflection. Ann Probab
  32(1A):191--215, \urlprefix\url{https://doi.org/10.1214/aop/1078415833}

\bibitem[{Zangeneh(1990)}]{zangeneh1990measurabilityofSolutionSemiLinSPDE}
Zangeneh BZ (1990) Measurability of the Solution of a Semilinear Evolution
  Equation, Birkh{\"a}user Boston, pp 335--351.
  \urlprefix\url{https://doi.org/10.1007/978-1-4684-0562-0_18}

\bibitem[{Zangeneh(1995)}]{Zangeneh_Stochastics}
Zangeneh BZ (1995) Semilinear stochastic evolution equations with monotone
  nonlinearities. Stochastics Stochastics Rep 53(1-2):129--174,
  \urlprefix\url{https://doi.org/10.1080/17442509508833986}

\end{thebibliography}

\end{document}